\documentclass[11pt,oneside,reqno]{amsart}
\usepackage{amssymb}
\usepackage{amsmath}
\usepackage{amscd}
\usepackage{amsfonts}
\usepackage{mathrsfs}
\usepackage{stmaryrd}
\usepackage{tikz}
\usetikzlibrary{arrows}
\usepackage[T1]{fontenc}
\allowdisplaybreaks
\usepackage[all,cmtip]{xy}
\usepackage{enumitem}

\makeatletter
\DeclareFontFamily{U}{tipa}{}
\DeclareFontShape{U}{tipa}{m}{n}{<->tipa10}{}

\usepackage{scalerel}
\usepackage{stackengine,wasysym}

\newcommand\reallywidetilde[1]{\ThisStyle{%
  \setbox0=\hbox{$\SavedStyle#1$}%
  \stackengine{-.1\LMpt}{$\SavedStyle#1$}{%
    \stretchto{\scaleto{\SavedStyle\mkern.2mu\AC}{.5150\wd0}}{.6\ht0}%
  }{O}{c}{F}{T}{S}%
}}

\usepackage{scalerel}
\newcommand\reallywidehat[1]{\arraycolsep=0pt\relax%
\begin{array}{c}
\stretchto{
  \scaleto{
    \scalerel*[\widthof{\ensuremath{#1}}]{\kern-.5pt\bigwedge\kern-.5pt}
    {\rule[-\textheight/2]{1ex}{\textheight}} 
  }{\textheight} %
}{0.5ex}\\           
#1\\                 
\rule{-1ex}{0ex}
\end{array}
}

\newcommand{\arc@char}{{\usefont{U}{tipa}{m}{n}\symbol{62}}}%

\newcommand{\arc}[1]{\mathpalette\arc@arc{#1}}

\newcommand{\arc@arc}[2]{%
  \sbox0{$\m@th#1#2$}%
  \vbox{
    \hbox{\resizebox{\wd0}{\height}{\arc@char}}
    \nointerlineskip
    \box0
  }%
}
\makeatother

\usepackage{bbm}

\theoremstyle{definition}
\newtheorem{theorem}{Theorem}[section]

\newtheorem{thm}[theorem]{Theorem}
\newtheorem{prop}[theorem]{Proposition}

\newtheorem{defn}[theorem]{Definition}
\newtheorem{lemma}[theorem]{Lemma}
\newtheorem{cor}[theorem]{Corollary}
\newtheorem{prop-def}[theorem]{Proposition-Definition}
\newtheorem{rema}[theorem]{Remark}

\newtheorem{exam}[theorem]{Example}
\newtheorem{nota}[theorem]{Notation}

\newcommand{\N}{{\mathbb N}}
\newcommand{\C}{{\mathbb C}}
\newcommand{\Z}{{\mathbb Z}}

\renewcommand{\H}{{\mathcal H}}

\newcommand{\Hom}{\textrm{Hom}}

\newcommand{\one}{\mathbf{1}}
\renewcommand{\d}{\mathbf{d}}
\newcommand{\wt}{\text{wt}}

\newcommand{\Res}{\text{Res}}

\newcommand{\g}{{\mathfrak g}}
\newcommand{\h}{{\mathfrak h}}

\newcommand{\Ext}{\text{Ext}}
\newcommand{\CC}{\mathcal{C}}
\newcommand{\Ind}{\text{Ind}}

\begin{document}

\setlength{\oddsidemargin}{0cm} \setlength{\evensidemargin}{0cm}
\baselineskip=18pt

\title[On Ext of left modules for a MOSVA, I]{On the extensions of the left modules for a meromorphic open-string vertex algebra, I}
\author{Fei Qi}

\begin{abstract}
We study the extensions of two left modules $W_1, W_2$ for a meromorphic open-string vertex algebra $V$. We show that the extensions satisfying some technical but natural convergence conditions are in bijective correspondence to the first cohomology classes associated to the $V$-bimodule $\H_N(W_1, W_2)$ constructed in \cite{HQ-Red}. When $V$ is grading-restricted and contains a nice vertex subalgebra $V_0$, those convergence conditions hold automatically. In addition, we show that the dimension of $\Ext^1(W_1, W_2)$ is bounded above by the fusion rule $N\binom{W_2}{VW_1}$ in the category of $V_0$-modules. In particular, if the fusion rule is finite, then $\Ext^1(W_1, W_2)$ is finite-dimensional. We also give an example of an abelian category consisting of certain modules of the Virasoro VOA that does not contain any nice subalgebras, while the convergence conditions hold for every object. 
\end{abstract}

\maketitle

\section{Introduction}

In the representation theory of various algebras (including, but not limited to, commutative associative algebras, associative algebras, Lie algebras, etc.), one of the main tools is the cohomological method. The powerful method of homological algebra often provides a unified treatment of many results in representation theory, giving not only solutions to open problems, but also conceptual understandings to the results.

Vertex operator algebras (VOAs hereafter) arose naturally in both mathematics and physics (see \cite{BPZ}, \cite{B}, and \cite{FLM}) and are analogous to both Lie algebras and commutative associative algebras. In \cite{Hcoh}, Yi-Zhi Huang introduced two cohomology theories for grading-restricted vertex algebras and modules that are analogous to Hochschild and Harrison cohomology for commutative associative algebras. Huang's construction uses linear maps from the tensor powers of the vertex algebra to suitable spaces of ``rational functions valued in the algebraic completion of a $V$-module'' satisfying some natural conditions, including a technical convergence condition. Geometrically, this cohomology theory is consistent with the geometric and operadic formulation of VOA. Algebraically, the first and second cohomologies are given by grading-preserving derivations and first-order deformations, similarly to those for commutative associative algebras. 

In \cite{Q-Coh}, the author generalized Huang's work on the Hochschild-like cohomology in \cite{Hcoh}, and introduced a Hochschild-like cohomology theory for meromorphic open-string vertex algebras (a non-commutative generalization of VOAs) and their bimodules, analogous to the Hochschild cohomology for (not-necessarily commutative) associative algebras and their bimodules. As an application of the cohomology theory, in an early version of \cite{HQ-Red}, Huang and the author proved that if $V$ is a meromorphic open-string vertex algebra such that the first cohomology $\widehat{H}^1(V, W) = 0$ for every $V$-bimodule $W$, then every left $V$-module satisfying a technical but natural convergence condition is completely reducible. See also the author's Ph.D. thesis \cite{Q-Thesis}. 

We remark that the proof in the early version of \cite{HQ-Red} and in \cite{Q-Thesis} essentially uses only a particular type of $V$-bimodule $W$, namely the bimodule $\H_N(W_1, W_2)$ whose construction takes over the major part of the paper. Thus it is not necessary to require $\widehat{H}^1(V, W)= 0$ for every $V$-bimodule. Indeed, after the author's graduation, we found that $\widehat{H}^1(V, W)$ is nonzero for a large class of vertex operator algebras whose module category is semisimple. In these cases, $\widehat{H}^1(V, W)$ is given by the space $Z^1(V, W)$ of zero-mode derivations. The final version of \cite{HQ-Red} proves the same conclusion with the revised assumption: for every $V$-bimodule $W$, the first cohomology $\widehat{H}^1(V, W)=Z^1(V, W)$. It is conjectured that if every $V$-module is semisimple, then $\widehat{H}^1(V, W) = Z^1(V, W)$. The conjecture is supported by \cite{Q-1st-coh}, which computes $\widehat{H}^1(V, W)$ for the most common examples of VOA $V$ and $V$-modules $W$. 

The current paper pushes forward our understanding of the cohomology theory of a MOSVA $V$. Let $W_1, W_2$ be two left $V$-modules. We study the space $\Ext_N^1(W_1, W_2)$ consisting of (equivalence classes of) extensions of $W_1$ by $W_2$ satisfying the technical and natural convergence conditions imposed in \cite{HQ-Red}, namely, the composability condition and the $N$-weight-degree condition (where the subscript $N$ comes from). We establish a bijective correspondence between $\Ext^1_N(W_1, W_2)$ and the first cohomology $\widehat{H}^1(V, \H_N(W_1, W_2))$. With this bijective correspondence, it is clear that if every $V$-module is semisimple, $\widehat{H}^1(V, \H_N(W_1, W_2)) = 0$. This corollary provides a converse to the conclusion in \cite{HQ-Red} (though weaker than the conjecture). It also explains why we totally ignored the zero-mode derivations in the early version of \cite{HQ-Red} and in \cite{Q-Thesis}: they simply do not exist in such context. This result means a lot to the author, who views it as a final defense to his thesis \cite{Q-Thesis}. 

In \cite{HQ-Red}, it is shown that if $V$ is a grading-restricted vertex algebra containing a nice vertex subalgebra $V_0$, then there exists some number $N$ such that the composability and the $N$-weight-degree conditions automatically holds. In the current paper, we further show the following: if $V$ is a grading-restricted MOSVA containing a nice vertex subalgebra $V_0$, then the convergence conditions hold automatically for every extension of $W_2$ and $W_1$. In other words, the $\Ext^1(W_1, W_2)$ in the usual sense coincides with $\Ext_N^1(W_1, W_2)$ for some $N$. Moreover, the cohomology class in $\hat{H}^1(V, \H_N(W_1,W_2))$ corresponding to an extension in $\Ext^1(W_1, W_2)$ can be represented as a $V_0$-intertwining operator of type $\binom{W_2}{VW_1}$. This representation indeed gives an embedding of $\Ext^1(W_1, W_2)$ into a quotient of the space of $V_0$-intertwining operators of type $\binom{W_2}{VW_1}$. In particular, if the fusion rule $N\binom{W_2}{VW_1}$ in the category of $V_0$-module is finite, then $\Ext^1(W_1, W_2)$ is finite-dimensional.  Therefore, we can expect to compute $\Ext^1(W_1, W_2)$ explicitly using the information of $V_0$-modules and $V_0$-intertwining operators. 

The proposed program for computing $\Ext^1(W_1, W_2)$ does not have much requirements on $V_0$ and $V$. $V_0$ is not required to conformally embed in $V_0$. The category of $V$-modules are not required to admit tensor bifunctors. $V$ does not even need to be commutative. One class of important example is the case where $V$ is an affine VOA associated with a Lie algebra, $V_0$ is the Heisenberg subalgebra associated with the Cartan subalgebra, and $W_1, W_2$ are weight modules where the Heisenberg acts semisimply. The program requires no restrictions on the level. We expect the results to be useful for the study the representation theory of affine VOA at nonadmissible level. 

The current paper also shows that if the convergence conditions hold for a module $W$, then they hold for every submodule and (sub)quotient of the module $W$. This leads us to consider the category of modules satisfying the convergence conditions. However, generally speaking, such a category is not necessarily abelian as it is not closed under direct sums. Nevertheless, we construct a concrete abelian category consisting of subquotients of finite direct sums of Verma modules for the Virasoro VOA of all types except for the type $III_\pm$ (as in Feigin-Fuchs classification in \cite{Ash}, or the Case $1^\pm$ as in \cite{IK}, or the thick block as in \cite{BNW}). We show that the convergence conditions hold automatically for every object in the category. So in this case, the convergence conditions hold automatically even when the VOA $V$ does not contain any nice subalgebras. This hints that the convergence conditions we need might not be a serious obstruction for us to apply results in this paper, though more studies are necessary to understand them further. 

The paper is organized as follows: Section 2 reviews the necessary prerequisites, including the definitions of MOSVAs and modules, cohomology theory, and the construction of the bimodule $\H_N(W_1, W_2)$ for two left $V$-modules $W_1, W_2$. Section 3 defines $\Ext^1_N(W_1, W_2)$ and establishes the bijective correspondence to $\widehat{H}^1(V, \H_N(W_1, W_2))$. Section 4 first investigates category $\CC_N$ of modules satisfying the convergence conditions, then discusses the case when $V$ contains a nice vertex subalgebra $V_0$ and realizes cohomology classes as $V_0$-intertwining operators, and finally concludes with the estimate of $\dim \Ext^1(W_1, W_2)$ by the fusion rule $N\binom{W_2}{VW_1}$.  Section 5 discusses the example of Virasoro VOA and gives the example of the abelian category where the convergence conditions hold for every object. 

\noindent \textbf{Acknowledgement. } I would like to thank Yi-Zhi Huang for his long-term constant support, especially for his encouragement in writing down the convergence results for the Virasoro VOA. I would also like to thank Kenji Iohara for answering my questions about the representation theory of Virasoro algebra and providing the idea for Theorem \ref{submodule-classify}. Thanks also to Florencia Hunziker, Shashank Kanade, Andrew Linshaw, Jiayin Pan, and Eric Schippers for discussing various aspects of the current work. Finally, my deepest gratitude to the anonymous Reviewer 1, whose meticulous review and constructive suggestions greatly improved the paper.

\section{Preliminaries}

\subsection{Meromorphic open-string vertex algebra and its modules} We briefly review the definitions of the MOSVA, its left module, right module, and bimodules. Please find further details in \cite{H-MOSVA}, \cite{Q-Mod} and \cite{Q-2d-space-form}. 

\begin{defn}\label{DefMOSVA}
{\rm A {\it meromorphic open-string vertex algebra} (hereafter MOSVA) is a $\Z$-graded vector space 
$V=\coprod_{n\in\Z} V_{(n)}$ (graded by {\it weights}) equipped with a {\it vertex operator map}
\begin{eqnarray*}
   Y:  V\otimes V &\to & V[[x,x^{-1}]]\\
	u\otimes v &\mapsto& Y(u,x)v,
  \end{eqnarray*}
and a {\it vacuum} $\one\in V$, satisfying the following axioms:
\begin{enumerate}
\item Axioms for the grading:
\begin{enumerate}
\item {\it Lower bound condition}: When $n$ is sufficiently negative,
$V_{(n)}=0$.
\item {\it $\d$-commutator formula}: Let $\d_{V}: V\to V$
be defined by $\d_{V}v=nv$ for $v\in V_{(n)}$. Then for every $v\in V$
$$[\d_{V}, Y(v, x)]=x\frac{d}{dx}Y(v, x)+Y(\d_{V}v, x).$$
\end{enumerate}

\item Axioms for the vacuum: 
\begin{enumerate}
\item {\it Identity property}: Let $1_{V}$ be the identity operator on $V$. Then
$Y(\mathbf{1}, x)=1_{V}$. 
\item {\it Creation property}: For $u\in V$, $Y(u, x)\mathbf{1}\in V[[x]]$ and 
$\lim_{x\to 0}Y(u, x)\mathbf{1}=u$.
\end{enumerate}

\item {\it $D$-derivative property and $D$-commutator formula}:
Let $D_V: V\to V$ be the operator
given by
$$D_{V}v=\lim_{x\to 0}\frac{d}{dx}Y(v, x)\one$$
for $v\in V$. Then for $v\in V$,
$$\frac{d}{dx}Y(v, x)=Y(D_{V}v, x)=[D_{V}, Y(v, x)].$$

\item {\it Weak associativity with pole-order condition}: For every $u_1, v\in V$, there exists $p\in \mathbb{N}$ such that for every $u_2\in V$, 
$$(x_0+x_2)^p Y(u_1, x_0+x_2)Y(u_2, x_2)v = (x_0+x_2)^p Y(Y(u_1, x_0)u_2, x_2)v.$$

\end{enumerate}  }
\end{defn}

\begin{prop-def}\label{oppMOSVA}
Let $(V, Y, \one)$ be a MOSVA. Define the \textit{skew-symmetry vertex operator} as follows: 
$$\begin{aligned}
&Y^s: & V\otimes V &\to V[[x, x^{-1}]]\\
&& u\otimes v &\mapsto e^{xD_V} Y(v, -x)u. 
\end{aligned}$$
Then $(V, Y^s, \one)$ is also a MOSVA, called the \textit{opposite MOSVA} of $(V, Y, \one)$, denoted by $V^{op}$. Clearly $(V^{op})^{op} = V$.
\end{prop-def}

\begin{defn}\label{DefMOSVA-L}
Let $V$ be a MOSVA.
A \textit{left $V$-module} is a $\C$-graded vector space 
$W=\coprod_{m\in \C}W_{[m]}$ (graded by \textit{weights}), equipped with 
a \textit{vertex operator map}
\begin{eqnarray*}
Y_W^L: V\otimes W & \to & W[[x, x^{-1}]]\\
u\otimes w & \mapsto & Y_W^L(u, x)w,
\end{eqnarray*}
an operator $\d_{W}$ of weight $0$, 
satisfying the following axioms:
\begin{enumerate}

\item Axioms for the grading: 
\begin{enumerate}
\item \textit{Lower bound condition}:  When $\text{Re}{(m)}$ is sufficiently negative,
$W_{[m]}=0$. 
\item  \textit{$\mathbf{d}$-grading condition}: for every $w\in W_{[m]}$, $\d_W w = m w$.
\item  \textit{$\mathbf{d}$-commutator formula}: For $u\in V$, 
$$[\mathbf{d}_{W}, Y_W^L(u,x)]= Y_W^L(\mathbf{d}_{V}u,x)+x\frac{d}{dx}Y_W^L(u,x).$$
\end{enumerate}

\item The \textit{identity property}:
$Y_W^L(\one,x)=1_{W}$.

\item The \textit{$D$-derivative property}: 
For $u\in V$,
\begin{eqnarray*}
\frac{d}{dx}Y_W^L(u, x)=Y_W^L(D_{V}u, x). 
\end{eqnarray*}

\item {\it Weak associativity with pole-order condition}: For every $v_1\in V, w\in W$, there exists $p\in \mathbb{N}$ such that for every $v_2\in V$, 
$$(x_0+x_2)^p Y_W^L(v_1, x_0+x_2)Y_W^L(v_2, x_2)w = (x_0+x_2)^p Y_W^L(Y(v_1, x_0)v_2, x_2)w. $$
\end{enumerate} 
\end{defn}

\begin{rema}
    In \cite{H-MOSVA} and \cite{Q-Mod}, the definition of modules also contains a requirement on an operator $D_W: W \to W$ of weight 1 satisfying the $D$-commutator formula
    $$[D_W, Y_W^L(v, x)] = \frac{d}{dx}Y_W^L(v, x)$$
    for every $v\in V$. We should provide a few comments here. 
    \begin{enumerate}
        \item It should be noted that this requirement does not follow from the general philosophy ``all the axioms in the definition of algebra that make sense hold'' (see also the discussion of non-conformal vertex algebras in \cite{LL}, Remark 4.1.4). While for many non-conformal vertex algebras, such a $D_W$-operator naturally exists in their modules, the author recently found an example of MOSVA and left module where such a $D_W$-operator cannot exist (see \cite{Q-Fermion-2}, Remark 5.8).
        \item Nevertheless, if $V$ is a VOA with a conformal element $\omega$, then a left $V$-module in the sense of Definition \ref{DefMOSVA-L} agrees with the notion of generalized $V$-module in the sense of Definition 2.2 of \cite{Hcoh} (see \cite{LL}, Theorem 4.4.5). In this case, such a $D_W$-operator exists and can be taken as $\Res_{x=0}Y_W(\omega, x)$. 
    \end{enumerate} 
    In this paper, we will drop this condition for the most general left $V$-module and denote it by $(W, Y_W^L, \d_W)$. We will use the terminology ``left $V$-module $W$ with a $D_W$-operator'' to describe the situation when we need such an operator and will denote it by $(W, Y_W^L, \d_W, D_W)$. 
\end{rema}

\begin{rema}
    Theorem 3.4 of \cite{Q-Mod} proves the convergence of the product of more than three vertex operators with the assumption of a $D_W$-operator. But we may remove that assumption by replacing the part using $e^{yD_W}Y_W^L(u, x)e^{-yD_W} = Y_W^L(u, x+y)$ with $Y_W^L(e^{yD_V}u, x) = Y_W^L(u, x+y)$. The conclusion still holds. 
\end{rema}

\begin{defn}
Let $W_1$, $W_2$ be two left $V$-modules. A linear map $f: W_1\to W_2$ is a homomorphism if $f \d_{W_1} = \d_{W_2} f$, and for every $v\in V$, $fY_{W_1}(v, x) = Y_{W_2}(v, x)f$. 
\end{defn}

\begin{defn}\label{right-module}
Let $V$ be a MOSVA.
A \textit{right $V$-module} is a $\C$-graded vector space 
$W=\coprod_{m\in \C}W_{[m]}$ (graded by \textit{weights}), equipped with 
a \textit{vertex operator} 
\begin{eqnarray*}
Y_W^{s(R)}: V\otimes W & \to & W[[x, x^{-1}]]\\
u\otimes w & \mapsto & Y_W^{s(R)}(u, x)w,
\end{eqnarray*}
an operator $\d_{W}$ of weight $0$, such that $(W, Y_W^{s(R)}, \d_W)$ forms a left $V^{op}$-module. 
\end{defn}

\begin{rema}
If in addition, $(W, Y_W^{s(R)}, \d_W, D_W)$ is a left $V^{op}$-module with a $D_W$-operator, then one may define a vertex operator $Y_{W}^R: W\otimes V \to W[[x, x^{-1}]]$ via the skew-symmetry relation
$$Y_W^R(w, x)v = e^{xD_W} Y_W^{s(R)}(v, -x)w. $$
The operator $Y_W^R$ is analogous to an intertwining operator of type $\binom{W}{WV}$. One may use $Y_W^R$ to give an equivalent definition of the right $V$-module with a $D_W$-operator (see \cite{Q-Mod}). 
\end{rema}

\begin{defn}
Let $(V, Y, \one)$ be a MOSVA. A \textit{$V$-bimodule} is a vector space equipped with a left $V$-module structure and right $V$-module structure such that these two structures are compatible. 
More precisely, a \textit{$V$-bimodule} is a $\C$-graded vector space 
$$W=\coprod_{n\in \C}W_{[n]}$$
equipped with two \textit{vertex operators}
\begin{align*}
    \begin{aligned}
Y_{W}^{L}: V\otimes W&\to W[[x, x^{-1}]]\\
u\otimes w&\mapsto Y_{W}^{L}(u, x)v,
\end{aligned} &  &  \begin{aligned}
Y_{W}^{s(R)}: V\otimes W&\to W[[x, x^{-1}]]\\
u\otimes w&\mapsto Y_{W}^{s(R)}(u, x)w,
\end{aligned}
\end{align*}
and linear operators $\d_W$ on $W$ satisfying the following conditions.
\begin{enumerate}

 \item $(W, Y_W^L, \d_W)$ is a left $V$-module.

 \item $(W, Y_W^{s(R)}, \d_W)$ is a left $V^{op}$-module.
 
 \item {\it Compatibility with pole-order condition}: For every $v_1, v_2\in V$, there exists $p\in \N$ such that for every $w\in W$, 
$$(x_1-x_2)^p Y_W^L(v_1, x_1)Y_W^{s(R)}(v_2, x_2)w = (x_1-x_2)^p Y_W^{s(R)}(v_2, x_2)Y_W^L(v_1, x_1)w. $$
\end{enumerate}
\end{defn}

\begin{rema}

The conditions we imposed here are sufficient to guarantee that for every $w'\in W'=\coprod_{n\in \C}W_{(n)}^*$ (the restricted dual of $W$), $v_1, ..., v_{l}, v_{l+1}, ..., v_{l+r}\in V$, $w\in W$, the series
$$\langle w', Y_W^{L}(v_1, z_1)\cdots Y_W^{L}(v_l, z_l)Y_W^{s(R)}(v_{l+1}, z_{l+1})\cdots Y_W^{s(R)}(v_{l+r}, z_{l+r})w\rangle
$$
converges absolutely in the region $$|z_1|>\cdots >|z_l|>|z_{l+1}|>\cdots >|z_{l+r}|>0$$ 
to a rational function in $z_1, ..., z_{l}, z_{l+1}, ..., z_{l+r}$ with the only possible poles at $z_i = 0$ $(i = 1, ..., l+r)$ and $z_i = z_j$ $(1 \leq i < j \leq l+r)$. Please see \cite{Q-Mod} for further details. 
\end{rema}

\begin{rema}
    For the cohomology theory, we will require in addition that there exists a $D_W$-operator, such that $(W, Y_W^L, \d_W, D_W)$ forms a left $V$-module with a $D_W$-operator, and $(W, Y_W^{s(R)}, \d_W, D_W)$ forms a left $V^{op}$-module with \textit{the same} $D_W$-operator. Starting from this point, all $V$-bimodules are equipped with such a $D_W$-operator. 
\end{rema}

\subsection{Cohomology theory of meromorphic open-string vertex algebras} We briefly recall the definition of cohomologies for MOSVA. Please see \cite{Hcoh} and \cite{Q-Coh} for further details.

\begin{defn}\label{arcWRat} 
For $n\in \Z_+$, we consider the configuration space
$$F_n \C = \{(z_1, ..., z_n)\in \C^n: z_i \neq z_j, i\neq j\}.$$
Let $W= \coprod_{m\in \C}$ be a $\C$-graded vector space. A \textit{$\overline{W}$-valued rational function in $z_1, ..., z_n$ with the only possible poles at $z_i = z_j, i\neq j$} is a map
$$\begin{aligned}
f:  F_n \C &\to \overline{W} = \prod_{m\in \C}W_{[m]}\\
(z_1, ..., z_n)&\mapsto f(z_1, ..., z_n)
\end{aligned}$$
such that 
\begin{enumerate}
\item For any $w'\in W'$, 
$$\langle w', f(z_1, ..., z_n)\rangle$$
is a rational function in $z_1, ..., z_n$ with the only possible poles at $z_i = z_j, i\neq j$.
\item There exists integers $p_{ij}, 1\leq i < j \leq n$ and a formal series $g(x_1, ..., x_n)\in W[[x_1, ..., x_n]]$, such that for every $w'\in W'$ and $(z_1, ..., z_n)\in F_n\C$,  
$$\prod_{1\leq i < j \leq n} (z_i-z_j)^{p_{ij}} \langle w', f(z_1, ..., z_n)\rangle = \langle w', g(z_1, ..., z_n)\rangle$$
as a polynomial function. 
\end{enumerate}
The space of all such functions will be denoted by $\widetilde{W}_{z_1,...,z_n}$.

Similarly, by modifying (1), we define \textit{$\overline{W}$-valued rational function with only possible poles at $z_i = 0$ $(i = 1,...,n)$ and at $z_i = z_j$ $(1\leq i < j \leq n)$}. We denote the space of all such functions by $\widehat{W}_{z_1...z_n}$. 
\end{defn}

\begin{nota}
For a series $f(z_1, ..., z_n)\in W[[z_1, z_1^{-1}, ..., z_n, z_n^{-1}]]$ that converges absolutely to a $\overline{W}$-valued rational function, we will use the notation $E(f(z_1, ..., z_n))$ to denote the limit. For example, let $V$ be a MOSVA and $W$ be a bimodule. 
Then for $v_1, v_2\in V, w\in W$ and $w'\in W'$, both
\begin{equation}\label{E-nota-1}
\langle w', Y_W^{L}(v_1, z_1)Y_W^{s(R)}(v_2, z_2)w\rangle
\end{equation}
and 
\begin{equation}\label{E-nota-2}
    \langle w', Y_W^{s(R)}(v_2, z_2)Y_W^{L}(v_1, z_1)w\rangle
\end{equation}
converge absolutely to the same rational function in $z_1, z_2$ with poles at $z_1=0, z_2=0$ and $z_1=z_2$. Since the series (\ref{E-nota-1}) and (\ref{E-nota-2}) converge in disjoint regions, we cannot equate them. But we can say that
$$E(Y_W^L(v_1, z_1)Y_W^{s(R)}(v_2, z_2)w) = E(Y_W^{s(R)}(v_2, z_2) Y_W^L(v_1, z_1)w)$$
as elements in $\widehat{W}_{z_1z_2}$. 
\end{nota}

\begin{nota}
The cohomology theory of MOSVA is built upon the linear maps  $\Phi: V^{\otimes n}\to \widetilde{W}_{z_1...z_n}$. We will use the notation 
$$\Phi(v_1\otimes \cdots v_n; z_1, ..., z_n)$$
to denote the image of $v_1\otimes \cdots \otimes v_n$ in $\widetilde{W}_{z_1...z_n}$. 
\end{nota}

\begin{defn}
A linear map $\Phi: V^{\otimes n} \to \widetilde{W}_{z_1, ..., z_n}$ is said to have the $D$-derivative property if
\begin{enumerate}
\item For $i=1,..., n$, $v_1, ..., v_n\in V, w'\in W'$, 
$$ \langle w', \Phi(v_1\otimes \cdots\otimes v_{i-1} \otimes D_V v_i \otimes v_{i+1}\otimes \cdots  \otimes v_n;z_1, ..., z_n)\rangle =\frac{\partial}{\partial z_i} \langle w', \Phi(v_1\otimes \cdots \otimes v_n;z_1, ..., z_n)\rangle $$
\item For $v_1, ..., v_n\in V, w'\in W'$, 
$$\langle w', D_W (\Phi(v_1\otimes \cdots v_n;z_1, ..., z_n))\rangle=\left(\frac{\partial}{\partial z_1} + \cdots + \frac{\partial}{\partial z_n} \right) \langle w', \Phi(v_1\otimes \cdots \otimes v_n;z_1, ..., z_n)\rangle $$
\end{enumerate}
\end{defn}

\begin{defn}
A linear map $\Phi: V^{\otimes n} \to \widetilde{W}_{z_1, ..., z_n}$ is said to have the $\d$-conjugation property if for $v_1, ..., v_n\in V, w'\in W', (z_1, ..., z_n)\in F_n\C$ and $z\in \C^\times$ so that $(zz_1, ..., zz_n)\in F_n \C$, 
$$\langle w', z^{\d_W}\Phi(v_1\otimes \cdots \otimes v_n;z_1, ..., z_n)\rangle = \langle w', \Phi(z^{\d_V}v_1 \otimes \cdots z^{\d_V}v_n;zz_1, ..., zz_n)\rangle$$
\end{defn}

\begin{defn}\label{Composable}
Let $\Phi: V^{\otimes n} \to \widetilde{W}_{z_1, ..., z_n}$ be a linear map. Let $m\in \Z_+$. $\Phi$ is said to be composable with $m$ vertex operators if for every $\alpha_0, \alpha_1, ..., \alpha_n\in \Z_+$ such that $\alpha_0+\cdots + \alpha_n = m+n$, every $l_0= 0, ..., \alpha_0$, and every $u_{1}^{(0)}, ..., u_{\alpha_0}^{(0)}, ..., u_1^{(n)}, ..., u_{\alpha_n}^{(n)}\in V$, the series of $\overline{W}$-valued rational functions
\begin{align*}
& \begin{aligned} 
Y_W^L(u_1^{(0)}, z_1^{(0)})& \cdots Y_W^L(u_{l_0}^{(0)}, z_{l_0}^{(0)})Y_W^{s(R)}(u_{l_0+1}^{(0)}, z_{l_0+1}^{(0)}) \cdots Y_W^{s(R)}(u_{\alpha_0}^{(0)}, z_{\alpha_0}^{(0)})\\
\cdot \Phi(& Y(u_1^{(1)}, z_1^{(1)}-\zeta_1)\cdots Y(u_{\alpha_1}^{(1)}, z_{\alpha_1}^{(1)}-\zeta_1)\one \\
\otimes & \cdots \\
\otimes & Y(u_1^{(n)}, z_1^{(n)}-\zeta_n)\cdots Y(u_{\alpha_n}^{(n)}, z_{\alpha_n}^{(n)}-\zeta_n)\one;\zeta_1, ..., \zeta_n)
\end{aligned}\\
& \begin{aligned}
= \sum_{\substack{k_1^{(0)}, ..., k_{\alpha_0}^{(0)},\\ ..., k_1^{(n)}, ..., k_{\alpha_n}^{(n)}\in \Z}}(Y_W^L)&_{k_1^{(0)}}(u_1^{(0)})\cdots (Y_W^L)_{k_{l_0}^{(0)}}(u_{l_0}^{(0)})(Y_W^{s(R)})_{k_{l_0+1}^{(0)}}(u_{l_0+1}^{(0)})\cdots (Y_W^{s(R)})_{k_{\alpha_0}^{(0)}}(u_{\alpha_0}^{(0)})\\
\cdot \Phi(&(Y)_{k_1^{(1)}}(u_1^{(1)}) \cdots (Y)_{k_{\alpha_1}^{(1)}}(u_{\alpha_1}^{(1)})\one \\
\otimes & (Y)_{k_1^{(2)}}(u_1^{(2)}) \cdots (Y)_{k_{\alpha_2}^{(2)}}(u_{\alpha_2}^{(2)})\one  \\
\otimes & \cdots \\
\otimes & (Y)_{k_1^{(n)}}(u_1^{(n)}) \cdots (Y)_{k_{\alpha_n}^{(n)}}(u_{\alpha_n}^{(n)})\one;\zeta_1, ..., \zeta_n)\rangle \\
& \prod_{i=1}^{\alpha_0}(z_i^{(0)})^{-k_i^{(0)}-1}\prod_{i=1}^n \prod_{j=1}^{\alpha_i}(z_j^{(i)}-\zeta_i)^{-k_j^{(i)}-1}
\end{aligned}
\end{align*}
converges absolutely when 
\begin{align*}
& |z_1^{(0)}|> \cdots > |z_{\alpha_0}^{(0)}| > |\zeta_i|+|z_1^{(i)}-\zeta_i|, i = 1, ..., n;\\
& |z_1^{(i)}-\zeta_i|>\cdots > |z_{\alpha_i}^{(i)}-\zeta_i|, i = 1, ..., n; \\
& |z_s^{(i)}-\zeta_i - z_t^{(j)}+\zeta_j|<|\zeta_i-\zeta_j|, 1\leq i < j \leq	n, 1\leq s \leq \alpha_i, 1\leq t \leq \alpha_j.
\end{align*}
and the sum can be analytically extended to a rational function in $z_1^{(1)}, ..., z_{\alpha_1}^{(1)}, ..., z_1^{(n)}, ..., z_{\alpha_n}^{(n)}$ that is independent of $\zeta_1, ..., \zeta_n$ and has the only possible poles at $z_s^{(i)}= z_t^{(j)}$, for $1\leq i < j \leq n, s=1 ,..., \alpha_i, t = 1, ..., \alpha_j$. We require in addition that for each $i, j, s, t$, the order of the pole $z_s^{(i)}=z_t^{(j)}$ is bounded above by a constant that depends only on $u_s^{(i)}$ and $u_t^{(j)}$. 
\end{defn}

\begin{defn}
For $n\in \Z_+$, we define 
\begin{align*}
    \hat{C}_0^n(V, W) = \{\Phi: V^{\otimes n} \to \widetilde{W}_{z_1, ..., z_n} &| \Phi \text{ satisfies the $D$-derivative property and $\d$-conjugation property}\}.\\
    \hat{C}_\infty^n(V, W) =\{\Phi\in \hat{C}_0^n(V, W) &| \Phi \text{ is composable with any number of vertex operators}\}
\end{align*}
For $n=0$, we define 
$$\hat{C}^0(V, W)=\{w\in W_{(0)} | D_W w = 0\}, $$
i.e., $\hat{C}^0(V, W)$ consists of vacuum-like vectors in $W$. 
\end{defn}

\begin{prop-def}
Let $m,n\in \Z_+, \Phi \in \hat{C}_\infty^n(V, W)$. Then for every $i=1, ..., n$, every $v_1, ..., v_{n+1}\in V$, the following series \begin{align}
    \Phi(v_1 \otimes \cdots \otimes Y(v_i, z_i-\zeta) &Y(v_{i+1}, z_{i+1}-\zeta)\one \otimes v_{i+2} \otimes \cdots \otimes v_{n+1};z_1, ..., z_{i-1}, \zeta, z_{i+2}, ..., z_{n+1}),\label{Cobdry-mid}\\
    &Y_W^L(v_1, z_1)\Phi(v_{2}\otimes \cdots \otimes v_{n+1};z_{2}, ..., z_{n+1}),\label{Cobdry-first}\\
    &Y_W^{s(R)}(v_{n+1}, z_{n+1})\Phi(v_{1}\otimes \cdots \otimes v_{n};z_{1}, ..., z_{n})\label{Cobdry-last}
\end{align}
converge absolutely (in certain regions) to $\overline{W}$-valued rational functions depending only on $z_1, ..., z_{n+1}$ with the only possible poles at $z_i = z_j$ $(1\leq i < j \leq n+1)$. We will then define the maps
\begin{align*}
    \Phi\circ_i E_V^{(2)}: V^{\otimes (n+1)}\to \widetilde{W}_{z_1, ..., z_{n+1}}\\
    E_W^{(1, 0)}\circ_{2}\Phi:  V^{\otimes (n+1)}\to \widetilde{W}_{z_1, ..., z_{n+1}}\\
    E_W^{(0, 1)}\circ_{2}\Phi:  V^{\otimes (n+1)}\to \widetilde{W}_{z_1, ..., z_{n+1}}
\end{align*}
sending $v_1\otimes \cdots \otimes v_{n+1}$ to the respective $\overline{W}$-valued rational functions given by the limits of the series (\ref{Cobdry-mid}), (\ref{Cobdry-first}), (\ref{Cobdry-last})
\end{prop-def}


\begin{defn}
For $n \in \Z_+$, we define the coboundary operator as follows
$$\hat{\delta}_\infty^n: \hat{C}_\infty^n(V, W) \to \hat{C}_{\infty}^{n+1}(V, W)$$
by
$$ \hat{\delta}_\infty^n\Phi =  E_W^{(1, 0)}\circ_2 \Phi + \sum_{i=1}^{n} (-1)^i\Phi\circ_i E_V^{(2)} + (-1)^{n+1} E_W^{(0,1)}\circ_2 \Phi
$$
For $n=0$, we define $\hat{\delta}^0: \hat{C}^0(V, W)\to \hat{C}_{\infty}^1(V, W)$ by the following: for $w\in \hat{C}^0(V, W)$
$$[(\hat{\delta}^0(w))(v)](z) = E(Y_W^L(v, z)w-Y_W^{s(R)}(v, z)w)$$
\end{defn}

\begin{thm}
For every $n\in \Z_+$, \begin{align*}
    \hat{\delta}^0(\hat{C}^0(V, W))&\subseteq \hat{C}_\infty ^1(V, W), \\
    \hat{\delta}_\infty^n (\hat{C}_\infty^n(V, W)) &\subseteq \hat{C}_{\infty}^{n+1}(V, W). 
\end{align*}
Moreover, for every $n\in \Z_+$, we have $\hat{\delta}_\infty^1 \circ \hat{\delta^0} = 0$,  
and $\hat{\delta}_{\infty}^{n+1}\circ \hat{\delta}_\infty^n = 0$. 
The sequence
$$\hat{C}^0(V, W) \xrightarrow{\hat{\delta}^0} \hat{C}_\infty^1(V, W) \xrightarrow{\hat{\delta}^1_\infty} \hat{C}_\infty^2(V, W) \xrightarrow{\hat{\delta}^2_\infty} \hat{C}_\infty^3(V, W) \rightarrow \cdots \rightarrow \hat{C}_\infty^m(V, W) \rightarrow \cdots $$
forms a cochain complex. 
\end{thm}

\begin{defn} For every $n\in \N$, the \textit{$n$-th cohomology group} is defined as 
$$\hat{H}_\infty^n(V, W)=\text{Ker} \hat{\delta}_\infty^n / \text{Im} \hat{\delta}_\infty^{n-1}.$$
Elements in $\text{Ker} \hat{\delta}_\infty^n$ are called $n$-cocycles. Elements in $\text{Im} \hat{\delta}_\infty^{n-1}$ are called $n$-coboundaries.
\end{defn}

\begin{exam}
We describe the first cohomology explicitly. 

\begin{enumerate}
    \item Let $\Phi: V\to \widetilde{W}_{z}$ be a 1-cocycle. Then the map $V\to W$ given by $v\mapsto \Phi(v;0)$ is a grading-preserving derivation, i.e., 
    $$\Phi(Y(u,x)v;0) = Y_W^L(u, x)\Phi(v;0) + e^{xD_W} Y_W^{s(R)}(v, -x)\Phi(u;0).$$
    Conversely, let $f: V\to W$ be a grading-preserving derivation, then the map $v\mapsto e^{zD_W} f(v)$ gives a 1-cocycle. 
    \item Let $\Phi: V\to \widetilde{W}_{z}$ be a 1-coboundary. Then the map $V\to W$ given by $v\mapsto \Phi(v;0)$ is a grading-preserving inner-derivation, i.e., there exists a vacuum-like vector $w\in W$, such that 
    $$\Phi(v;0) = \lim_{x\to 0}Y_W^L(v, x)w - \lim_{x\to 0} Y_W^{s(R)}(v, x)w.$$
    Conversely, let $f: V\to W$ be a grading-preserving inner derivation, then the map $v\mapsto e^{zD_W} f(v)$ gives a 1-coboundary. 
    \item Thus $\hat{H}_\infty^1(V, W) = \{\text{grading-preserving derivations}\} / \{\text{grading-preserving inner derivations}\}$. For brevity, we shall use $\hat{H}^1(V, W)$ for the first cohomology group hereafter. 
\end{enumerate}

\end{exam}

\subsection{The bimodule $\H_N(W_1, W_2)$} We briefly review the definition of the $V$-bimodule $\H_N(W_1, W_2)$ associated with two left $V$-modules $W_1$ and $W_2$. Please find further details in \cite{Q-Thesis} and \cite{HQ-Red}. 
\begin{prop-def}\label{Def-H_N}
Let $V$ be a MOSVA. Let $W_{1}$ and $W_{2}$ be two left $V$-modules. 
Let $\widehat{(W_{2})}_{\zeta}$ be the space of $\overline{W_{2}}$-valued rational 
functions with the only possible pole at $\zeta=0$. Recall that $\widetilde{(W_{2})}_{\zeta}$ is the space of 
$\overline{W_{2}}$-valued holomorphic functions. Thus $\widehat{(W_{2})}_{\zeta}\supset \widetilde{(W_{2})}_{\zeta}$.

Let $\H_N(W_1, W_2)$ be the subspace  of $\Hom_\C(W_{1}, \widehat{(W_{2})}_{\zeta})$ spanned by
elements, denoted by 
$$\phi: w_1\mapsto \phi(\zeta)w_1 \in \widehat{(W_2)}_\zeta$$ satisfying the following conditions:
\begin{enumerate}
\item The {\it $\mathbf{d}$-conjugation property}: There exists $n\in \Z$ (called the {\it weight of $\phi$}
and denoted by $\wt(\phi)$) such that 
for $a\in \C^{\times}$ and $w_{1}\in W_{1}$,
$$a^{\mathbf{d}_{W_{2}}}(\phi(\zeta)w_1)=a^{n}(\phi(a\zeta)a^{\mathbf{d}_{W_{1}}}w_{1}).$$

\item The {\it composability condition}: For $k, l\in \N$ and $v_{1}, \dots, v_{k+l}\in V$, $w_{1}\in W_{1}$
and $w_{2}'\in W_{2}'$,
the series 
\begin{equation}\label{Def-H_N-series}
    \left\langle w_{2}', Y_{W_{2}}(v_{1}, z_{1})\cdots Y_{W_{2}}(v_{k}, z_{k})
\phi(\zeta)Y_{W_{1}}(v_{k+1}, z_{k+1})\cdots Y_{W_{1}}(v_{k+l}, z_{k+l})w_{1}\right\rangle
\end{equation}is absolutely convergent in the region $|z_{1}|>\cdots >|z_{k}|>|\zeta|>|z_{k+1}|>\cdots |z_{k+l}|>0$
to a rational function 
\begin{equation}\label{r-fun}
f(z_1, ..., z_k, \zeta, z_{k+1}, ..., z_{k+l})
\end{equation}
in $z_{1}, \dots, z_{k+l}$ and $\zeta$ with the only possible 
poles $z_{i}=0$ for $i=1, \dots, k+l$, $\zeta=0$, $z_{i}=z_{j}$ for $i, j=1, \dots, k+l$, $i\ne j$ and 
$z_{i}=\zeta$ for $i=1, \dots, k+l$. Moreover, there 
exist $r_{i}\in \N$ depending only on the pair $(v_{i}, w_{1})$ for $i=1, \dots, k+l$, $m\in \N$
depending only on the pair $(\phi, w_{1})$, $p_{ij} \in \N$ depending only on the pair $(v_{i}, v_{j})$
for $i, j=1, \dots, k+l$, $i\ne j$, $s_{i}\in \N$ depending only on the pair $(v_{i}, \phi)$  for $i=1, \dots, k+l$
and $g(z_{1}, \dots, z_{k+l}, \zeta)\in W_{2}[[z_{1}, \dots, z_{k+l}, \zeta]]$ such that for $w_{2}'\in W_{2}'$, 
\begin{align*}
\zeta^{m}\prod_{i=1}^{k+l}&z_{i}^{r_{i}}\prod_{1\le i<j\le k+l}(z_{i}-z_{j})^{p_{ij}}
\prod_{i=1}^{k+l}(z_{i}-\zeta)^{s_{i}}\cdot f(z_1, ..., z_k, \zeta, z_{k+1}, ..., z_{k+l})
\end{align*}
is a polynomial and is equal to 
$\langle w_{2}', g(z_{1}, \dots, z_{k+l}, \zeta)\rangle$.
\item The {\it $N$-weight-degree condition}: Expand $f(z_1, ..., z_k, \zeta, z_{k+1}, ..., z_{k+l}) $ in the region $|z_{k+l}|>|z_{1}-z_{k+l}|>\cdots >|z_{k}-z_{k+l}|>|\zeta-z_{k+l}|>|z_{k+1}-z_{k+l}|>\cdots
>|z_{k+l-1}-z_{k+l}|>0$ as a Laurent series in $z_{i}-z_{k+l}$ for $i=1, \dots, k+l-1$ and $\zeta-z_{k+l}$
with Laurent polynomials in $z_{k+l}$ as coefficients. Then the total degree of each monomial 
in $z_{i}-z_{k+l}$ for $i=1, \dots, k+l-1$ and $\zeta-z_{k+l}$ (that is, the sum of the powers of 
$z_{i}-z_{k+l}$ for $i=1, \dots, k+l-1$ and $\zeta-z_{k+l}$) in the expansion 
is greater than or equal to $N-\sum_{i=1}^{k+l}\wt v_{i}-\wt \phi$.
\end{enumerate}
Let $\H_N(W_1, W_2)_{[n]}$ be the 
subspace of $\H_N(W_1, W_2)$ consisting of elements of weight $n$. Then 
$$\H_N(W_1, W_2)=\coprod_{n\in \Z}\H_N(W_1, W_2)_{[n]}.$$ 
We define a $V$-bimodule structure using the left and right vertex operator maps:
\begin{align*}
    Y_\H^L: V\otimes \H_N(W_1, W_2) &\to \H_N(W_1, W_2)[[x, x^{-1}]]\\
    v\otimes \phi&\mapsto Y_\H^L(v,x)\phi\\
    Y_\H^{s(R)}: V\otimes \H_N(W_1, W_2) &\to \H_N(W_1, W_2)[[x, x^{-1}]]\\
    v\otimes \phi&\mapsto Y_\H^{s(R)}(v,x)\phi
\end{align*}
defined by 
\begin{align}
    \left(Y_\H^L(v, x)\phi\right)(\zeta)w_1 &= \iota_{\zeta x} E\left(Y_{W_2}(v, x + \zeta) \phi(\zeta)w_1\right), \\
    \left(Y_\H^{s(R)}(v, x)\phi\right)(\zeta)w_1 &= \iota_{\zeta x} E\left(\phi(\zeta)Y_{W_1}(v, x+\zeta)w_1\right). 
\end{align}
Define also the operators $\mathbf{d}_{\H}$ and $D_{\H}$ on $\H_N(W_1, W_2)$ 
$$\mathbf{d}_{\H}\phi=n\phi\text{ for }\phi\in \H_N(W_1, W_2)_{[n]}.$$  
$$(D_{\H}\phi)(\zeta)w_1=\frac{\partial}{\partial \zeta}(\phi(\zeta)w_1)$$
for $\phi\in \H_N(W_1, W_2)$, $w_{1}\in W_{1}$. Then $(\H_N(W_1, W_2), Y_\H^L, Y_\H^{s(R)}, \d_\H, D_\H)$ forms a $V$-bimodule with lowest $\d_\H$-weight being $N$. 
\end{prop-def}

\begin{rema}
The convergence conditions mentioned in the abstract are precisely Condition (3) and (4) in Proposition-Definition \ref{Def-H_N}. See also Definition \ref{composibility} and Remark \ref{composibility-formal}. 
\end{rema}

\begin{rema}
    The following example may help the reader to understand the $N$-weight-degree condition. Let $W_1 = W_2 = W$. Fix $v\in V$. Let $\phi(\zeta) = Y_W(v, \zeta)$. Then $\phi$ is homogeneous with $\wt(\phi) = \wt(v)$. For $w_2'\in W', w_1\in W$, the series (\ref{Def-H_N-series}) is now precisely 
    \begin{align}
        \left\langle w_{2}', Y_W(v_{1}, z_{1})\cdots Y_W(v_{k}, z_{k})
        Y_W(v, \zeta)Y_W(v_{k+1}, z_{k+1})\cdots Y_W(v_{k+l}, z_{k+l})w_{1}\right\rangle\label{N-wt-deg-ex-1}
    \end{align}
    that converges to a rational function satisfying the pole-order condition. From associativity, we see the series
    \begin{align}
        & \quad \begin{aligned}\big\langle w_{2}', & Y_W(Y(v_{1}, z_{1}-z_{k+l})\cdots Y(v_{k}, z_{k}-z_{k+l})\nonumber \\
        & \cdot Y(v, \zeta-z_{k+l})Y(v_{k+1}, z_{k+1}-z_{k+l})\cdots Y(v_{k+l-1}, z_{k+l-1}-z_{k+l})v_{k+l}, z_{k+l})w_{1}\big\rangle\end{aligned}\label{N-wt-deg-ex-2}\\
        & \begin{aligned}= \sum_{m_1, ..., m, ..., m_{k+l-1}\in \Z} \langle w_2', &Y_W(Y_{m_1}(v_1)\cdots Y_{m}(v) \cdots Y_{m_{k+l-1}}(v_{k+l-1})v_{k+l}, z_{k+l})w_1\rangle \\
        & \cdot (z_1-z_{k+l})^{-m_1-1}\cdots (\zeta-z_{k+l})^{-m-1} \cdots (z_{k+l-1}-z_{k+l})^{-m_{k+l-1}-1}\end{aligned}
    \end{align}
    is the expansion of the rational function given by (\ref{N-wt-deg-ex-1}) in the region specified in the $N$-weight-degree condition. If $N$ is the lowest weight of elements in $V$, then necessarily, 
    \begin{align*}
        & \wt v_1- m_1-1  + \cdots + \wt v - m - 1 + \cdots + \wt v_{k+l-1}-m_{k+l-1}-1 + \wt v_{k+l} \geq N\\
        \Rightarrow\ & -m_1-1 -\cdots - m -1 - \cdots - m_{k+l-1}-1\geq N - \sum_{i=1}^{k+l} \wt v_i - \wt \phi
    \end{align*}
    In other words, if we view (\ref{N-wt-deg-ex-2}) as a series in $(\C[z_{k+l}, z_{k+l}^{-1}])[[z_1-z_{k+l}, (z_1-z_{k+l})^{-1}, ..., \zeta-z_{k+l}, (\zeta-z_{k+l})^{-1}, ..., z_{k+l-1}-z_{k+l}, (z_{k+l-1}-z_{k+l})^{-1}]]$, then the total degree of $z_1-z_{k+l}, ..., \zeta - z_{k+l}, ..., z_{k+l-1}-z_{k+l}$ is at least as large as $N - \sum_{i=1}^{k+l} \wt v_i - \wt \phi$. If $0$ is the lowest weight of $V$, then we may take $N=0$. 
\end{rema}

\begin{rema}
    Note that the $D_\H$ operator on $\H_N(W_1, W_2)$ does not require the existence of $D$-operators on the left $V$-modules $W_1$ or $W_2$. 
\end{rema}

\section{Extensions and the first cohomology}\label{Section-3}

\subsection{Extensions of left modules}

\begin{defn}
    Let $W_1 = (W_1, Y_{W_1}^L, \d_{W_1})$ and $W_2 = (W_2, Y_{W_2}^L, \d_{W_2})$ be two left $V$-modules. Let $U = (U, Y_U^L, \d_{U})$ be a left $V$-module that fits in the exact sequence
$$\xymatrix{ 0 \ar[r] & W_2 \ar[r] & U \ar[r]^p & W_1 \ar[r] & 0.  }$$
We call $U$ an extension of $W_1$ by $W_2$. We will say that $W_2$ is a submodule of $U$ without distinguishing $W_2$ and its image in $U$. Two extensions $U^{(1)}$, $U^{(2)}$ are equivalent, if there exists a homomorphism $T: U^{(1)}\to U^{(2)}$, such that the following diagram commutes
\begin{align}
    \begin{aligned}\xymatrix{ 0 \ar[r] & W_2 \ar[r] \ar@{=}[d] & U^{(1)} \ar[r]^{p^{(1)}} \ar[d]^T & W_1 \ar[r] \ar@{=}[d] & 0
\\
0 \ar[r] & W_2 \ar[r] & U^{(2)} \ar[r]^{p^{(2)}} & W_1 \ar[r] & 0
} \end{aligned}\label{CommDiag}
\end{align}
\end{defn}

\begin{rema}
    We emphasize that in Definition 3.1, the modules $W_1$, $W_2$ and $U$ are \textit{not} required to have $D_{W_1}$-, $D_{W_2}$- and $D_U$-operators. Even when $D_{W_1}$- and $D_{W_2}$-operators exist for $W_1$ and $W_2$, we do not require the existence of a $D_U$-operator on $U$. In particular, if $V$ is a non-conformal vertex algebra and $W_1, W_2$ have $D$-operators, $D_U$ may still be nonexistent. On the other hand, if $V$ is a VOA with a conformal element $\omega$, then a $D_U$-operator naturally exists from $\Res_{x=0}Y_U(\omega, x)$. 
\end{rema}

\begin{defn}\label{Ext_N-def}
Let $\Ext_N^1(W_1, W_2)$ be the set of equivalence classes of extensions $U$ satisfying the following additional property: there exists a grading-preserving linear map $\psi: W_1 \to U$ that is the section of the \textit{linear map} $p: U\to W_1$ (i.e. $\psi \circ p =Id_{W_1}$), such that $U = W_2 \amalg \psi(W_1)$ as vector spaces, and for every $v\in V$, the map 
$$\pi_2 Y_U(v, \zeta) \psi \in \H_N(W_1, W_2). $$
where $\pi_2: W_2\amalg \psi(W_1) \to W_2$ is the projection operator. 
\end{defn}

\begin{rema}
We will keep the $\oplus$ notation for direct sums of $V$-modules. The notation $W_1\amalg W_2$ means we have vector space direct sum of $W_1$ and $W_2$ that does not respect the $V$-actions. 
\end{rema}

\begin{prop}\label{pi-one-module}
Let $U, \psi, \pi_2$ be as in Definition \ref{Ext_N-def}. Let $\pi_1 = 1-\pi_2: U \to \psi(W_1)$. Then $(\psi(W_1), \pi_1Y_U, \d_U|_{\psi(W_1)})$ forms a left $V$-module. Moreover, the map $\psi: (W_1, Y_{W_1}, \d_{W_1}) \to (\psi(W_1), \pi_1 Y_U, \d_{U}|_{\psi(W_1)})$ is a left $V$-module isomorphism.
\end{prop}

\begin{proof}
For the first part, we only verify 
weak associativity. Other axioms follow trivially. 


From the weak associativity on $U$, for every $u\in U, w_1\in W_1$, we have $p\in \N$, such that for every $v\in V$, 
\begin{equation}
    (x_0+x_2)^p Y_U(Y(u, x_0)v,x_2)\psi(w_1) = (x_0+x_2)^p Y_U(u, x_0+x_2)Y_U(v,x_2)\psi(w_1)\label{pi-one-module-0}
\end{equation}
Apply $\pi_1$ on both sides of (\ref{pi-one-module-0}). Then 
\begin{align}
    (x_0+x_2)^p \pi_1 Y_U(Y(u, x_0)v,x_2)\psi(w_1) &= (x_0+x_2)^p \pi_1 Y_U(u, x_0+x_2) (\pi_1+\pi_2) Y_U(v, x_2)\psi(w_1)\nonumber\\
    &= (x_0+x_2)^p \pi_1 Y_U(u, x_0+x_2) \pi_1 Y_U(v, x_2)\psi(w_1) \label{pi-one-module-1}\\
    &\quad + (x_0+x_2)^p \pi_1 Y_U(u, x_0+x_2) \pi_2 Y_U(v, x_2)\psi(w_1)\label{pi-one-module-2}
\end{align}
However, since $W_2$ is a submodule, $\pi_1 Y_U(u, x)w_2 = 0$ for every $w_2\in W_2$. Thus (\ref{pi-one-module-2}) is zero. Thus we see that 
$$(x_0+x_2)^p \pi_1 Y_U(Y(u, x_0)v,x_2)\psi(w_1)=(x_0+x_2)^p \pi_1 Y_U(u, x_0+x_2) \pi_1 Y_U(v, x_2)\psi(w_1). $$
So we have shown the weak associativity. 

Now we show that $\psi$ is an isomorphism. Let $p: U\to W_1$ be the homomorphism in the exact sequence. Then 
$$p Y_U(v, x) \psi(w_1) = Y_{W_1}(v, x) p \psi (w_1) = Y_{W_1}(v, x) w_1 $$
Note that the left-hand-side can also be written as 
$$p (\pi_1+\pi_2) Y_U(v, x)\psi(w_1)= p\pi_1 Y_U(v,x)\psi(w_1)$$
because $p$ restricted to $W_2$ is zero. Thus we see that $p \pi_1 Y_U(v, x)\psi(w_1) = Y_{W_1}(v, x)w_1$. 
Thus $p|_{\psi(W_1)}$ is an isomorphism from $\psi(W_1)$ to $W_1$. Obviously $\psi$ is the inverse of $p|_{\psi(W_1)}$. Thus we see that $\psi$ is an isomorphism. 
\end{proof}

\subsection{Derivation from an extension in $\Ext_N^1(W_1, W_2)$}

\begin{prop}\label{Ext-der-prop}
    Let $U$ be an extension of $W_1$ by $W_2$, $\psi: W_1\to U$ be the section of $p: U\to W_1$, $\pi_2: U\to W_2$ be the projection operator as in Definition \ref{Ext_N-def}. Then the map
    \begin{align*}
        F: V & \to \H_N(W_1, W_2)\\
        v& \mapsto  \pi_2 Y_U(v,\zeta) \psi
    \end{align*}    
    is a derivation.
\end{prop}

\begin{proof}
From the weak associativity on $U$, for every $u\in U, w_1\in W_1$, we have $q\in \N$, such that for every $v\in V$, 
$$(x_0+x_2)^q Y_U(Y(u, x_0)v,x_2)\psi(w_1) = (x_0+x_2)^q Y_U(u, x_0+x_2)Y_U(v,x_2)\psi(w_1)$$
Apply $\pi_2$ on both sides. 
\begin{align}
    (x_0+x_2)^q \pi_2 Y_U(Y(u, x_0)v,x_2)\psi(w_1) &= (x_0+x_2)^q \pi_2 Y_U(u, x_0+x_2)Y_U(v,x_2)\psi(w_1)\nonumber\\
    &=(x_0+x_2)^q \pi_2 Y_U(u, x_0+x_2) \pi_2  Y_U(v,x_2)\psi(w_1) \label{Ext-der-1}\\
    &\quad + (x_0+x_2)^q \pi_2 Y_U(u, x_0+x_2) (1-\pi_2) Y_U(v,x_2)\psi(w_1) \label{Ext-der-2}
\end{align}
We have seen in Proposition \ref{pi-one-module} that for $\pi_1 = 1-\pi_2$,  $(\psi(W_1), \pi_1 Y_U, \d_U|_{\psi(W_1)})$ is a $V$-module isomorphic to $W_1$. Thus (\ref{Ext-der-2}) can be simplified as 
$$(x_0+x_2)^q \pi_2 Y_U(u, x_0+x_2) \pi_1 Y_U(v,x_2)\psi(w_1) = (x_0+x_2)^q \pi_2 Y_U(u, x_0+x_2) \psi (Y_{W_1}(v,x_2)w_1). $$
Notice also that $\pi_2 Y_U(v, x) \pi_2 = Y_U(v, x) \pi_2$ since $W_2$ is a submodule, thus we can remove the first $\pi_2$ in (\ref{Ext-der-1}). Combining our knowledge for (\ref{Ext-der-1}) and (\ref{Ext-der-2}), we see that 
\begin{align}
    \begin{aligned}(x_0+x_2)^q \pi_2 Y_U(Y(u, x_0)v,x_2)\psi(w_1) &= (x_0+x_2)^q Y_U(u, x_0+x_2) \pi_2  Y_U(v,x_2)\psi(w_1)  \\
    &\quad + (x_0+x_2)^q \pi_2 Y_U(u, x_0+x_2) \psi (Y_{W_1}(v,x_2)w_1)
    \end{aligned}\label{Ext-der-3}
\end{align}
Let $f(x_0, x_2)$ to be the lower-truncated series in $W_2((x_0, x_2))$ representing both sides of (\ref{Ext-der-3}), then 
\begin{align*}
\pi_2 Y_U(Y(u,x_0)v, x_2)\psi(w_1) &= \iota_{x_2x_0} \frac{f(x_0, x_2)}{(x_2+x_0)^q} \\
&= \iota_{x_2x_0} E\left(Y_U(u, x_0+x_2) \pi_2  Y_U(v,x_2)\psi(w_1) \right)\\
&\quad +\iota_{x_2x_0} E\left(\pi_2Y_U(u, x_0+x_2) \psi (Y_{W_1}(v,x_2)w_1) \right)
\end{align*}
Now we substitute $x_2=\zeta$. With the knowledge of the definition of $Y_\H^L$ and $Y_\H^{s(R)}$, we immediately see that
\begin{align*}
\pi_2 Y_U(Y(u,x_0)v, \zeta)\psi(w_1) &= \left(Y_\H^L(u, x_0) \pi_2  Y_U(v,\zeta)\psi\right)(w_1) + \left(e^{x_0D_\H}Y_\H^{s(R)}(v, -x_0) \pi_2Y_U(u, \zeta) \psi \right)(w_1)
\end{align*}
Thus the map $v\mapsto \pi_2Y_U(v, \zeta)\psi$ is a derivation from $V$ to $\H_N(W_1, W_2)$. 
\end{proof}

\begin{prop}\label{equiv-ext-inner-prop}
    Let $U^{(1)}$, $U^{(2)}$ be two equivalent extensions with $T: U^{(1)}\to U^{(2)}$ be the $V$-module isomorphism fitting in the diagram (\ref{CommDiag}). Let $\psi^{(i)}: W_1 \to U^{(i)}$ be the section and $\pi_2^{(i)}: U^{(i)}\to W_2$ be projection as in Definition \ref{Ext_N-def} $(i = 1, 2)$. Then 
    $$\pi_2^{(1)}Y_U^{(1)}(v, \zeta)\psi^{(1)}-\pi_2^{(2)}Y_U^{(2)}(v, \zeta)\psi^{(2)}$$ 
    is an inner derivation corresponding to the element $\phi(\zeta)\in \H_N(W_1, W_2)$ given by 
    $$\phi(\zeta)w_1 = \pi_2^{(2)} T\psi^{(1)}w_1. $$
    Here $\phi(\zeta)$ is a constant $\overline{W_2}$-valued rational function. 
\end{prop}

\begin{proof}
For $i = 1, 2$, let $p^{(i)}: U^{(i)}\to W_1$ be the map as in (\ref{CommDiag}); write $U^{(i)} = W_2 \amalg \psi^{(i)}(W_1)$ where $\psi^{(i)}$ is the section of $p^{(i)}$ as in Definition \ref{Ext_N-def}; let $\pi^{(i)}_2$ be the projection as in Definition \ref{Ext_N-def}, $\pi^{(i)}_1 = 1- \pi_2^{(i)}$. Since (\ref{CommDiag}) is commutative, we have 
$$T(w_2, 0) = (w_2, 0), \qquad  p^{(1)}(w_2, \psi^{(1)}(w_1)) = p^{(2)} T(w_2,\psi^{(1)}(w_1)),$$
which tells us that 
$$Tw_2 = w_2,$$ 
and 
$$w_1 = p^{(2)} T \psi^{(1)} w_1. $$
If we write $T\psi^{(1)}w_1 = \pi_2^{(2)}T\psi^{(1)}w_1 + \pi_1^{(2)}T\psi^{(1)}w_1$, then since $p^{(2)}(W_2) = 0$, we have 
$$w_1 = p^{(2)}\pi_1^{(2)}T\psi^{(1)}w_1. $$
From Proposition \ref{pi-one-module}, we see that $\psi^{(2)}$ is the inverse of $p^{(2)}$. Thus we have
$$\psi^{(2)} w_1 = \pi_1^{(2)}T \psi^{(1)} w_1. $$
In summary, we have 
\begin{equation}\label{equiv-ext-inner-0}
    T(w_2, \psi^{(1)}w_1) = (w_2 + \pi_2^{(2)}T\psi^{(1)}w_1, \psi^{(2)} w_1).
\end{equation}
Since $T$ is an isomorphism, we know that 
\begin{equation}\label{equiv-ext-inner-1}
    T Y_U^{(1)}(v, x) = Y_U^{(2)}(v, x) T.
\end{equation}
Apply the left-hand-side of (\ref{equiv-ext-inner-1}) to $(w_2, \psi^{(1)}w_1)$, we have
\begin{align}
     &\quad T Y_U^{(1)}(v, x) (w_2, \psi^{(1)}w_1)\nonumber\\
     &= T(Y_U^{(1)}(v, x)w_2 + \pi_2^{(1)} Y_U^{(1)}(v, x) \psi^{(1)}w_1, \pi_1^{(1)} Y_U^{(1)}(v, x) \psi^{(1)}w_1)\nonumber \\
    &= T(Y_U^{(1)}(v, x)w_2 + \pi_2^{(1)} Y_U^{(1)}(v, x) \psi^{(1)}w_1, \psi^{(1)} (Y_{W_1}(v, x) w_1)\label{equiv-ext-inner-2}\\
    &= (Y_U^{(1)}(v, x)w_2 + \pi_2^{(1)} Y_U^{(1)}(v, x) \psi^{(1)}w_1 + \pi_2^{(2)}T\psi^{(1)}Y_{W_1}(v,x)w_1, \psi^{(2)} Y_{W_1}(v, x) w_1)\label{equiv-ext-inner-3}
\end{align}
Here (\ref{equiv-ext-inner-2}) follows from $\psi^{(1)}$ being a homomorphism (cf. Proposition \ref{pi-one-module}); (\ref{equiv-ext-inner-3}) follows from (\ref{equiv-ext-inner-0}). Apply the right-hand-side of (\ref{equiv-ext-inner-1}) to $(w_2, \psi^{(1)}w_1)$, we have
\begin{align}
    &\quad Y_U^{(2)}(v, x) T(w_2,\psi^{(1)}w_1) \nonumber\\
    &= Y_U^{(2)}(v, x) (w_2 + \pi_2^{(2)} T\psi^{(1)}w_1,\psi^{(2)}w_1)\label{equiv-ext-inner-4}\\
    &= (Y_U^{(2)}(v, x)w_2 + Y_U^{(2)}(v, x)\pi_2^{(2)} T\psi^{(1)}w_1 + \pi_2^{(2)}Y_U^{(2)}(v, x)\psi^{(2)}w_1, \pi_1^{(2)}Y_U^{(2)}(v, x)\psi^{(2)}w_1)\nonumber\\
    &= (Y_U^{(2)}(v, x)w_2 + Y_U^{(2)}(v, x)\pi_2^{(2)} T\psi^{(1)}w_1 + \pi_2^{(2)}Y_U^{(2)}(v, x)\psi^{(2)}w_1, \psi^{(2)}Y_{W_1}(v, x)w_1)\label{equiv-ext-inner-5}
\end{align}
Here (\ref{equiv-ext-inner-4}) follows from (\ref{equiv-ext-inner-0}); (\ref{equiv-ext-inner-5}) follows from $\psi^{(2)}$ being a homomorphism (cf. Proposition \ref{pi-one-module}). Notice also that 
$$Y_U^{(1)}(v,x)w_2 = Y_U^{(2)}(v, x) w_2 = Y_{W_2}(v, x)w_2,$$
thus the equality of (\ref{equiv-ext-inner-3}) and (\ref{equiv-ext-inner-5}), after substituting $x=\zeta$, implies  that
\begin{equation}\label{equiv-ext-inner-6}
    \pi_2^{(1)}Y_U^{(1)}(v, \zeta) \psi^{(1)}w_1 -  \pi_2^{(2)}Y_U^{(2)}(v, \zeta)\psi^{(2)}w_1=  
    -\pi_2^{(2)}T\psi^{(1)} Y_{W_1}(v, \zeta)w_1 + Y_{W_2}(v, \zeta)\pi_2^{(2)} T\psi^{(1)}w_1.
\end{equation}
If we now set $\phi(\zeta) = \pi_2^{(2)}T\psi^{(1)}$ as an element in $\H_N(W_1, W_2)$ that is constant in $\zeta$, then 
$$(Y_\H^L(v, x)\phi)(\zeta)w_1 = \iota_{\zeta x}E(Y_{W_2}(v, x+\zeta) \phi(\zeta))w_1 = Y_{W_2}(v, \zeta+x)\pi_2^{(2)}T\psi^{(1)}w_1$$
$$(Y_\H^{s(R)}(v, x)\phi)(\zeta)w_1 = \iota_{\zeta x}E(\phi(\zeta)Y_{W_1}(v,x+\zeta)w_1) = \pi_2^{(2)}T\psi^{(1)}Y_{W_1}(v, \zeta+x)w_1$$
Clearly, $D_\H\phi(\zeta) = 0$. This means that $\phi$ is a vacuum-like element in $\H_N(W_1, W_2)$. Thus we see that right-hand-side of (\ref{equiv-ext-inner-6}) is precisely 
$$\lim_{x\to 0} \left((Y_\H^L(v, x)\phi)(\zeta) - (Y_\H^{s(R)}(v, x)\phi)(\zeta) \right)$$
This shows that the left-hand-side of (\ref{equiv-ext-inner-6}) is an inner derivation associated to $\phi(\zeta) = \pi_2^{(2)}T\psi^{(1)}$. \end{proof}

\begin{rema}\label{rmk-indep-section}
In particular, Proposition \ref{equiv-ext-inner-prop} shows that the cohomology class obtained Proposition \ref{Ext-der-prop} is independent of the choice of the section $\psi^{(1)}$. 
\end{rema}

\begin{cor}\label{def-cal-F}
The map $\mathcal{F}: \Ext_N^1(W_1, W_2) \to \widehat{H}^1(V, \H_N(W_1, W_2))$ sending $U$ to the derivation $v\mapsto \pi_2 Y_U(v, \zeta) \psi$ is well-defined. Here $U, \pi_2$ and $\psi$ are as in Definition \ref{Ext_N-def}. 
\end{cor}

\subsection{1-cocycles in the bimodule $\H_N(W_1, W_2)$}

To understand the converse map, we need some convergence results regarding 1-cocycles. We will start from the following lemmas concerning complex series. 

\begin{lemma}\label{IterSeries}
\begin{enumerate}
\item \label{IterSeries-L}Let $n$ be a positive integer. Let $f$ be a rational function in $z_1, ..., z_n$. Let $T$ be a connected multicircular domain on which the lowest power of $z_n$ in the Laurent series expansion of $f(z_1, ..., z_n)$ is the same as the negative of the order of pole $z_n=0$. Let $S$ be a nonempty open subset of $T$ and $S'$ be the image of $S$ via the projection $(z_1, ..., z_n)\mapsto (z_1, ..., z_{n-1})$. Assume that for each fixed $k_n \in \Z$, the series
$$\sum_{k_1, k_2, ..., k_{n-1}\in \Z}a_{k_1k_2...k_{n-1}k_n} z_1^{k_1}z_2^{k_2}\cdots z_{n-1}^{k_{n-1}}$$
converges absolutely for every $(z_1, z_2, ..., z_{n-1})\in S'$, and 
\begin{equation}\label{IterSeries-1}
\sum_{k_n\in \Z}\left(\sum_{k_1, k_2, ..., k_{n-1}\in \Z}a_{k_1k_2...k_{n-1}k_n} z_1^{k_1}z_2^{k_2}\cdots z_{n-1}^{k_{n-1}}\right)z_n^{k_n},
\end{equation}
viewed as a series whose terms are $\left(\sum\limits_{k_1, k_2, ..., k_{n-1}\in \Z}a_{k_1k_2...k_{n-1}k_n} z_1^{k_1}z_2^{k_2}\cdots z_{n-1}^{k_{n-1}}\right)z_n^{k_n}$, is lower-truncated in $z_n$ and converges to $f(z_1, ..., z_n)$ for every $(z_1, z_2, ..., z_{n-1}, z_n)\in S$.  Then the corresponding Laurent series
\begin{equation}\label{IterSeries-2}
\sum_{k_1, k_2, ...,  k_{n-1}, k_n\in \Z}a_{k_1k_2...k_n} z_1^{k_1}z_2^{k_2}\cdots z_{n-1}^{k_{n-1}} z_n^{k_n},
\end{equation}
converges absolutely to $f(z_1, ..., z_n)$ for every $(z_1, ..., z_n)\in T$ 
\item \label{IterSeries-U}Let $n$ be a positive integer. Let $f$ be a rational function in $z_1, ..., z_n$. Let $T$ be a connected multicircular domain on which the highest power of $z_n$ in the Laurent series expansion of $f(z_1, ..., z_n)$ is the same as the negative of the order of pole $z_n=\infty$. Let $S$ be a nonempty open subset of $T$ and $S'$ be the image of $S$ via the projection $(z_1, ..., z_n)\mapsto (z_1, ..., z_{n-1})$. Assume that for each fixed $k_n \in \Z$, the series
$$\sum_{k_1, k_2, ..., k_{n-1}\in \Z}a_{k_1k_2...k_{n-1}k_n} z_1^{k_1}z_2^{k_2}\cdots z_{n-1}^{k_{n-1}}$$
converges absolutely for every $(z_1, z_2, ..., z_{n-1})\in S'$, and 
\begin{equation*}
\sum_{k_n\in \Z}\left(\sum_{k_1, k_2, ..., k_{n-1}\in \Z}a_{k_1k_2...k_{n-1}k_n} z_1^{k_1}z_2^{k_2}\cdots z_{n-1}^{k_{n-1}}\right)z_n^{k_n},
\end{equation*}
viewed as a series whose terms are $\left(\sum\limits_{k_1, k_2, ..., k_{n-1}\in \Z}a_{k_1k_2...k_{n-1}k_n} z_1^{k_1}z_2^{k_2}\cdots z_{n-1}^{k_{n-1}}\right)z_n^{k_n}$, is upper-truncated in $z_n$ and converges to $f(z_1, ..., z_n)$ for every $(z_1, z_2, ..., z_{n-1}, z_n)\in S$.  Then the corresponding Laurent series
\begin{equation*}
\sum_{k_1, k_2, ...,  k_{n-1}, k_n\in \Z}a_{k_1k_2...k_n} z_1^{k_1}z_2^{k_2}\cdots z_{n-1}^{k_{n-1}} z_n^{k_n},
\end{equation*}
converges absolutely to $f(z_1, ..., z_n)$ for every $(z_1, ..., z_n)\in T$ 
\end{enumerate}
\end{lemma}

\begin{proof}
An exposition can be found in \cite{Q-Mod}, Lemma 4.5 and Lemma 4.7. 
\end{proof}

\begin{prop}\label{der-weak-assoc-conv}
Let $\Phi: V\to \reallywidetilde{\H_N(W_1, W_2)}_{z}$ be a 1-cocycle. For $v\in V$, let $\phi(v, \zeta)=(\Phi(v;0))(\zeta)\in \H_N(W_1, W_2)$. Then for every $v_1\in V, w_1\in W_1$, the series
\begin{align*}
    (Y_\H^L(v_1, z_1)\Phi(v_2;z_2))(\zeta)w_1 &&  & |\zeta|>|z_1|>|z_2|\\
    (Y_\H^{s(R)}(v_2, z_2)\Phi(v_1;z_1))(\zeta)w_1 && \text{     converge absolutely when}\hspace{1cm} & |\zeta|>|z_2|>|z_1|\\
    \Phi(Y(v_1, z_1-z_2)v_2; z_2)(\zeta)w_1 && & |\zeta|>|z_1-z_2|+|z_2|, |z_1-z_2|>0
\end{align*}
respectively to the $\overline{W_2}$-valued rational functions
    $$E(Y_{W_2}(v_1, \zeta+z_1)\phi(v_2, \zeta+z_2)w_1)$$
    $$E(\phi(v_2, \zeta+z_2)Y_{W_1}(v_1, \zeta+z_1)w_1)$$
    $$E(Y_{W_2}(v_1, \zeta+z_1)\phi(v_2, \zeta+z_2)w_1)+E(\phi(v_2, \zeta+z_2)Y_{W_1}(v_1, \zeta+z_1)w_1)$$
\end{prop}

\begin{proof}
We start with the following observation: Since $\Phi$ satisfies the $D_\H$-derivative properties, we have
$$\Phi(v;z)= e^{zD_\H}\Phi(v;0)$$
for every $v\in V$ and every complex number $z$. Acting both sides on $w_1\in W$ with the parameter $\zeta$, we see that
$$(\Phi(v;z))(\zeta)w_1 = e^{zD_\H}(\Phi(v;0))(\zeta)w_1 = e^{zD_\H}\phi(v,\zeta)w_1 = \sum_{n=0}^\infty\left(\phi^{(n)}(v, \zeta)w_1\right)\frac{z^n}{n!}, $$
where $\phi^{(n)}(v, \zeta)$ is the $n$-th derivative of $\phi(v, \zeta)$. Now we fix $v_1, v_2\in V, w_1\in W_1$ and analyze each term in the cocycle condition. 

\begin{enumerate}
    \item For the left action term in the cocycle condition, we have
\begin{align}
    (Y_\H^L(v_1, z_1)\Phi(v_2;z_2))(\zeta)w_1 &= \sum_{n=0}^\infty\left((Y_\H^L(v_1, z_1)\phi^{(n)})(v_2, \zeta)w_1 \right)\frac{z_2^n}{n!} \nonumber\\
    &= \sum_{n=0}^\infty\left(\iota_{\zeta z_1}E(Y_{W_2}(v_1, \zeta+z_1)\phi^{(n)}(v_2, \zeta)w_1) \right)\frac{z_2^n}{n!}.\label{Cocycle-1-1}
\end{align}
We now show that (\ref{Cocycle-1-1}) converges to the same $\overline{W_2}$-valued rational function as 
\begin{align}
Y_{W_2}(v_1, \zeta+z_1)\phi(v_2, \zeta+z_2)w_1.\label{Cocycle-1-1-R}
\end{align}
From the definition of $\H_N(W_1, W_2)$, we know that for every $w_2'\in W_2'$, 
\begin{align}
    \langle w_2', Y_{W_2}(v_1, \zeta+z_1)\phi(v_2, \zeta+z_2)w_1\rangle \label{Cocycle-1-2}
\end{align}
converges absolutely in the region $|\zeta+z_1|>|\zeta+z_2|>0$ to a rational function $f(z_1, z_2, \zeta)$ with the only possible poles at $z_1-z_2=0, z_1+\zeta=0, z_2+\zeta=0$. From the formal Taylor theorem,  (\ref{Cocycle-1-2}) equals to 
\begin{align}
    \sum_{n=0}^\infty \langle w_2', Y_{W_2}(v_1, \zeta+z_1)\phi^{(n)}(v_2, \zeta)w_1\rangle \frac{z_2^n}{n!}\label{Cocycle-1-3}
\end{align}
in the region $|\zeta|>|z_2|, |\zeta+z_1|>|\zeta+z_2|>0$. In particular, (\ref{Cocycle-1-3}) converges absolutely to $f(z_1, z_2, \zeta)$ when $|\zeta|>|z_2|, |\zeta+z_1|>|\zeta|+|z_2|$. 

Note that for every fixed power $n$ of $z_2$, the coefficient 
\begin{align}
    \langle w_2', Y_{W_2}(v_1, \zeta+z_1)\phi^{(n)}(v_2, \zeta)w_1\rangle \label{Cocycle-1-4}
\end{align}
of the series of (\ref{Cocycle-1-3}) converges absolutely to the rational function 
$$\langle w_2', E(Y_{W_2}(v_1, \zeta+z_1)\phi^{(n)}(v_2, \zeta)w_1)\rangle$$
with the only possible pole at $\zeta=0, z_1=0, z_1+\zeta = 0$. If we expand the negative powers of $\zeta+z_1$ in (\ref{Cocycle-1-4}) as power series in $z_1$, we see that (\ref{Cocycle-1-4}) equals to
\begin{align}
    \iota_{\zeta z_1}\langle w_2', E(Y_{W_2}(v_1, \zeta+z_1)\phi^{(n)}(v_2, \zeta)w_1)\rangle. \nonumber
\end{align}
when $|\zeta+z_1|>|\zeta|>|z_1|$. Thus we know (\ref{Cocycle-1-3}) equals to the series 
\begin{align}\sum_{n=0}^\infty \left(\iota_{\zeta z_1}\langle w_2', E(Y_{W_2}(v_1, \zeta+z_1)\phi^{(n)}(v_2, \zeta)w_1)\rangle\right)\frac{z_2^n}{n!}\label{Cocycle-1-5}
\end{align}
when $|\zeta|>|z_2|, |\zeta+z_1|>|\zeta|+|z_2|, |\zeta|>|z_1|$. 
Regard (\ref{Cocycle-1-5}) as a series in $z_1, z_2, \zeta$ that is globally truncated in $z_2$ and $z_1$, take 
\begin{align*}
    S &= \{(z_1, \zeta, z_2): |\zeta|>|z_2|, |\zeta+z_1|>|\zeta|+|z_2|, |\zeta|>|z_1|\},\\
    S' &= \{(z_1, \zeta): |\zeta+z_1|>|\zeta|>|z_1|\}, \\
    T &= \{(z_1, \zeta, z_2): |\zeta|>|z_1|>|z_2|\}, 
\end{align*}
and apply Lemma \ref{IterSeries} (\ref{IterSeries-L}), we see that (\ref{Cocycle-1-5}) converges absolutely to $f(z_1, z_2, \zeta)$ when $|\zeta|>|z_1|>|z_2|$. Thus the series (\ref{Cocycle-1-1}) converges absolutely to the $\overline{W_2}$-valued rational function (\ref{Cocycle-1-1-R}).
\item For the right action term in the cocycle condition, we know that 
\begin{align}
    (Y_\H^{s(R)}(v_2, z_2)\Phi(v_1; z_1))(\zeta)w_1 &= \sum_{n=0}^{\infty} \left((Y_\H^{s(R)}(v_2, z_2)\phi^{(n)})(v_1, \zeta)w_1 \right)\frac{z_1^n}{n!}\nonumber \\
    &=\sum_{n=0}^{\infty} \left(\iota_{\zeta z_2} E(\phi^{(n)}(v_1, \zeta)Y_{W_1}(v_2, \zeta+z_2)w_1) \right)\frac{z_1^n}{n!}. \label{Cocycle-2-1}
\end{align}
Similarly, we can show that the (\ref{Cocycle-2-1}) converges absolutely to the same $\overline{W_2}$-value rational function as 
\begin{align}
    \phi(v_1, \zeta+z_1)Y_{W_1}(v_2, \zeta+z_2)w_1. \label{Cocycle-2-1-R}
\end{align}
While we will not repeat the argument here, it should be noted that the region of convergence is $|\zeta|>|z_2|>|z_1|$ and is disjoint to that of the series (\ref{Cocycle-1-1}). Nevertheless, it does not really matter in our subsequent discussion. 

\item For the middle term in the cocycle condition, we first note that
$$\Phi(Y(v_1, z_1-\eta)Y(v_2, z_2-\eta)\one, \eta)$$
converges absolutely to something in $\reallywidetilde{\H_N(W_1, W_2)}_{z_1z_2}$ in the region $|\eta|>|z_1-\eta|>|z_2-\eta|>0$ that is independent of the choice of $\eta$. If we evaluate $\eta=z_2$ then apply to $w_1$, we obtain the following iterated series
\begin{align}
    (\Phi(Y(v_1, z_1-z_2)v_2, z_2))(\zeta)w_1&= \sum_{n=0}^\infty \left((\phi^{(n)}(Y(v_1, z_1-z_2)v_2, \zeta)w_1\right)\frac{z_2^n}{n!}\nonumber\\
    &= \sum_{n=0}^\infty \left(\sum_{m\in \Z} \left(\phi^{(n)} ((v_1)_m v_2, \zeta)w_1\right) (z_1-z_2)^{-m-1}\right)\frac{z_2^n}{n!}\label{Cocycle-3-1}
\end{align}
If we only assumed the composable condition of $\Phi$, then it is not necessary for (\ref{Cocycle-3-1}) to converge to a rational function. However, $\Phi$ is now a 1-cocycle in the bimodule $\H_N(W_1, W_2)$, thus we have
\begin{align}
E((\Phi(Y(v_1, z_1-z_2)v_2, z_2))(\zeta)w_1) = E((Y_\H^L(v_1, z_1)\Phi)(v_2, z_2)w_1) + E((Y_\H^{s(R)}(v_2, z_2)\Phi)(v_1, z_1)w_1).    \label{Cocycle-condition}
\end{align}
Since the $\overline{W_2}$-valued rational function on the right-hand-side are given by $Y_{W_2}(v_1, \zeta+z_1)\phi(v_2, \zeta+z_2)w_1$ and $\phi(v_1, \zeta+z_1)Y_{W_1}(v_2, \zeta+z_2)w_1$, from the definition of $\H_N(W_1, W_2)$, there exists constants $p_{1}, p_2$, $p_{12}$ depending only on the pairs $(v_1, w_1), (v_2, w_2)$ and $(v_1, v_2)$ respectively, such that 
\begin{align}
    (z_1-z_2)^{p_{12}}(z_1+\zeta)^{p_1}(z_2+\zeta)^{p_2} (\Phi(Y(v_1, z_1-z_2)v_2, z_2))(\zeta)w_1 \in W_2[[z_1, z_2, \zeta]]\label{Cocycle-3-2}
\end{align}
Therefore, the series (\ref{Cocycle-3-1}) is precisely the product of the series (\ref{Cocycle-3-2}) with the expansions of $$\iota_{\zeta z_2}(\zeta+z_2)^{-p_{2}} = \sum_{n=0}^\infty \binom{-p_2}{n} \zeta^{-p_2-n}z_2^n$$
and 
$$\iota_{\zeta, z_2+z_1-z_2}(\zeta+z_1)^{-p_{1}} = \sum_{n=0}^\infty \binom{-p_1}{n} \zeta^{-p_1-n}\sum_{i=0}^n \binom{n}i (z_1-z_2)^{n-i}z_2^i.$$
Consequently, the series (\ref{Cocycle-3-1}) converges absolutely to a $\overline{W_2}$-valued rational function with the only possible poles at $z_1+\zeta = 0, z_2+\zeta=0, z_1-z_2=0$ when $|\zeta|>|z_1-z_2|+|z_2|, |z_1-z_2|>0$. 
\end{enumerate}

\end{proof}

\begin{rema}
The proof of the convergence unfortunately does not generalize to higher cohomologies. Details will be discussed in the subsequent paper \cite{Q-Ext-n}. 
\end{rema}

\begin{prop}\label{der-weak-assoc}
Let $\Phi: V\to \reallywidetilde{\H_N(W_1, W_2)}_{z}$ be a 1-cocycle. For $v\in V$, let $\phi(v, \zeta)=(\Phi(v;0))(\zeta)\in \H_N(W_1, W_2)$. Then for every $v_1\in V, w_1\in W_1$, there exists $p\in \N$, such that for every $v_2\in V$, 
\begin{align*}
    (z_1+\zeta)^p\phi(Y(v_1, z_1)v_2, \zeta)w_1 &= (z_1+\zeta)^pY_{W_2}(v_1, z_1+\zeta)\phi(v_2, \zeta)w_1 \\ &\quad + (z_1+\zeta)^p\phi(v_1, \zeta+z_1)Y_{W_1}(v_2, \zeta)w_1
\end{align*}
in $W_2((z_1, \zeta))$. 
\end{prop}

\begin{proof}
From the cocycle condition (\ref{Cocycle-condition}) together with the conclusion of Proposition \ref{der-weak-assoc-conv}, we have
$$E(\Phi(Y(v_1, z_1-z_2)v_2, z_2)w_1) = E(Y_{W_2}(v_1, z_1+\zeta)\phi(v_2, z_2+\zeta)w_1) + E(\phi(v_1, z_1+\zeta)Y_{W_1}(v_2, z_2+\zeta)w_1))$$
Multiplying both side by $(z_1-z_2)^{p_{12}}(z_1+\zeta)^{p_1}(z_2+\zeta)^{p_2}$ where $p_1, p_2, p_{12}$ are numbers depending only on the pairs $(v_1, w_1), (v_2, w_2), (v_1, v_2)$ (see the proof of Proposition \ref{der-weak-assoc-conv}), then evaluate $z_2=0$ on both sides, we have 
\begin{align*}
    z_1^{p_{12}}(z_1+\zeta)^{p_1}\zeta^{p_2}\phi(Y(v_1, z_1)v_2, \zeta)w_1 &= z_1^{p_{12}}(z_1+\zeta)^{p_1}\zeta^{p_2}Y_{W_2}(v_1, z_1+\zeta)\phi(v_2, \zeta)w_1 \\ &\quad + z_1^{p_{12}}(z_1+\zeta)^{p_1}\zeta^{p_2}\phi(v_1, \zeta+z_1)Y_{W_1}(v_2, \zeta)w_1 
\end{align*}
as series in $W_2[[z_1, \zeta]]$. The conclusion then follows by dividing both sides by $z_1^{p_{12}}$ and $\zeta^{p_2}$. 
\end{proof}

\begin{prop}\label{inner-der}
Let $\Phi: V\to \reallywidetilde{\H_N(W_1, W_2)}_{z}$ be a 1-coboundary. For $v\in V$, let $\phi(v, \zeta)=(\Phi(v;0))(\zeta)\in \H_N(W_1, W_2)$. Then  $\phi(v,\zeta) = Y_{W_2}(v, \zeta)\phi_{-1} - \phi_{-1}Y_{W_1}(v, \zeta)$ for some linear map $\phi_{-1}: W_1 \to W_2$. 
\end{prop}

\begin{proof}
Since $\Phi$ is a 1-coboundary, there exists some $\phi\in \H_N(W_1, W_2)$ satisfying $D_\H\phi = 0$, such that,
\begin{align}
    \Phi(v;z) = Y_\H^L(v,z)\phi - Y_\H^{s(R)}(v, z)\phi \label{Coboundary-1}
\end{align}
Note that since $D_\H\phi = 0$ means for every $w_1\in W_1$, $(\partial / \partial \zeta)\phi(\zeta)w_1 = 0$, which further implies that $\phi(\zeta)w_1 = \phi_{-1}(w_1)\in W_2$. Moreover, for every $v\in V$, 
\begin{align}
    Y_\H^L(v, z)\phi(\zeta)w_1 &= \iota_{\zeta z} Y_{W_2}(v, \zeta+z) \phi_{-1}w_1 \label{Coboundary-2}\\
    Y_\H^{s(R)}(v, z)\phi(\zeta)w_1 &= \iota_{\zeta z} \phi_{-1} Y_{W_1}(v, \zeta+z) w_1 \label{Coboundary-3}
\end{align}
Since $D_\H\phi=0$, we are allowed to evaluate $z=0$ in (\ref{Coboundary-1}), (\ref{Coboundary-2}) and (\ref{Coboundary-3}), so as to conclude that 
$$(\Phi(v;0))(\zeta)w_1 = Y_{W_2}(v, \zeta)\phi_{-1}w_1 - \phi_{-1} Y_{W_1}(v,\zeta),$$
the left-hand-side of which is precisely $\phi(v,\zeta)w_1$. 


\end{proof}

\subsection{Extension in $\Ext^1_N(W_1, W_2)$ from a derivation}

\begin{prop}\label{Der-to-Ext-Prop}
Let $F: V\to \H_N(W_1, W_2)$ be a derivation. Then for every $v\in V$, $F(v)$ is a map $W_1 \to \widehat{(W_2)}_\zeta$. We use the notation $F(v, \zeta)w_1$ to denote the image of $w_1$ via the map $F(v)$ in $\widehat{(W_2)}_\zeta$. 
Then the vector space $$U = W_2 \amalg W_1$$
equipped with the vertex operator map 
$$Y_U(v, x)(w_2, w_1) = (Y_{W_2}(v, x)w_2 + F(v, x) w_1, Y_{W_1}(v, x)w_1),$$
the operator $\d_U = \d_{W_2}\amalg \d_{W_1}$  forms a left $V$-module that fits in the exact sequence 
$$\xymatrix{ 0 \ar[r] & W_2 \ar[r] & U \ar[r]^p & W_2 \ar[r] & 0.  }$$
where the map $p: U = W_2\amalg W_1 \to W_1$ is precisely the projection. Moreover, $U$ satisfies the conditions in Definition \ref{Ext_N-def}. 
\end{prop}

\begin{proof}
We only give a sketch here. The grading on $U$ is given by the gradings on $W_1$ and $W_2$, with the homogeneous subspaces being $U_{[m]} = (W_2)_{[m]}\amalg (W_1)_{[m]}$. Thus the grading is lower bounded. The $\d$-commutator formula follows from the $\d$-conjugation property in Proposition-Definition \ref{Def-H_N}. The identity property follows from $F(\one) = 0$, due to $F$ being a derivation. The $D$-derivative formula also follows from 
$$F(D_V v, \zeta) = D_\H F(v, \zeta) = \frac d {d\zeta} F(v, \zeta). $$
To show the weak associativity, we notice that from Proposition \ref{der-weak-assoc}, 
$v_1\in V, w_1\in W_1$, there exists $p\in \N$ such that for every $v_2\in V$, 
\begin{align*}
    (x_0+\zeta)^p F(Y(v_1, x_0)v_2, \zeta)w_1 &= (x_0+\zeta)^p Y_{W_2}(v_1, x_0+\zeta)F(v_2,\zeta)w_1\\
    & \quad + (x_0+\zeta)^p F(v_1,x_0+\zeta)Y_{W_1}(v_2, \zeta)w_1. 
\end{align*}
This implies that the vertex operator $Y_U$ satisfies the weak associativity with pole-order condition. The relation $p Y_U(v, x) = Y_{W_1}(v, x) p$ follows trivially from the definition. The conditions of Definition \ref{Ext_N-def} are satisfied with the section $\psi: W_1\to W_2\amalg W_1, w_1 \mapsto (0, w_1)$ and the projection $\pi_2: W_2\amalg W_1 \to W_2, (w_2, w_1)\mapsto w_2$. Indeed, $\pi_2 Y_U(v, x) \psi = F(v, x)$ for every $v\in V$. 
\end{proof}

\begin{prop}
Let $F^{(1)}, F^{(2)}: V\to \H_N(W_1, W_2)$ be two derivations that differ by an inner derivation. Let $U^{(1)}$ and $U^{(2)}$ be the extensions given respectively by $F^{(1)}$ and $F^{(2)}$. Then $U^{(1)}$ and $U^{(2)}$ are equivalent extensions. \end{prop}

\begin{proof}
We only give a sketch here. From Proposition \ref{inner-der}, we see that 
$$F^{(1)}(v, \zeta) - F^{(2)}(v, \zeta) = Y_{W_2}(v, \zeta) \phi_{-1} - \phi_{-1}Y_{W_1}(v, \zeta)$$
for some linear homogeneous map $\phi_{-1}: W_1\to W_2$. Define $T: U^{(1)} = W_2\amalg W_1\to W_2\amalg W_1 = U^{(2)}$ by 
$$T(w_2, w_1) = (w_2 + \phi_{-1}(w_1), w_1)$$
It is straightforward to check that for every $v\in V$,
$$T Y_U^{(1)}(v,x) = Y_{U}^{(2)}(v, x)T$$
where $Y_{U}^{(1)}$ and $Y_U^{(2)}$ be the vertex operators constructed in Proposition \ref{Der-to-Ext-Prop}. 
\end{proof}

\begin{cor}\label{def-cal-G}
The map $\mathcal{G}: \widehat{H}^1(V, \H_N(W_1, W_2))\to\Ext_N^1(W_1, W_2)$ sending a derivation $F$ to equivalence class of the module $U=W_2\amalg W_1$ given in Proposition \ref{Der-to-Ext-Prop} is well-defined. 
\end{cor}

\subsection{Proof of the main theorem}

\begin{thm}\label{Main-Ext-1}
$\Ext_N^1(W_1, W_2)$ is in one-to-one correspondence to $\widehat{H}^1(V, \H_N(W_1, W_2))$. 
\end{thm}

\begin{proof}
Recall the maps $\mathcal{F}$ and $\mathcal{G}$ from Corollary \ref{def-cal-F} and Corollary \ref{def-cal-G}. We first show that $\mathcal{F}\circ\mathcal{G} = Id$. Starting from a 1-cohomology class $F$, we have obtained an extension $U = \mathcal{G}(F)$ satisfying conditions in Definition \ref{Ext_N-def}. From Remark \ref{rmk-indep-section}, we see that the image $\mathcal{F}(U) = \mathcal{G}(\mathcal{F}(F))$ is independent of the choice of the section. If we use the section $\psi$ as described at the end of the proof of Proposition \ref{Der-to-Ext-Prop}, the derivation obtained in Proposition \ref{Ext-der-prop} is precisely $F$. So we conclude that  $\mathcal{F}(\mathcal{G}(F))=F$, i.e., $\mathcal{F}\circ\mathcal{G} = Id$. 

Now we show that $ \mathcal{G}\circ\mathcal{F} = Id$. Starting from an extension $U$ satisfying the condition in Definition \ref{Ext_N-def}, we obtained a 1-cohomology class corresponding to $F(v, x) = \pi_2 Y_U(v, x)\psi^{(1)}$. If we use this $F$ to construct the module $\mathcal{G}(\mathcal{F}(U))$, we see that 
$$Y_{\mathcal{G}(\mathcal{F}(U))}(v, x)(w_2, w_1) = (Y_{W_2}(v, x)w_2 + \pi_2 Y_{U}(v, x)\psi^{(1)} w_1, Y_{W_1}(v, x)w_1),$$
It is straightforward to check that $id\amalg \psi^{(1)}: \mathcal{G}(\mathcal{F}(U)) = W_2\amalg W_1 \to U = W_2\amalg \psi^{(1)}(W_1)$ sending $(w_2, w_1)$ to $(w_2, \psi^{(1)}w_1)$ is an isomorphism. So $\mathcal{G}(\mathcal{F}(U))$ represents the same element in $\Ext_N^1(W_1, W_2)$ as $U$ does. 
\end{proof}

\begin{rema}
The bimodule $\H_N(W_1, W_2)$ is the vertex algebraic analogue of the bimodule $\Hom_\C(M_1, M_2)$ associated with the two left modules $M_1, M_2$ for an associative algebra. What we have shown in Theorem \ref{Main-Ext-1} is just the analogue of the bijection 
$$\Ext^1(M_1, M_2) \simeq HH^1(A, \Hom_\C(M_1, M_2))$$
where $HH^1(A, \Hom_\C(M_1, M_2))$ is the first Hochschild cohomology with respect to the bimodule $\Hom_\C(M_1, M_2)$.
\end{rema}


\subsection{Some comments on reductivity}

\begin{thm}\label{Ext1-Thm}
Let $V$ be a MOSVA such that every left $V$-module is completely reducible, then for every left $V$-module $W_1$ and $W_2$, 
$$\widehat{H}^1(V, \H_N(W_1, W_2)) = 0$$
\end{thm}

\begin{proof}
If every $V$-module is completely reducible, then the only possible extension of $W_1$ by $W_2$ is represented by the $V$-module direct sum $W_1\oplus W_2$. If we pick $\psi$ to be the embedding of $W_1$ to $W_1\oplus W_2$, then clearly $\pi_2 Y_{W_1\oplus W_2}(v, x)\psi = 0$. Thus every 1-cocycle is cohomologous to 0. 
\end{proof}

\begin{rema}
In \cite{HQ-Red}, we studied the left $V$-module $W$ satisfying some technical but natural convergence conditions. The conditions can be reformulated as: for every $V$-submodule $W_2$ of $U$, we require the existence of a complementary subspace $W_1$, such that the projection operators $\pi_{W_2}: U\to W_2$ and $\pi_{W_1}: U\to W_1$ makes $\pi_{W_2} Y_U(v, z)\pi_{W_1} \in \H_N(W_1, W_2)$ (here $W_1$ is a $V$-module with the vertex operator $\pi_{W_1} Y_{W_1})$. We proved that such a left $V$-module $U$ is completely reducible if for every $V$-bimodule, the first cohomology is given by the zero-mode derivations. It should be remarked here that proof in \cite{HQ-Red} used only the bimodule $\H_N(W_1, W_2)$ and the cohomology $\widehat{H}^1(V, \H_N(W_1, W_2))$. Therefore, Theorem \ref{Ext1-Thm} can be viewed as a converse of the result in \cite{HQ-Red}. 
\end{rema}

\begin{rema}
In the early version of \cite{HQ-Red} and in the author's PhD thesis \cite{Q-Thesis}, the reductivity was proved with the assumption that $\widehat{H}^1(V, W) = 0$ for every $V$-bimodule $W$. Theorem \ref{Ext1-Thm} explains why we ignored the zero-mode derivations: $\H_N(W_1, W_2)$ does not have nontrivial zero-mode derivations. Indeed, one can directly show that every zero-mode derivation $V\to \H_N(W_1, W_2)$ is an inner derivation. We decide not to include the details here as it is too technical and much less conceptual than what is shown here. 
\end{rema}

\section{Category $\mathcal{C}_N$}

In this section we study the category $\mathcal{C}_N$ consisting of modules $W$ such that for every $V$-submodule $W_2$ of $U$, there exists a complementary subspace $W_1$, such that the projection operators $\pi_{W_2}: U\to W_2$ and $\pi_{W_1}: U\to W_1$ makes $\pi_{W_2} Y_U(v, z)\pi_{W_1} \in \H_N(W_1, W_2)$. Based on the results and remarks from Section \ref{Section-3} and from the paper \cite{Q-Ext-n}, we believe that $\mathcal{C}_N$ is the correct category to develop the homological methods of vertex algebras. In case $V$ is a grading-restricted vertex algebra containing a nice subalgebra, we show that every $V$-module is in $\mathcal{C}_N$, and the dimension of the space of extensions is finite.

\subsection{Category $\mathcal{C}_N$ and some general facts} 
\begin{defn}\label{composibility}
Let $\mathcal{C}_N$ be the category whose objects are left $V$-modules $W$ such that for every $V$-submodule $W_2\subseteq W$, there exists a vector subspace $W_1\subseteq W$ satisfying the following conditions:
\begin{enumerate}
    \item $W = W_1 \amalg W_2$ as vector spaces with projection operator $\pi_{W_1}: W \to W_1, \pi_{W_2}: W \to W_2$.
    \item For every $v\in V$, $\pi_{W_2} Y_W(v, x) \pi_{W_1} \in \H_N(W_1, W_2)$. 
\end{enumerate}
Here $W_1$ is regarded as $V$-module with the vertex operator $v\mapsto \pi_{W_1}Y_W(v, x)$. Morphisms of $\mathcal{C}_N$ are the left $V$-module homomorphisms. 
\end{defn}

\begin{rema}
If $W_1$ is a complement of $W_2$, i.e., the vector space $W_1$ is also a submodule of $W$, then for every $v\in V$, $\pi_{W_2} Y_{W}(v,x)\pi_{W_1} = 0$. The condition (2) holds trivially. Consequently, if $V$ is a strongly rational vertex operator algebra, then for every $N\in \Z$, $\CC_N$ coincides with the category of $V$-modules. 
\end{rema}

\begin{rema}\label{composibility-formal}
The composability and $N$-weight-degree conditions in Proposition-Definition \ref{Def-H_N} can also be formulated in terms of formal variables. 
\begin{enumerate}
    \setcounter{enumi}{2}
    \item The composability condition: For every $k,l\in \N$, every $u_1, ..., u_{k+l}\in V$, there exists $p\in \N, p_i, q_i\in \N (i=1, ..., k+l), q_{ij}\in \N (1\leq i < j \leq k+l)$, and a formal power series $g(x_1, ..., x_k, x, x_{k+1},..., x_{k+l})\in W_2[[x_1, ..., x_k, x, x_{k+1}, ..., x_{k+l}]]$, such that 
    \begin{align*}
        & Y_{W_2}(u_1, x_1)\cdots Y_{W_2}(u_k, x_k)\phi(x) \pi_{W_1}Y_W(u_{k+1}, x_{k+1})\cdots \pi_{W_1}Y_W(x_{k+l}, x_{k+l})w_1
        \\
        & = \frac{g(x_1, ..., x_k, x, x_{k+1},.., x_{k+l})}{x^p \prod\limits_{i=1}^{k+l} x_i^{q_i} \prod\limits_{i=1}^{k}(x_i-x)^{p_i} \prod\limits_{i=k+1}^{k+l}(x-x_i)^{p_i} \prod\limits_{1\leq i < j \leq k+l}(x_i-x_j)^{q_{ij}}}
    \end{align*}
    where the right-hand-side series is in $W_2[[x_1, x_1^{-1}, ..., x_k, x_k^{-1}, x, x^{-1}, x_{k+1}, x_{k+1}^{-1}, ..., x_{k+l}, x_{k+l}^{-1}]]$ with all negative powers of binomials expanded as power series of the second variable in accordance with the binomial expansion convention. 
\item $N$-weight-degree condition: With same notations as in (3), the series 
\begin{align*}
    \frac{g(x_1+x_{k+l}, ..., x_k+x_{k+l}, x+x_{k+l}, x_{k+1}, ..., x_{k+l-1}+x_{k+l}, x_{k+l})}{(x_{k+l}+x)^p\prod\limits_{i=1}^{k+l}(x_{k+l}+x_i)^{q_i}\prod\limits_{i=1}^{k}(x_i-x)^{p_i}\prod\limits_{i={k+1}}^{k+l}(x- x_i)^{p_i}\prod\limits_{1\leq i < j \leq k+l-1}(x_i-x_j)^{q_{ij}} \prod\limits_{i=1}^{k+l-1}(x_{k+l}-x_i)^{q_{i,k+l}}}
\end{align*}
in $W_2[[x_1, x_1^{-1}, ..., x_k, x_k^{-1}, x, x^{-1}, x_{k+1}, x_{k+1}^{-1}, ..., x_{k+l}, x_{k+l}^{-1}]]$, viewed as a series in \\
$(W_2[[x_{k+l}, x_{k+l}^{-1}]])[[x_1, x_1^{-1}, ..., x_k, x_k^{-1}, x, x^{-1}, x_{k+1}, x_{k+1}^{-1}, ..., x_{k+l-1}, x_{k+l-1}^{-1}]]$, has total degree at least as large as $N-\wt(u_1)-\cdots -\wt(u_{k+1})-\wt(v)$. 
\end{enumerate}
\end{rema}

\begin{prop}\label{submodule-quotient}
The category $\mathcal{C}_N$ is closed under the operation of taking submodules and quotients. 
\end{prop}
\begin{proof}
Let $T\subseteq W$ be a submodule of a left $V$-module $W\in \mathcal{C}_N$. Fix any $V$-submodule $T_2$ of $T$. Then $T_2$ is also a left $V$-submodule of $W$. Since $W\in \CC_N$, there exists a vector subspace $W_1$, such that $W=W_1 \amalg T_2$, and for every $v\in V$, the corresponding projection operator $\pi_{T_2}$ and $\pi_{W_1}$ makes the map $\pi_{T_2} Y_W(v, x)\pi_{W_1}$ infinitely composable with $Y_{T_2}$ and $\pi_{W_1}Y_W$ and satisfy the $N$-weight-degree condition. Let $T_1 = W_1 \cap T$. Then $T = T_1 \amalg T_2$. Moreover, for every $v\in V$, 
$$\pi_{T_2}Y_T(v, x) \pi_{T_1} = \pi_{T_2}Y_W(v, x)\pi_{W_1} | _{T_1}, $$
and 
$$\pi_{T_1}Y_T(v, x) = \pi_{W_1} Y_W(v, x)|_{T_1}.$$
Thus for every $k, l\in \N$, every $u_1, ..., u_{k+l}\in V$, every $t_2'\in W_2', t_1\in T_1$, 
\begin{align*}
        & \langle t_2', Y_{W_2}(u_1, z_1)\cdots Y_{W_2}(u_k, z_k)\pi_{W_2}Y_W(v,z) \pi_{W_1} Y_W(u_{k+1}, z_{k+1})\cdots \pi_{W_1}Y_W(u_{k+l}, z_{k+l})t_1\rangle\\
        & \quad = \langle t_2', Y_{T_2}(u_1, z_1)\cdots Y_{T_2}(u_k, z_k)\pi_{T_2}Y_T(v,z) \pi_{T_1} Y_T(u_{k+1}, z_{k+1})\cdots \pi_{T_1}Y_T(u_{k+l}, z_{k+l})t_1\rangle
\end{align*}
Since the left-hand-side series converges absolutely to a rational function satisfying the $N$-weight-degree condition, so is the right-hand-side series. Thus we verified that for every fixed $V$-submodule $T_2$ of $T$, there exists a vector subspace $T_1$ of $T$ such that $T=T_1\amalg T_2$, and for every $v\in V$, the corresponding projection operators $\pi_{T_2}$ and $\pi_{T_1}$ makes the map $\pi_{T_2}Y_T(v, x)\pi_{T_1}$ satisfy the composability and the $N$-weight-degree conditions. Thus $T\in \CC_N$. 

Let $T\subseteq W$ be a left $V$-submodule and consider the quotient module $W/T$. For every $v\in V$, the vertex operator on $W/T$ is given naturally by 
$$Y_{W/T}(v, x)(w+T) = Y_W(v, x)w + T.$$
Fix any submodule $W_2/T \subseteq W/T$. Then $W_2$ is a submodule of $W$. Since $W\in \CC_N$, there exists a vector subspace $W_1$ of $W$ such that $W = W_2\amalg W_1$, and for every $v\in V$, the corresponding projection operators $\pi_{W_2}$ and $\pi_{W_1}$ makes the map $\pi_{W_2}Y_{W}(v, x) \pi_{W_1}$ satisfy the composability (with $Y_{W_2}$ and $\pi_{W_1}Y_W$) and the $N$-weight-degree conditions. We use the vector space $(W_1+T)/T$, so that $W/T = (W_1+T)/T \amalg W_2/T$. Clearly, for every $w\in W$, from $T \cap W_1  = 0$, we have
$$\pi_{(W_1+T)/T}(w+T) = \pi_{W_1}w + T$$
Therefore, for every $v\in V$ and $w\in W$, the corresponding projection operators $\pi_{(W_1+T)/T}$ and $\pi_{W_2/T}$ satisfy
$$\pi_{W_2/T}Y_{W/T}(v, x) \pi_{(W_1+T)/T}(w+T)= \pi_{W_2}Y_{W}(v, x)\pi_{W_1}w + T,$$
and 
$$\pi_{(W_1+T)/T}Y_{W/T}(v, x)(w+T) = \pi_{W_1}Y_W(v,x)w + T. $$
These identifications allow us to pass the Conditions (3) and (4) in Remark \ref{composibility-formal} satisfied by the map $\pi_{W_2}Y_W(v,x)\pi_{W_1}$ from $W_2$ to $W_2/T$. Then, we see that the map $\pi_{W_2/T}Y_{W/T}(v, x) \pi_{W_1+T}$ also satisfies Conditions (3) and (4). Thus we proved that for every submodule $W_2/T$ of $W/T$, there exists a vector subspace $(W_1+T)/T$ of $W/T$, such that $W/T = (W_1+T)/T \oplus W_2/T$, and corresponding projections makes the map $\pi_{W_2/T}Y_{W/T}(v, x) \pi_{W_1+T}$ infinitely composable and satisfy the $N$-weight-degree condition. So $W/T\in \CC_N$. 
\end{proof}

\begin{rema}
In general, we cannot prove that $\mathcal{C}_N$ is closed under direct sums. Thus $\mathcal{C}_N$ generally does not form an abelian category. However, we will see in Section \ref{s-4-3} that if $V$ is a grading-restricted meromorphic open-string vertex algebra containing a nice vertex subalgebra, then $\CC_N$ coincides with the category of grading-restricted (generalized) $V$-modules. 
\end{rema}

\begin{rema}
Even if $V$ does not contain a nice subalgebra, it is still possible that certain $V$-modules form an abelian category whose objects are in $\CC_N$. In Section \ref{Section-5}, we will illustrate such an example using certain modules for the Virasoro VOA. 
\end{rema}

\subsection{When $V$ contains a nice subalgebra $V_0$}\label{s-4-3}

Let $V$ be a grading-restricted meromorphic open-string vertex algebra and $V_0$ be a vertex subalgebra of $V$. Note that the grading operator $\d_{V_0}$ and the operator $D_{V_0}$ for the vertex subalgebra $V_0$ of $V$ are induced from those for $V$. In particular, when $V$ is a vertex operator algebra such that $\d_V = L_V (0)$ and $D_V = L_V (-1)$ and $V_0$ is a vertex operator subalgebra with a different conformal vector, we insist that $\d_{V_0} = L_{V} (0)|_{V_0}$ and $D_{V_0} = L_V (-1)|_{V_0}$. In general, they may be different from $L_{V_0}(0)$ and $L_{V_0}(-1)$, respectively.

We start by interpreting Proposition 6.3, Theorem 6.4 in \cite{HQ-Red} to the following form. 

\begin{thm}\label{Huang-thm}
Let $V$ be a grading-restricted meromorphic open-string vertex algebra. Let $V_0$ be a vertex subalgebra of $V$ satisfying the following conditions: 
\begin{enumerate}
\item Every grading-restricted left $V_{0}$-module (or generalized $V_{0}$-module satisfying the grading-restriction conditions 
in the terminology in \cite{HLZ2}) is completely reducible. 
\item For any $n\in \Z_{+}$, products of $n$ intertwining operators among grading-restricted left $V_{0}$-modules evaluated at $z_{1}, \dots, z_{n}$ are absolutely convergent in the region $|z_{1}|>\cdots >|z_{n}|>0$ and can be analytically extended to a (possibly multivalued) analytic function in $z_{1}, \dots, z_{n}$ with the only possible singularities (branch points or poles) $z_{i}=0$ for $i=1, \dots, n$ and $z_{i}=z_{j}$ for $i, j=1, \dots, n$, $i\ne j$. 
\item The associativity of intertwining operators among grading-restricted left $V_{0}$-modules holds. 
\end{enumerate}
Let $U$ be an extension of $W_1$ by $W_2$. Then $U$ is in $\mathcal{C}_N$ for some $N$ determined by the associativity property of $V_0$-intertwining operators. \end{thm}

\begin{proof}
When $V$ is a grading-restricted vertex algebra, the proofs are given to Proposition 6.3 and Theorem 6.4 in \cite{HQ-Red}. We shall not repeat here. When $V$ is a grading-restricted meromorphic open-string vertex algebra, it suffices to notice that every $V$-module is automatically a $V_0$-module (since weak associativity holds, see \cite{LL}, Theorem 4.4.5). The proofs are basically identical to Proposition 6.3 and Theorem 6.4 in \cite{HQ-Red}. 
\end{proof}

\begin{rema}
The proof of Proposition 6.3 and Theorem 6.4 in \cite{HQ-Red} was discovered by Yi-Zhi Huang. The proof uses the theory of intertwining operators and the tensor category theory for module categories for grading-restricted vertex algebras satisfying suitable conditions. It hints deep connections between the cohomology theory and the theory of intertwining operators or the tensor category theory.
\end{rema}

\begin{rema}
It can also be seen from the proof that if $N$ is a number that is smaller than the lowest weight of every $V_0$-module, then every $V$-module is in $\CC_N$. 
\end{rema}

Recall that the fusion rule $N\binom{W_2}{VW_1}$ is the dimension of the space of $V_0$-intertwining operators of type $\binom{W_2}{VW_1}$. 

\begin{thm}\label{Ext-finite}
Let $V$ be a grading-restricted meromorphic open-string vertex algebra satisfying the conditions in Theorem \ref{Huang-thm}. Let $W_1, W_2$ be grading-restricted $V$-modules. Then $\dim_\C \Ext^1(W_1, W_2) \leq N\binom{W_2}{VW_1}$. In particular, if the fusion rule $N\binom{W_2}{VW_1}$ is finite, then $\Ext^1(W_1, W_2)$ is finite-dimensional. 
\end{thm}

\begin{proof}
Let $U\in \Ext^1(W_1, W_2)$. Then $U$ is a grading restricted $V_0$-module. From the assumption, 
$$0 \to W_2 \to U \to W_1 \to 0$$
splits as $V_0$-modules. We pick $\psi_1: W_1 \to U$ to be the splitting $V_0$-homomorphism. Thus $\pi_2: U \to W_2$ is also a $V_0$-homomorphism. Consider the 1-cohomology class given by the map $F: v\mapsto \pi_2 Y_U(v, \zeta)\psi_1$ in $H^1(V, \H_N(W_1, W_2))$. From Remark \ref{rmk-indep-section}, any other choices of $\psi_1$ results in the same cohomology class. 

Consider now the space $\mathcal{I}$ of $V_0$-intertwining operators of type $\binom{W_2}{VW_1}$. Let $\mathcal{I}_0$ be the subspace of $I$ consisting of $V_0$-intertwining operators of the form 
$$v\otimes w_1 \mapsto Y_{W_2}(v, \zeta) \psi w_1 - \psi Y_{W_1}(v, \zeta) w_1, \psi\in \Hom_{V_0}(W_1, W_2) $$
(the element is zero only when $v\in V_0$). Let $I = \mathcal{I}/\mathcal{I}_0$. Then $\dim_\C I \leq N\binom{W_2}{VW_1}$. Since both $\pi_2$ and $\psi_1$ are $V_0$-homomorphisms while $Y_U(v, \zeta)$ is an intertwining operator of type $\binom{U}{VU}$, the map 
$$\mathcal{Y}: v\otimes w_1 \mapsto \pi_2 Y_U(v, x)\psi_1w_1$$
is then an intertwining operator of type $\binom{W_2}{VW_1}$. 

Consider now the map $\Ext^1(W_1, W_2) \to I$ given by $U \mapsto [\mathcal{Y}]$. This map is well-defined because any $V$-extension equivalent to the trivial extension corresponds to an inner derivation $V\to \H_N(W_1, W_2)$, which then corresponds to an element in $\mathcal{I}_0$. We now show the map is injective. Assume $U$ has image zero in the space $I$. Then for every $v\in V$, $\pi_2 Y_U(v, z)\psi_1\in \mathcal{I_0}$. So the cohomology class $[F]$ given by the map $F: v\mapsto \pi_2 Y_U(v, z)\psi_1$ is zero in $\widehat{H}^1(V, \H_N(W_1, W_2))$. From the bijective correspondence proved in Theorem \ref{Main-Ext-1}, we see that $U$ is equivalent to the trivial extension $W_1\oplus W_2$. 

Therefore, $\Ext^1(W_1, W_2)$ embeds into $I$. Hence 
$$\dim_\C \Ext^1(W_1, W_2) \leq \dim_\C I \leq N\binom{W_2}{VW_1} < \infty. $$
\end{proof}

\begin{rema}
Though the inclusion of a nice vertex subalgebra $V_0$ is a very strong condition, there are no further requirements. In particular, we do not require $\dim V_{(0)} = 1$ (as in \cite{HKL}), or the conformal element of $V$ coincides with the conformal element of $V_{(0)}$. Indeed, $V$ is not even necessarily commutative. One immediate example of such noncommutative extensions of nice vertex algebras are vertex superalgebras (see, for example, \cite{CY}), which are indeed $\frac{\Z}2$-graded meromorphic open-string vertex algebras (see \cite{FQ} for further details), to which all results generalize. 
\end{rema}

\begin{rema}
Theorem \ref{Ext-finite} hints a program to compute $\Ext^1(W_1, W_2)$. We start by classifying the equivalence classes of intertwining operators of type $\binom{W_2}{VW_1}$ that are derivations from $V$ to $\H_N(W_1, W_2)$ modulo those intertwining operators that are inner derivations. Then from each representative of the equivalence class, we may construct an extension of $W_1$ by $W_2$ explicitly. The process requires the knowledge of $V_0$-intertwining operators, branching rules of $V$ and $V$-modules as $V_0$-modules, and correlation functions of related intertwining operators. 
One immediate example is the affine VOA $V=V_k(\g)$, $V_0=V_k(\h)$ where $\h$ is the Cartan subalgebra, and $W_1, W_2$ are weight modules. Initial attempts shows that although the fusion rule is infinite, the equation for the intertwining operator to form a derivation is very restrictive. Many obvious choices of $\mathcal{Y}$ do not satisfy equations. We believe that the program will help in studying the representation theory of affine VOA with nonadmissible levels.  
\end{rema}

\section{An example from the Virasoro VOA}\label{Section-5}

In this section, we demonstrate an example of an abelian category of modules for the Virasoro VOA satisfying the conditions Definition \ref{composibility}. Note that in this case, the VOA does not contain any nice subalgebra as in Section \ref{s-4-3}. This suggests that that the conditions in Definition \ref{composibility} might not be a serious obstruction for the application of cohomology theory developed in this paper.

\subsection{Review of Virasoro algebra and Feigin-Fuchs theory}
Let $Vir$ be the Virasoro algebra, i.e., $Vir = \bigoplus_{n\in \Z}\C L_n \oplus \C \mathbf{k}$, with 
$$[\mathbf{k}, L_n] = 0, [L_m, L_n] = (m-n)L_{m+n} + \delta_{m+n,0} \frac{m^3-m}{12} k, $$
for every $m,n\in \Z$. Fix $c,h\in \C$. Consider the one-dimensional space $\C\one_{c,h}$ where $\mathbf{k}$ acts by the scalar $c$, $L_0$ acts by the scalar $h$, and $L_n$ acts trivially for every $n>0$. The induced module
$$M(c,h) = U(Vir)\otimes_{\bigoplus_{n\geq 0}\C L_n \oplus \C \mathbf{k}} \C\one_{c,h}$$
is called the Verma module (aka, universal restricted module) associated with central charge $c$ and lowest weight $h$. For every $n\in \Z$, we will use the notation $L(n)$ to denote the action of $L_n$ on $M(c,h)$. It is known that $M(c,h)$ is graded by the $L(0)$-eigenvalues. For each $n\in \N$, the eigenspace $M(c,h)_{(h+n)}$ has a basis of vectors of the following form
\begin{align}
    L(-n_1)\cdots L(-n_s)\one_{c,h}, n_1 \geq \cdots \geq n_s \geq -1, n_1 + \cdots + n_s = n. \label{Verma-basis}
\end{align}
It is also known that $M(c,h)$ has a unique irreducible quotient, denoted by $L(c,h)$. Define a partial order among the pairs $(c,h)$: 
$$(c_1, h_1)\preceq (c_2, h_2) \text{ iff } L(c_2, h_2) \text{ is a subquotient of }M(c_1, h_1).$$
We use $\sim$ to denote the equivalence relation given by transitive closure of $\preceq$, i.e., 
$(c_1, h_1) \sim (c_2, h_2)$ iff there exists $(c^{(1)}, h^{(1)}), ..., (c^{(n)},h^{(n)})$ such that $(c^{(1)}, h^{(1)}) = (c_1, h_1), (c^{(n)}, h^{(n)})=(c_2, h_2)$, and for every $i=1, ..., n-1$, either $(c^{(i)}, h^{(i)})\preceq (c^{(i+1)}, h^{(i+1)})$ or $(c^{(i+1)}, h^{(i+1)})\preceq (c^{(i)}, h^{(i)})$. The equivalence class of $(c,h)$ is called the block, denoted by $[c,h]$. 

Here we list all the previously known results regarding the Verma modules, their submodules, and structure of blocks (see \cite{Ash}, \cite{BNW}, \cite{IK}, \cite{FF}, \cite{MP}). For simplicity, we follow the notation and classification in \cite{BNW}. 
\begin{enumerate}
    \item For any $c,h\in \C$, the submodules of $M(c,h)$ is generated by at most two singular vectors, i.e., vectors $w\in M(c,h)$ with $L(n) w = 0$ for every $n>0$. \item For any $c, h\in \C$ and any $n\in \N$, there exists at most one singular vector in the homogeneous subspace $M(c,h)_{(h+n)}$.  
    \item The submodule generated by a singular vector of weight $h'$ is isomorphic to $M(c,h')$. 
    \item For fixed $c, h\in \C$, let
    $$\nu = \frac{c-13+\sqrt{(c-1)(c-25)}}{12}, \beta = \sqrt{-4\nu h +(\nu+1)^2}$$
    and the line in the $rs$-plane
    $$\mathcal{L}_{c,h}: r + \nu s + \beta = 0.$$
    \begin{enumerate}
        \item Suppose $\mathcal{L}_{c,h}$ passes through no integer points or one integer point $(r,s)$ with $rs=0$, then $M(c,h)$ is irreducible. The block $[c,h] = \{(c,h)\}$. 
        \item Suppose $\mathcal{L}_{c,h}$ passes through no integer points or one integer point $(r,s)$ with $rs\neq 0$. Then the block $[c,h]$ is given by $\{(c,h), (c,h+rs)\}$. 
        \item Suppose $\mathcal{L}_{c,h}$ passes through infinitely many integer points and crosses an axis at an
        integer point. Label these points by $(r_i, s_i)$ so that
        $$ \cdots < r_{-2}s_{-2} < r_{-1}s_{-1} < 0 = r_0s_0 < r_1s_1 < r_2s_2 <\cdots$$
        The block $[c,h]$ is given by $[c,h]=\{(c, h+r_is_i): i \in \Z\}$. 
        \item Suppose $\mathcal{L}_{c,h}$ passes through infinitely many integer points and does not cross either axis at an
        integer point. Label these points by $(r_i, s_i)$ so that
        $$ \cdots < r_{-2}s_{-2} < r_{-1}s_{-1} < r_0s_0 < 0 < r_1s_1 < r_2s_2 <\cdots$$
        Also consider the auxiliary line $\tilde{\mathcal{L}}_{c,h}$ with the same slope as $L_{c,h}$ passing through the point $(-r_1, s_1)$. Label the integer points on $\tilde{\mathcal{L}}_{c,h}$ by $(r_j' , s_j')$ in the same way as $\mathcal{L}_{c,h}$. The block $[c,h]$ is given by $[(c,h)] = \{(c,h_i), (c,h_i'): i\in \Z\}$, where
        $$h_i = \left\{\begin{array}{ll} 
        h + r_is_i & i \text{ odd}\\
        h + r_1s_1 + r_i's_i' & i \text{ even}\end{array}\right.,h_i' = \left\{\begin{array}{ll} 
        h + r_{i+1}s_{i+1} & i \text{ odd}\\
        h + r_1s_1 + r_{i+1}'s_{i+1}' & i \text{ even}\end{array}\right. $$
    \end{enumerate}
    \item If $[c_1, h_1]\neq [c_2, h_2]$, then Ext$^1(M(c_1, h_1), M(c_2, h_2)) = 0$. 
\end{enumerate}
Let $V=V(c, 0)$ be the quotient of $M(c,0)$ by the submodule generated by $L(-1)\one_{c,0}$. It is known that $V(c,0)$ is a vertex operator algebra, and every $M(c,h)$ is a $V$-module. Let $W$ be a $V$-module which is a subquotient of a finite direct sum of Verma modules $M(c,h)$ appearing in the blocks described by (a), (b) and (c). The goal of this section is to show that $W$ satisfies Definition \ref{composibility}. Clearly, such modules form an abelian category. Thus the objects of the category $\CC$ formed by such modules are all in $\CC_N$.



\subsection{Verma module case}\label{Section-5-2}
We first study the simplest case where $W$ is the Verma module $M(c,h)$ appearing in the blocks (a), (b) and (c). We would omit the lowest weight vector and write 
$$W = span \{L(-n_1)\cdots L(-n_s) | n_1, ..., n_s\in \Z_+, n_1 \geq \cdots \geq n_s \geq 1\}. $$
Note that the vectors in the spanning set are linearly independent. 
For the vector $L(-n_1)\cdots L(-n_s)$, the number $\sum_{i=1}^s n_i$ is called the level. 

Every submodule of $M(c,h)$ is generated by singular vectors. We know from \cite{Ash} that for each $N\in \Z_+$, there exists at most one singular vector of level $N$ (up to a scalar), and the coefficient of $L(-1)^N$ is nonzero. 

Let $W_2 = M(c, h_1)$ be a submodule. By our choice of $W$, $W_2$ is generated by one singular vector
$$s_{h_1} = L(-1)^N + \sum_{\substack{L<N,M_1 \geq ...\geq M_L\geq 1 \\M_1+\cdots +M_L = N}}a_{M_1...M_L} L(-M_1)\cdots L(-M_L)$$
of some fixed level $N$. Then $h_1 = h + N$, and  
$$W_2 = span \{L(-n_1)\cdots L(-n_s)L(-1)^{n-N}s_{h_1} | n_1, ..., n_s\in \Z_+, n_1 \geq \cdots \geq n_s \geq 2, n\geq N\}. $$
Note also that the vectors in the spanning set are linearly independent. We now choose 
$$W_1 = span\{L(-n_1)\cdots L(-n_s) L(-1)^n| n_1, ..., n_s\in \Z_+, n_1\geq \cdots \geq n_s \geq 2, n < N\}.$$
Clearly, $W = W_1 \coprod W_2$ as a vector space.
Let $\pi_{W_2}: W \to W_2$ be the projection of $W$ onto $W_2$, i.e., 
$$\pi_{W_2}|_{W_2} = Id_{W_2}, \pi_{W_2}|_{W_1} = 0. $$
Then, for $n_1\geq \cdots \geq n_s \geq 2$, we have $$\pi_{W_2} (L(-n_1)\cdots L(-n_s)L(-1)^N) = L(-n_1)\cdots L(-n_s) s_{h_1}. $$
Moreover, 
$$\pi_{W_2} (L(-n_1)\cdots L(-n_s)L(-1)^n) \neq 0 $$
only when $n\geq N$.

\begin{defn}
Every element $w\in W$ can be written uniquely as
$$w = \sum_{p_1 \geq \cdots \geq p_t \geq 2} \sum_{q\geq 0}a_{p_1\cdots p_t, q}L(-p_1)\cdots L(-p_t) L(-1)^q. $$
We will refer this summation as the standard expression of $w$. 
Clearly the sum is finite, i.e., there exists finitely many $p_1, ..., p_t$ and $q$, such that $a_{p_1\cdots p_t, q}\neq 0$. We define the index of $w$ as
$$\Ind(w) = \max \{q|a_{p_1\cdots p_t, q} \neq 0, p_1 \geq \cdots \geq p_t \geq 2\}. $$
If $w = 0$, we define $\Ind(w) = -1$. Note that $\Ind(w_1+w_2) \leq \max\{\Ind(w_1), \Ind(w_2)\}. $
Note also that 
if $\Ind(w) < N$, then $\pi_{W_2}(w) = 0$. 
\end{defn}

\begin{prop}\label{Ind-Prop}
\begin{enumerate}
    \item For every $m\geq 2$, 
    $$\Ind(L(-m) w) = \Ind(w). $$
    \item Fix $n \in \Z_+$ and $n_1, ..., n_s \geq 2$ (not necessarily decreasing). Then
    $$\Ind(L(m)L(-n_1)\cdots L(-n_s) L(-1)^n) \leq n+1. $$
    The equality holds only for finitely many $m$. Moreover, if we restrict that $m\geq -1$, then only finitely many $m$ makes
    $$\Ind(L(m)L(-n_1)\cdots L(-n_s) L(-1)^n) \geq n.$$
\end{enumerate}
\end{prop}
\begin{proof}
These follow easily from the commutator relation 
$$L(m)L(n) - L(n)L(m) = (m-n)L(m+n) + \delta_{m+n,0}\frac{m^3-m}{12}, m,n\in \Z.$$
\end{proof}

\begin{prop}\label{pi-L(-m)-comm}
For every $w\in W$ and every integer $m\geq 2$, 
$$\pi_{W_2}(L(-m)w) = L(-m)\pi_{W_2}(w). $$
\end{prop}

\begin{proof}
It suffices to consider 
$$w = L(-n_1)\cdots L(-n_s)L(-1)^n$$
with $n_1\geq \cdots \geq n_s\geq 2$ and $n\geq N$ (when $n<N$ the conclusion holds trivially). Clearly, in the universal enveloping algebra of the Virasoro algebra, the standard expression of $L(-m)L(-n_1)\cdots L(-n_s)$ does not have any summands that contains factors of $L(-1)$. In other words, 
$$L(-m)L(-n_1)\cdots L(-n_s) = \sum_{t\geq 0, p_1 \geq \cdots \geq p_t \geq 2}b_{p_1\cdots p_t}L(-p_1)\cdots L(-p_s). $$
Now we argue by induction. For $n=N$, we have
\begin{align*}
    & \pi_{W_2} (L(-m)L(-n_1)\cdots L(-n_s) L(-1)^N)\\ =& \pi_{W_2}\left( \sum_{t\geq 0, p_1 \geq \cdots \geq p_t \geq 2}b_{p_1\cdots p_t}L(-p_1)\cdots L(-p_s)L(-1)^N \right)\\
    =& \sum_{t\geq 0, p_1 \geq \cdots \geq p_t \geq 2}b_{p_1\cdots p_t}L(-p_1)\cdots L(-p_s)s_{h_1} \\
    =& L(-m)L(-n_1)\cdots L(-n_s) s_{h_1} \\
    =& L(-m)\pi_{W_2}(L(-n_1)\cdots L(-n_s) L(-1)^N),
\end{align*}
where the second and fourth equal signs follow from the definition of $\pi_{W_2}$, the third equal sign follows from the equality in the universal enveloping algebra. 
For general $n$, we have 
\begin{align*}
    & \pi_{W_2} (L(-m)L(-n_1)\cdots L(-n_s) L(-1)^n)\\ 
    =& \pi_{W_2}\left( L(-m)L(-n_1)\cdots L(-n_s)L(-1)^{n-N}s_{h_1} \right)\\
    & -  \pi_{W_2}\left( L(-m)L(-n_1)\cdots L(-n_s)L(-1)^{n-N}\sum_{\substack{L\leq N, M_1 \geq \cdots \geq M_L \geq 2\\ M_1 + \cdots + M_L = N}}a_{M_1\cdots M_L}L(-M_1)\cdots L(-M_L) \right)
\end{align*}
From the definition of $\pi_{W_2}$, the first term is precisely 
$$L(-m)L(-n_1)\cdots L(-n_s) L(-1)^{n-N}s_{h_1} = L(-m)\pi_{W_2}(L(-n_1)\cdots L(-n_s)L(-1)^{n-N}L(-1)^N).$$ For the second sum, since the standard expression of 
$$L(-n_1)\cdots L(-n_2)L(-1)^{n-N}L(-M_1)\cdots L(-M_L)$$
for each fixed $M_1, ..., M_L$ has index strictly smaller than $n$, the induction hypothesis applies, showing that 
\begin{align*}
    & \pi_{W_2}(L(-m)L(-n_1)\cdots L(-n_2)L(-1)^{n-N}L(-M_1)\cdots L(-M_L))\\
    =& L(-m)\pi_{W_2}(L(-n_1)\cdots L(-n_2)L(-1)^{n-N}L(-M_1)\cdots L(-M_L))
\end{align*}
for every $M_1, ..., M_L$ appearing in the sum. The conclusion then follows. 
\end{proof}

\begin{rema}
In general we do not have $\pi_{W_2}(L(m)w) = L(m)\pi_{W_2}(w)$ for $m\geq -1$. The simplest counter-example is $m=-1, w=L(-1)^{N-1}$. Left-hand-side is nonzero while the right-hand-side is zero. 
\end{rema}

\begin{prop}\label{base-case}
Fix $w_1\in W_1$. Then 
\begin{align*}
    \pi_{W_2} Y_W(\omega, x)w_1\in W_2[x, x^{-1}].
\end{align*}
Moreover, $\pi_{W_2} Y_W(\omega, x)w_1\neq 0$ only when $\Ind(w_1) = N-1$. 
\end{prop}

\begin{proof}
We focus on a basis vector $$w_1 = L(-n_1)\cdots L(-n_s)L(-1)^{n}, n_1\geq \cdots \geq n_s \geq 2, n < N.$$
Write
\begin{align*}
    \pi_{W_2} Y_W(\omega, x)w_1 &= \pi_{W_2} Y_{W}^-(\omega, x)w_1 + \pi_{W_2} Y_W^+(\omega, x)w_1
\end{align*}
where $Y_W^+(\omega, x) = \sum\limits_{m\leq -2}L(m)x^{-m-2}$ is the regular part of $Y_{W}(\omega, x)$, $Y_W^-(\omega, x) = \sum\limits_{m\geq -1}L(m)x^{-m-2}$ is the singular part of $Y_W(\omega, x)$. From Proposition \ref{pi-L(-m)-comm}, we know that 
$$\pi_{W_2}Y_W^+(\omega, x)w_1 = Y_W^+(\omega, x)\pi_{W_2}w_1 = 0.$$
Therefore, 
\begin{align}
    \pi_{W_2}Y_{W}(\omega, x)w_1 = \pi_{W_2}Y_W^-(\omega, x)w_1 = \sum_{m\geq -1} \pi_{W_2} L(m)L(-n_1)\cdots L(-n_s)L(-1)^{n}x^{-m-2} \label{5-1}
\end{align}
is a finite sum in $W_2[x, x^{-1}]$. Moreover, from Proposition \ref{Ind-Prop}, if $n< N-1$, then the indices of $L(m)L(-n_1)\cdots L(-n_s)L(-1)^n$ appearing in (\ref{5-1}) are all strictly less than $N$. So in this case, $\pi_{W_2}Y_{W}(\omega, x)w_1 = 0$. Therefore, (\ref{5-1}) is nonzero only when $\Ind(w_1) = N-1$. 
\end{proof}

\begin{defn}
Let $l\in \Z_+$ and $p\in \{1, ..., l\}$. We call $\alpha_1, ..., \alpha_p$ an increasing sequence in $\{1, ..., l\}$ if $1\leq \alpha_1 < \cdots < \alpha_p \leq l$. The number $p$ is called the length of the sequence $\alpha_1, ..., \alpha_p$. Sometimes we will also use the notation $(\alpha_1, ..., \alpha_p)$ to denote an increasing sequence. For a fixed increasing sequence $\alpha_1, ..., \alpha_p$, there exists a unique increasing sequence $\alpha^c_1, ..., \alpha^c_{l-p}$ such that $\{\alpha_1, ..., \alpha_p\}\coprod\{\alpha^c_1, ..., \alpha^c_{l-p}\} = \{1, ..., l\}$. We call the increasing sequence $\alpha^c_1, ..., \alpha^c_{l-p}$ the complement of $\alpha_1, ..., \alpha_p$. Note that when $p = l$, then an increasing sequence is uniquely $1, ..., l$ whose complement is empty. We call $(\alpha_1, ..., \alpha_p)$ and $(\alpha_1^c, ..., \alpha_{l-p}^c)$ two complementary increasing sequences in $\{1, ..., l\}$. 
\end{defn}

\begin{thm}\label{general-l-lemma}
For every $l\in \Z_+$ and every $w\in W$, the series 
$$ Y_{W}(\omega, x_1) \cdots Y_W(\omega, x_l)w$$
is a linear combination of series of the form 
\begin{align}
    Y_W^+(v_1, x_{\alpha_1^c})\cdots Y_W^+(v_{l-p}, x_{\alpha_{l-p}^c})\iota_{1...l}\left(\frac{g(x_{\alpha_1}, ..., x_{\alpha_p})}{\prod\limits_{1\leq i<j\leq l}(x_i-x_j)^{P_{ij}}}\right), \label{5-5}
\end{align}
where $p\in \N, v_1, ..., v_{l-p}\in V$ falls in the subspace spanned by $\one$ and $L(-1)^q\omega, q\in \N$ , $g(x_{\alpha_1}, ..., x_{\alpha_{p}}) \in W[x_{\alpha_1}, x_{\alpha_1}^{-1}, ..., x_{\alpha_p}, x_{\alpha_{p}}^{-1}]$ is a Laurent polynomial with coefficient in $W_2$, $P_{ij}$ are nonnegative integers satisfying $\sum\limits_{1\leq i < j \leq l}P_{ij}\leq 2p$.

\end{thm}

\begin{proof}
The idea of the following proof is suggested by Reviewer 1. We will proceed with induction on $l$. For the base case $l=1$, we write 
$$Y_W(\omega, x_1) w = Y_W^+(\omega, x_1) w + Y_W^-(\omega, x_1)w$$
Regarding $w$ in the first term as a constant polynomal, and note that $Y_W^-(\omega, x_1)w$ in the second term is a Laurent polynomial, we see the conclusion for $l=1$. 

To argue the inductive step, we recall (\cite{LL}, Formula 3.8.7) that for every $u, v$ in a vertex algebra, 
$$[Y_W^-(u, x_0+x_2), Y_W(v, x_2)] = Y_W(Y^-(u, x_0)v, x_2)$$
as formal series in $x_0, x_2$. Since $Y^-(u, x_0)v$ consists of finitely many powers of $x_0$, the series $Y_W(Y^-(u, x_1-x_2)v, x_2)$ exists as a formal series in $x_1, x_2$, and is equal to $[Y_W^-(u, x_1), Y_W(v, x_2)]$. Thus for $u=v=\omega$, 
\begin{align}
    & [Y_W^-(\omega, x_1), Y_W(\omega, x_2)]
    = Y_W(Y^-(\omega, x_1-x_2)\omega, x_2)=\sum_{j\geq 0} Y_W(\omega_j \omega, x_2)\iota_{12}(x_1-x_2)^{-j-1}\nonumber\\
    = \ &Y_W(L(-1)\omega, x_2)\cdot \iota_{12}(x_1-x_2)^{-1} + 2 Y_W(\omega, x_2)\cdot \iota_{12}(x_1-x_2)^{-2} + \frac 1 2 c_V Y_W(\one, x_2)\cdot \iota_{12} (x_1-x_2)^{-4}\label{5-6}
\end{align} 


Now we assume the statement holds for all smaller $l$ and show the inductive step. Using formula (\ref{5-6}), we rewrite 
\begin{align}
    & Y_W(\omega, x_1)\cdots Y_W(\omega, x_l)w \nonumber\\
    = \ & Y_W^+(\omega, x_1)Y_W(\omega, x_2)\cdots Y_W(\omega, x_l)w \label{5-2}\\
    & + \sum_{k=2}^l Y_W(\omega, x_2) \cdots [Y_W^-(\omega, x_1), Y_W(\omega, x_k)] \cdots Y_W(\omega, x_l)w \label{5-3}\\
    & + Y_W(\omega, x_2)\cdots Y_W(\omega, x_l)Y_W^-(\omega, x_1)w \label{5-4}
\end{align}
We proceed to show that (\ref{5-2}), (\ref{5-3}) and (\ref{5-4}) are linear combinations of terms of the same form as (\ref{5-5}). 

For (\ref{5-2}), by the induction hypothesis, (\ref{5-2}) is a linear combination of 
\begin{align*}
    Y_W^+(\omega, x_1) Y_W^+(v_1, x_{\alpha_1^c})\cdots Y_W^+(v_{l-1-p}, x_{\alpha_{l-1-p}^c}) \iota_{1\cdots l}\left(\frac{g(x_{\alpha_1}, ..., x_{\alpha_p})}{\prod\limits_{2\leq i < j \leq l}(x_i - x_j)^{P_{ij}}}\right).
\end{align*}
Here $(\alpha_1, ..., \alpha_p)$ and $(\alpha_1^c, ..., \alpha_{l-1-p}^c)$ are complementary increasing sequences in $\{2, ..., l\}$. Thus the conclusion holds for (\ref{5-2}). Moreover, for each summand in the linear combination, the sum of exponents of $x_i - x_j$ $(1 \leq i < j \leq l)$ in the denominator is less or equal to $2p$. 

For (\ref{5-3}), we analyze the summand for each fixed $k$. Using (\ref{5-6}), the summand can be rewritten as 
\begin{align}
    & Y_W(\omega, x_2)\cdots Y_W(\omega, x_{k-1})Y_W(L(-1)\omega, x_k)Y_W(\omega, x_{k+1})\cdots Y_W(\omega, x_l)w\cdot  \iota_{1k}(x_1-x_k)^{-1}\label{5-7}\\
    & + 2 Y_W(\omega, x_2)\cdots Y_W(\omega, x_{k-1})Y_W(\omega, x_k)Y_W(\omega, x_{k+1})\cdots Y_W(\omega, x_l)w\cdot  \iota_{1k}(x_1-x_k)^{-2}\label{5-8}\\
    & + \frac 1 2 c_V Y_W(\omega, x_2)\cdots Y_W(\omega, x_{k-1})Y_W(\omega, x_{k+1})\cdots Y_W(\omega, x_l)w \cdot \iota_{1k}(x_1-x_k)^{-4}\label{5-9}
\end{align}
Likewise, we argue that (\ref{5-7}), (\ref{5-8}) and (\ref{5-9}) are linear combinations of terms of the form (\ref{5-5})
\begin{itemize}
    \item For (\ref{5-7}), using the $L(-1)$-derivative property and the induction hypothesis, we see that (\ref{5-7}) is a linear combination of 
\begin{align*}
    \iota_{1k}(x_1-x_k)^{-1}\frac{\partial}{\partial x_k}\left(Y_W^+(v_1, x_{\alpha_1^c})\cdots Y_W^+(v_{l-1-p}, x_{\alpha_{l-1-p}^c})\iota_{2\cdots l}\left(\frac{g(x_{\alpha_1}, ..., x_{\alpha_p})}{\prod\limits_{2\leq i < j \leq l}(x_i-x_j)^{P_{ij}}}\right)\right) 
\end{align*}
Here $(\alpha_1, ..., \alpha_p)$ and $(\alpha_1^c, ..., \alpha_{l-1-p}^c)$ are complementary increasing sequences in $\{2, ..., l\}$, and $\sum\limits_{2\leq i < j \leq l}P_{ij} \leq 2(p-1)$. If $k$ appears in the sequence $(\alpha_1, ..., \alpha_p)$, then $\partial / \partial_k$ passes to the fraction, resulting in a finite sum of fractions whose denominators have their exponents added up to $\sum\limits_{2\leq i < j \leq l}P_{ij} + 1 \leq 2p + 1$. Otherwise, if $k$ appears in the sequence $\alpha_1^c, ..., \alpha_{l-1-p}^c$, say $\alpha_i^c = k$, then it suffices to add an extra term with $v_i$ modified into $L(-1)v_i$. Then we may combine $\iota_{1k}(x_1-x_k)^{-1}$ into the fraction to get an expression of the form (\ref{5-5}), where the sum of the exponents in the denominator is bounded above by $2p+1+1 = 2(p+1)$, where $p+1$ is precisely the updated length of the increasing sequence (by adding $1$). This concludes the discussion of (\ref{5-7}). 
\item For (\ref{5-8}), by induction hypothesis (ignoring the coefficient 2), we see that (\ref{5-8}) is a linear combination of 
\begin{align*}
    \iota_{1k}(x_1-x_k)^{-2}\cdot Y_W^+(v_1, x_{\alpha_1^c})\cdots Y_W^+(v_{l-1-p}, x_{\alpha_{l-1-p}^c})\iota_{2\cdots l}\left(\frac{g(x_{\alpha_1}, ..., x_{\alpha_p})}{\prod\limits_{2\leq i < j \leq l}(x_i-x_j)^{P_{ij}}}\right)
\end{align*}
Here $(\alpha_1, ..., \alpha_p)$ and $(\alpha_1^c, ..., \alpha_{l-1-p}^c)$ are complementary increasing sequences in $\{2, ..., l\}$, and $\sum\limits_{2\leq i < j \leq l}P_{ij} \leq 2p$. 
Combining the $\iota_{1k}(x_1-x_k)^{-2}$ into the fraction, we get an expression of the form (\ref{5-5}), where the sum of the exponents in the denominator is $\sum\limits_{2\leq i < j \leq l}P_{ij}+2 \leq 2p+2 = 2(p+1)$, where $p+1$ is precisely the updated length of the increasing sequence (by adding $1$).
\item For (\ref{5-9}), by induction hypothesis (ignoring the coefficient $c_V/2$), we see that (\ref{5-9}) is a linear combination of 
\begin{align*}
    \iota_{1k}(x_1-x_k)^{-4}\cdot Y_W^+(v_1, x_{\alpha_1^c})\cdots Y_W^+(v_{l-2-p}, x_{\alpha_{l-2-p}^c})\iota_{2\cdots l}\left(\frac{g(x_{\alpha_1}, ..., x_{\alpha_p})}{\prod\limits_{\substack{2\leq i < j \leq l\\i\neq k, j\neq k}}(x_i-x_j)^{P_{ij}}}\right)
\end{align*}
Here $(\alpha_1, ..., \alpha_p)$ and $(\alpha_1^c, ..., \alpha_{l-1-p}^c)$ are complementary increasing sequences in $\{2, ..., k-1, k+1, ..., l\}$, and $\sum\limits_{\substack{2\leq i < j \leq l\\i\neq k, j\neq k}}P_{ij} \leq 2p$. 
Combining the $\iota_{1k}(x_1-x_k)^{-4}$ into the fraction, we get an expression of the form (\ref{5-5}), where the sum of the exponents in the denominator is $\sum\limits_{\substack{2\leq i < j \leq l\\i\neq k, j\neq k}}P_{ij}+4 \leq 2p+4 = 2(p+2)$ where $p+2$ is the updated length of increasing sequence (by adding $1$ and $k$).
\end{itemize}
This concludes the discussion of (\ref{5-3}). 

For (\ref{5-4}), notice that it is a linear combination of 
$$Y_W(\omega, x_2)\cdots Y_W(\omega, x_l)L(m)w x_1^{-m-2}. $$
The induction hypothesis shows that it is a linear combination of 
$$Y_W^+(v_1, x_{\alpha_1^c})\cdots Y_W^+(v_{l-1-p}, x_{\alpha_{l-1-p}^c}) \iota_{1\cdots l} \left(\frac{g(x_{\alpha_1}, ..., x_{\alpha_p})}{\prod\limits_{2\leq i < j \leq l}(x_i-x_j)^{P_{ij}}}\right) x_1^{-m-2}$$
with $\sum\limits_{2\leq i < j \leq n}P_{ij} \leq 2p < 2(p+1)$ (here $x_1^{-m-2}$ should be incorporated into $g$). This conclude the proof of the theorem. 
\end{proof}

\begin{thm}\label{general-l-thm}
    For every $l\in \Z_+$ and every $w_1\in W_1$, the series 
    $$ \pi_{W_2}Y_{W}(\omega, x_1) \cdots Y_W(\omega, x_l)w_1$$
    is a linear combination of series of the form 
    \begin{align}
        Y_W^+(v_1, x_{\alpha_1^c})\cdots Y_W^+(v_{l-p}, x_{\alpha_{l-p}^c})\iota_{1...l}\left(\frac{g(x_{\alpha_1}, ..., x_{\alpha_p})}{\prod\limits_{1\leq i<j\leq l}(x_i-x_j)^{P_{ij}}}\right), \label{general-l-thm-1}
    \end{align}
    where $p\in \Z_+, v_1, ..., v_{l-p}\in V$ falls in the subspace spanned by $\one$ and $L(-1)^q\omega, q\in \N$ , $g(x_{\alpha_1}, ..., x_{\alpha_{p}}) \in W[x_{\alpha_1}, x_{\alpha_1}^{-1}, ..., x_{\alpha_p}, x_{\alpha_{p}}^{-1}]$ is a Laurent polynomial with coefficient in $W_2$, $P_{ij}$ are nonnegative integers satisfying $\sum\limits_{1\leq i < j \leq l}P_{ij}\leq 2(p-1)$.         
\end{thm}

\begin{proof}
    Note that from Propositoin \ref{pi-L(-m)-comm}, $\pi_{W_2}$ commutes with $Y_W^+(\omega, x)$. We then apply $\pi_{W_2}$ directly to the conclusion of Theorem \ref{general-l-lemma} to conclude the first part (where $f(x_{\alpha_1}, ..., x_{\alpha_p}) = \pi_{W_2} g(x_{\alpha_1}, ..., x_{\alpha_p})$). Note that since $\pi_2 w_1=0$, there does not exist any nonzero term with $l$ number of $Y^+$ operators. So necessarily, $p>0$. 
    
    Regarding the estimate of $\sum\limits_{1\leq i < j \leq l}P_{ij}$, note that when $l = 2$, we see from the following computation 
    \begin{align*}
        & \pi_{W_2} Y_W(\omega, x_1) Y_W(\omega, x_2) w_1 \\
        = \ &  \pi_{W_2} Y_W^+(\omega, x_1)Y_W(\omega, x_2)w_1 + \pi_{W_2} [Y_W^-(\omega, x_1), Y_W(\omega, x_2)] w_1 + \pi_{W_2}Y_W(\omega, x_2)Y_W^-(\omega, x_1)w_1\\
        = \ & Y_W^+(\omega, x_1)\pi_{W_2}Y_W^-(\omega, x_2)w_1 + \pi_{W_2} Y_W(L(-1)\omega, x_2)w_1 \cdot \iota_{12}(x_1-x_2)^{-1} \\
        & + \pi_{W_2} Y_W(2\omega, x_2)w_1 \cdot \iota_{12}(x_1-x_2)^{-2}  + \frac{c_V} 2\pi_{W_2} Y_W(\one, x_2)w_1 \cdot \iota_{12}(x_1-x_2)^{-4}\\
        & + Y_W^+(\omega, x_2)\pi_{W_2}Y_W^-(\omega, x_1)w_1 + \pi_{W_2}Y_W^-(\omega, x_2)Y_W^-(\omega, x_1)w_1\\
        = \ & Y_W^+(\omega, x_1)\pi_{W_2}Y_W^-(\omega, x_2)w_1 + \pi_{W_2} Y_W^-(L(-1)\omega, x_2)w_1 \cdot \iota_{12}(x_1-x_2)^{-1} \\
        & + \pi_{W_2} Y_W^-(2\omega, x_2)w_1 \cdot \iota_{12}(x_1-x_2)^{-2}  + 0 + Y_W^+(\omega, x_2)\pi_{W_2}Y_W^-(\omega, x_1)w_1 + \pi_{W_2}Y_W^-(\omega, x_2)Y_W^-(\omega, x_1)w_1
    \end{align*}
    that if there exists one $Y^+$ operator, then the exponent of $(x_1-x_2)^{-1}$ is at most $0 = 2 \cdot (1-1)$; if there exists no $Y^+$ operators, then the exponent of $(x_1-x_2)^{-1}$ is at most $2 = 2 \cdot (2-1)$. Thus we may conclude that    
    $\sum\limits_{1\leq i < j \leq 2}P_{ij}$ appearing in each term of form (\ref{general-l-thm-1}) is at most $2(p-1)$ when $p=2$. Repeating the inductive step as in Theorem \ref{general-l-lemma}, we see that generally, $\sum\limits_{1\leq i < j \leq l}P_{ij} \leq 2(p-1)$. 
\end{proof}

\begin{lemma}\label{negative-removal}
    Assume that the series
    \begin{align}
        \sum_{i_1, ..., i_n\in \Z} a_{i_1\cdots i_n}z_1^{i_1}\cdots z_n^{i_n}\label{negative-removal-1}
    \end{align}
    converges absolutely in the region $|z_1|>\cdots > |z_n|> 0$ to a rational function that is of the form 
    \begin{align}
        \frac{f(z_1, ..., z_n)}{\prod\limits_{i=1}^n z_i^{P_i} \prod\limits_{1\leq i < j \leq n}(z_i-z_j)^{P_{ij}}} \label{negative-removal-2}
    \end{align}
    where $f(z_1, ..., z_n)$ is a Laurent polynomial in $z_1, ..., z_n$. Let $M = \sum\limits_{1\leq i < j \leq n}P_{ij}$. Then removing all the negative powers of $z_m, z_{m+1}, ...., z_n$ in (\ref{negative-removal-1}) results in a series that converges absolutely in the same region to a linear combination of rational functions of the form 
   \begin{align}
        \frac{g(z_1, ..., z_n)}{\prod\limits_{i=1}^{m-1} z_i^{Q_i} \prod\limits_{1\leq i < j \leq m-1}(z_i-z_j)^{Q_{ij}}\prod\limits_{\substack{1\leq i \leq m-1\\m\leq  j \leq n}}(z_i-z_j)^{Q_{ij}}} \label{negative-removal-2T}
    \end{align}
    with $\sum\limits_{1\leq i < j \leq n}Q_{ij}\leq M$. 
\end{lemma}

\begin{proof}
    We argue by induction. For the base case $n=2$, we first consider the case when the series and rational function is of the form 
    \begin{align}
        \frac{1}{z_1^pz_2^q(z_1-z_2)^r} & = \frac{1}{z_1^p}\frac{1}{(r-1)!}\frac{\partial^{r-1}}{\partial z_1^{r-1}}\left(\frac 1 {z_2^q (z_1-z_2)}\right) = \sum_{i=0}^\infty z_1^{-p} \left(\frac{1}{(r-1)!}\frac{\partial^{r-1}}{\partial z_1^{r-1}}z_1^{-1-i}\right)z_2^{i-q}\label{neg-remov-1}
    \end{align}
    We first remove the negative powers of $z_2$ from (\ref{neg-remov-1}). In case $q\leq 0$, (\ref{neg-remov-1}) contains only nonnegative powers of $z_2$. So there exists no change in $r$. In case $q>0$, we obtain the series
    \begin{align}
        \sum_{i=0}^\infty z_1^{-p} \left(\frac{1}{(r-1)!}\frac{\partial^{r-1}}{\partial z_1^{r-1}}z_1^{-1-i-q}\right)z_2^{i} = \frac{1}{(r-1)!}\frac{1}{z_1^{p}}\frac{\partial^{r-1}}{\partial z_1^{r-1}}\left(\frac{1}{z_1^{q}(z_1-z_2)}\right) \label{neg-remov-2}
    \end{align}
    in the region $|z_1|>|z_2|>0$. Again there is no change in the order of the pole $z_1=z_2$. 

    If we further remove negative powers of $z_1$ from (\ref{neg-remov-2}), then (\ref{neg-remov-2}) is then a finite sum, i.e., a Laurant polynomial. In this case, $z_1=z_2$ is no longer a pole of the limit of the sum. Thus we proved the based case $n=2$.

    We focus on the case $n\geq 3$. We first study the case when the rational function is of the form 
    \begin{align}
        \frac{1}{(z_1-z_m)\cdots (z_{m-1}-z_m)z_m^{q}(z_m-z_{m+1})\cdots (z_m-z_n)}\label{neg-remov-3}
    \end{align}
    for some integer $n\geq 3, 1\leq m \leq n$. 
    \begin{itemize}[leftmargin=*]
        \item If $m=1$, then (\ref{neg-remov-3}) is of the form
        $$\frac{1}{z_1^{q}(z_1-z_{2})\cdots (z_1-z_n)}$$
        In the region $|z_1|>\cdots > |z_n|>0$, the rational function is expanded as
        $$\sum_{i_2, ..., i_n=0}^\infty z_1^{-q-i_2-1\cdots - i_n-1}z_2^{i_2}\cdots  z_{n}^{i_n}=\sum_{i_2, ..., i_n=0}^\infty z_1^{-n+1-q-i_2-\cdots - i_n}z_2^{i_2} \cdots z_{n}^{i_n}. $$
        Clearly, there is no negative powers of $z_2, ..., z_n$. So removing negative powers of $z_2, ..., z_n$ does not change the series. Removing negative powers of $z_1$ results in the series
        $$\sum_{\substack{k\geq 0 \\ k+ n+q-1\leq 0}} z_1^{-n+1-q-k}\sum_{i_2 + \cdots + i_n = k} z_2^{i_2}\cdots z_n^{i_n} $$
        that is a finite sum. In this case, $z_1=z_2, ..., z_1=z_n$ are no longer poles of the rational function. 
        \item If $m=2$, then (\ref{neg-remov-3}) is of the form 
        $$\frac{1}{(z_1-z_2)z_2^{q}(z_2-z_{3})\cdots (z_2-z_n)}.$$
        In the region $|z_1|>\cdots > |z_n|>0$, the rational function is expanded as
        $$\sum_{i_1, i_3, ..., i_n=0}^\infty z_1^{-i_1-1} z_2^{i_1-q-i_3-1-\cdots -i_n-1} z_3^{i_3} \cdots z_n^{i_n}=\sum_{i_1, i_3, ..., i_n=0}^\infty z_1^{-i_1-1} z_2^{i_1-q-i_3-\cdots -i_n-(n-2)} z_3^{i_3} \cdots z_n^{i_n}$$
        Clearly the series does not contain any negative powers of $z_3, ..., z_n$. Removing negative powers of $z_2, ... ,z_n$ does not change the series. Removing the negative powers of $z_2$ results in the series
        \begin{align*}
            \sum_{k\geq 0} \sum_{i_1=q+k+(n-2)}^\infty z_1^{-i_1-1} z_2^{i_1-q-k-(n-2)} \sum_{i_3+\cdots+ i_n=k}  z_3^{i_3} \cdots z_n^{i_n}=  \frac{1}{(z_1-z_2)z_1^{q}(z_1-z_3)\cdots (z_1-z_n)}
        \end{align*}
        in the region $|z_1|>\cdots >|z_n|>0$. Thus there is no change for the sum of the order of the poles $z_i=z_j$. So the claim is proved for $m=2$. If we further remove the negative powers of $z_1$, we may repeat the discussion of the $m=1$ case. 
        \item Consider first the case $3\leq m \leq n$. In the region $|z_1|>\cdots > |z_n|>0$, (\ref{neg-remov-3}) is expanded as
    $$\sum_{i_1, ...,i_{m-1}, i_{m+1}, ..., i_n=0}^\infty z_1^{-i_1-1}\cdots z_{m-1}^{-i_{m-1}-1} z_m^{i_1+\cdots+i_{m-1}    - 1-i_{m+1} -\cdots - 1-i_n-q}z_{m+1}^{i_{m+1}}\cdots z_{n}^{i_n}. $$
    Clearly there is no negative powers of $z_{m+1}, ..., z_n$. Removing all the negative powers of $z_m$ changes the series into the form
    \begin{align}
        & \sum_{\substack{i_1+\cdots + i_{m-1} \geq (n-m+q) + i_{m+1} + \cdots + i_n\\i_1, ...,i_{m-1}, i_{m+1}, ..., i_n\in \N}} z_1^{-i_1-1}\cdots z_{m-1}^{-i_{m-1}-1} z_m^{i_1+\cdots+i_{m-1} - (n-m+q)-i_{m+1} -\cdots -i_n}z_{m+1}^{i_{m+1}}\cdots z_{n}^{i_n}\nonumber\\
        = \ & \sum_{\substack{k\geq (n-m+q)+i_{m+1}+\cdots + i_{n}\\i_{m+1}, ..., i_n \in \N}}\left(\sum_{\substack{i_1+\cdots + i_{m-1}=k\\ i_1, ..., i_{m-1}\in \N}}z_1^{-i_1-1}\cdots z_{m-1}^{-i_{m-1}-1}\right) z_m^{k-(n-m+q)-i_{m+1}-\cdots - i_n} z_{m+1}^{i_{m+1}}\cdots z_n^{i_n} \label{negative-removal-3}
    \end{align}
    We focus on the inner sum. 
    \begin{align}
        & \sum_{\substack{i_1+\cdots + i_{m-1}=k\\ i_1, ..., i_{m-1}\in \N}} z_1^{-i_1-1}\cdots z_{m-2}^{-i_{m-2}-1}z_{m-1}^{-i_{m-1}-1}\nonumber\\
        = \ & \sum_{i_1=0}^k \sum_{i_2=0}^{k-i_1}\cdots \sum_{i_{m-3}=0}^{k-i_1-\cdots -i_{m-4}}\sum_{i_{m-2}=0}^{k-i_1-\cdots -i_{m-3}} z_1^{-i_1-1}\cdots z_{m-2}^{-i_{m-2}-1}z_{m-1}^{-k+i_1+\cdots + i_{m-2}-1 }\nonumber\\
        = \ & \sum_{i_1=0}^k \sum_{i_2=0}^{k-i_1}\cdots \sum_{i_{m-3}=0}^{k-i_1-\cdots -i_{m-4}} z_1^{-i_1-1}\cdots z_{m-3}^{-i_{m-3}-1}z_{m-1}^{-k+i_1+\cdots + i_{m-3}-1 }\sum_{i_{m-2}=0}^{k-i_1-\cdots -i_{m-3}}  z_{m-2}^{-i_{m-2}-1}z_{m-1}^{ i_{m-2} }\label{negative-removal-4}
    \end{align}
    The summation of $i_{m-2}$ yields
    \begin{align*}
        & z_{m-2}^{-1}\frac{1-(z_{m-2}^{-1}z_{m-1})^{k-i_1-\cdots -i_{m-3}+1}}{1-z_{m-2}^{-1}z_{m-1}}\\
        & \frac{1-z_{m-2}^{-k+i_1+\cdots +i_{m-3}-1}z_{m-1}^{k-i_1-\cdots -i_{m-3}+1}}{z_{m-2}-z_{m-1}}\\
    \end{align*}
    Thus (\ref{negative-removal-4}) becomes
    \begin{align*}
        & \sum_{i_1=0}^k \sum_{i_2=0}^{k-i_1}\cdots \sum_{i_{m-3}=0}^{k-i_1-\cdots -i_{m-4}} z_1^{-i_1-1}\cdots z_{m-3}^{-i_{m-3}-1}\frac{z_{m-1}^{-k+i_1+\cdots + i_{m-3}-1 }-z_{m-2}^{-k+i_1+\cdots +i_{m-3}-1}}{z_{m-2}-z_{m-1}}\\
        \ = & \frac{1}{z_{m-2}-z_{m-1}}\sum_{i_1=0}^k \sum_{i_2=0}^{k-i_1}\cdots \sum_{i_{m-3}=0}^{k-i_1-\cdots -i_{m-4}} z_1^{-i_1-1}\cdots z_{m-3}^{-i_{m-3}-1}z_{m-1}^{-k+i_1+\cdots + i_{m-3}-1 }\\
        & - \frac{1}{z_{m-2}-z_{m-1}}\sum_{i_1=0}^k \sum_{i_2=0}^{k-i_1}\cdots \sum_{i_{m-3}=0}^{k-i_1-\cdots -i_{m-4}} z_1^{-i_1-1}\cdots z_{m-3}^{-i_{m-3}-1}z_{m-2}^{-k+i_1+\cdots +i_{m-3}-1}\\
        \ = & \frac{1}{z_{m-2}-z_{m-1}}\sum_{i_1+\cdots + i_{m-3}+ i_{m-1}=k}z_1^{-i_1-1}\cdots z_{m-3}^{-i_{m-3}-1}z_{m-1}^{-i_{m-1}-1 }\\
        & - \frac{1}{z_{m-2}-z_{m-1}}\sum_{i_1+\cdots + i_{m-3}+ i_{m-2}=k} z_1^{-i_1-1}\cdots z_{m-3}^{-i_{m-3}-1}z_{m-2}^{-i_{m-2}-1}
    \end{align*}
    Therefore, (\ref{negative-removal-3}) is the difference of 
    \begin{align}
        \frac{1}{z_{m-2}-z_{m-1}} \sum_{\substack{k\geq (n-m+q)+i_{m+1}+\cdots + i_{n}\\i_{m+1}, ..., i_n \in \N}}& \left(\sum_{\substack{i_1+\cdots + i_{m-3}+i_{m-1}=k\\ i_1, ..., i_{m-3}, i_{m-1}\in \N}}z_1^{-i_1-1}\cdots z_{m-3}^{-i_{m-3}-1}z_{m-1}^{-i_{m-1}-1}\right) \nonumber\\
        & \cdot z_m^{k-(n-m+q)-i_{m+1}-\cdots - i_n} z_{m+1}^{i_{m+1}}\cdots z_n^{i_n} \label{negative-removal-4-1}
    \end{align}
    and 
    \begin{align}
        \frac{1}{z_{m-2}-z_{m-1}} \sum_{\substack{k\geq (n-m+q)+i_{m+1}+\cdots + i_{n}\\i_{m+1}, ..., i_n \in \N}}& \left(\sum_{\substack{i_1+\cdots + i_{m-3}+i_{m-2}=k\\ i_1, ..., i_{m-3}, i_{m-2}\in \N}}z_1^{-i_1-1}\cdots z_{m-3}^{-i_{m-3}-1}z_{m-2}^{-i_{m-2}-1}\right) \nonumber\\
        & \cdot z_m^{k-(n-m+q)-i_{m+1}-\cdots - i_n} z_{m+1}^{i_{m+1}}\cdots z_n^{i_n} \label{negative-removal-5}
    \end{align}
    (\ref{negative-removal-4-1}) is precisely obtained by removing the negative powers of $z_m$ from the expansion of the rational function 
    $$\frac{1}{z_{m-2}-z_{m-1}} \cdot \frac{1}{(z_1-z_m)\cdots (z_{m-3}-z_m)(z_{m-1}-z_m)z_m^q(z_m-z_{m+1})\cdots (z_m-z_n)}. $$
    (\ref{negative-removal-5}) is precisely obtained by removing the negative powers of $z_m$ from the expansion of the rational function 
    $$\frac{1}{z_{m-2}-z_{m-1}} \cdot \frac{1}{(z_1-z_m)\cdots (z_{m-3}-z_m)(z_{m-2}-z_m)z_m^q(z_m-z_{m+1})\cdots (z_m-z_n)}. $$
    The induction hypothesis applies to conclude the proof of this case. 
    \end{itemize}
    So the claim is proved if the rational function is of the form (\ref{neg-remov-3}).

    Now we study a slightly more general case when the rational function is of the form
    \begin{align}
        \frac{1}{(z_1-z_m)^{P_{1}}\cdots (z_{m-1}-z_m)^{P_{m-1}} z_m^q (z_{m}-z_{m+1})^{P_{m+1}}\cdots (z_m-z_n)^{P_n} }.  \label{negative-removal-5-1}
    \end{align}
    The expansion of (\ref{negative-removal-5-1}) can be expressed as 
    \begin{align*}
        & \frac{(-1)^{P_1-1+\cdots + P_{m-1}-1}}{\prod\limits_{1\leq i \leq n, i\neq m}(P_i-1)!}\prod_{1\leq i \leq n, i\neq m}\left(\frac{\partial^{P_i-1}}{\partial z_i^{P_i-1}}\right)\\
        &\cdot \sum_{i_1, ...,i_{m-1}, i_{m+1}, ..., i_n=0}^\infty z_1^{-i_1-1}\cdots z_{m-1}^{-i_{m-1}-1} z_m^{i_1+\cdots+i_{m-1}    - 1-i_{m+1} -\cdots - 1-i_n-q}z_{m+1}^{i_{m+1}}\cdots z_{n}^{i_n}.     
    \end{align*}
    Removing the negative powers of $z_m$ results in the series
    \begin{align*}
        & \frac{(-1)^{P_1-1+\cdots + P_{m-1}-1}}{\prod\limits_{1\leq i \leq n, i\neq m}(P_i-1)!}\prod_{1\leq i \leq n, i\neq m}\left(\frac{\partial^{P_i-1}}{\partial z_i^{P_i-1}}\right)\\
        &\cdot  \sum_{\substack{i_1+\cdots + i_{m-1} \geq (n-m+q) + i_{m+1} + \cdots + i_n\\i_1, ...,i_{m-1}, i_{m+1}, ..., i_n\in \N}} z_1^{-i_1-1}\cdots z_{m-1}^{-i_{m-1}-1} z_m^{i_1+\cdots+i_{m-1}    - 1-i_{m+1} -\cdots - 1-i_n-q}z_{m+1}^{i_{m+1}}\cdots z_{n}^{i_n}.     
    \end{align*}
    From the product rule, each partial derivative acting on a rational function results in a sum of rational functions whose denominator has total degree increased by exactly 1. The conclusion then directly follows. 
    
    Now we study the most general case. Without loss of generality, we may consider the case when the numerator of (\ref{negative-removal-2}) is a monomial in $z_1, ..., z_n$. Then we argue by reverse induction. The base case $m=n$ easily follows from the discussion of (\ref{negative-removal-5-1}). Assume the conclusion holds for all larger $m$. Without loss of generality, we may assume that negative powers of $z_{m+1}, ..., z_n$ are all removed. Thus, (\ref{negative-removal-2}) is of the form 
    \begin{align}
        \prod_{i=1}^{m-1}\frac 1 {z_i^{Q_i}}\prod_{i=m+1}^n z_i^{Q_i}\prod_{1\leq i < j \leq m-1}\frac 1 {(z_i-z_j)^{P_{ij}}}\cdot \frac 1 {z_m^{Q_m}\prod\limits_{i=1}^{m-1}(z_i- z_m)^{P_{im}}\prod\limits_{j=m+1}^n (z_m-z_j)^{P_{mj}}} \label{negative-removal-6}
    \end{align}
    where $Q_1, ..., Q_m\in \Z, Q_{m+1}, ..., Q_n\geq 0$. Clearly, in the expansion of (\ref{negative-removal-6}) in the region $|z_1|>\cdots > |z_n|$, only the expansion of the last fraction involves $z_m$, which is precisely of the form (\ref{negative-removal-5-1}) that has been handled. Thus follows the conclusion. 
\end{proof}

\begin{prop}\label{omega-convergence}
For every $l_1, l_2\in \N$, every $v_1, ..., v_m\in V$, every $w'\in W'$ and $w_1\in W_1$, the series 
$$\langle w', Y_W(\omega, z_1)\cdots Y_W(\omega, z_{l_1})\pi_{W_2} Y_W(\omega, z_{l_1+1})\cdots Y_W(\omega, z_{l_1+l_2})w_1\rangle $$
converges absolutely in the region 
$$|z_1|>\cdots > |z_{l_1}|> |z_{l_1+1}|>\cdots > |z_{l_1+l_2}| > 0$$
to a rational function whose only possible poles are at $z_i = 0$ $(i = 1, ..., l_1+l_2)$, and at $z_i = z_j$ $(1\leq i < j \leq l_1+l_2)$. Moreover, the rational function can be expressed as a linear combination of fractions 
\begin{align}
    \frac{h(z_1, ..., z_n)}{\prod\limits_{i=1}^{l_1+l_2}z_i^{P_i}\prod\limits_{1\leq i < j \leq l_1+l_2}(z_i-z_j)^{P_{ij}}}\label{omega-convergence-sum}
\end{align} with $\sum\limits_{1\leq i < j \leq l_1+l_2}P_{ij} \leq 2(l_1+l_2-1)$
\end{prop}

\begin{proof}

Apply Theorem \ref{general-l-thm} and rewrite the polynomial 
$$f(x_{\alpha_1}, ..., x_{\alpha_p}) = \sum_{i_1, ..., i_p} w_{\alpha_1\cdots \alpha_p}^{i_1\cdots i_p}x_{\alpha_1}^{i_1}\cdots x_{\alpha_p}^{i_p}, $$
we know that 
$$Y_W(\omega, x_1) \cdots Y_W(\omega, x_{l_1}) \pi_{W_2} Y_W(\omega, x_{{l_1}+1})\cdots Y_W(\omega, x_{{l_1}+{l_2}})w_1 $$
is a linear combination of series of the form 
\begin{align}
    & Y_W(\omega, x_1) \cdots Y_W(\omega, x_{l_1})\nonumber \\ 
    & \qquad \cdot  Y_W^+(v_{\alpha_1^c}, x_{\alpha_1^c})\cdots Y_W^+(v_{\alpha_{l_2-p}^c}, x_{\alpha_{l_2-p}^c})w_{\alpha_1\cdots\alpha_p}^{i_1\cdots i_p}\iota_{1...n}\left(\frac{x_{\alpha_1}^{i_1}\cdots x_{\alpha_p}^{i_p}}{\prod\limits_{l_1+1\leq i < j \leq l_1+l_2}(x_{i}- x_{j})^{Q_{ij}}}\right) \label{omega-convergence-0}
\end{align}
where $\alpha_1, ..., \alpha_p$ is an increasing sequence in in $\{l_1+1, ..., l_1+l_2\}$ with complement $\alpha_1^c, ..., \alpha_{l_2-p}^c$; $i_1, ..., i_p$ are integers; $w_{\alpha_1\cdots\alpha_p}^{i_1\cdots i_p}$ is an element in $W_2$, and $\sum\limits_{l_1+1\leq i < j \leq l_1+l_2}Q_{ij} \leq 2(p-1)$. We now proceed to show that for 
\begin{align}
    \langle w', Y_W(\omega, z_1) \cdots Y_W(\omega, z_{l_1})Y_W^+(\omega, z_{l_1+1})\cdots Y_W^+(\omega, z_{l_1+l_2-p})w_2\rangle. \label{omega-convergence-reg}
\end{align}
converges absolutely in the region $|z_1|> \cdots > |z_{l_1+q}|>0$ to a rational function with poles at $z_i = 0$ $(i = 1,..., l_1+l_2-p)$ and $z_i = z_j$ $(1\leq i < j \leq l_1+l_2-p)$. Note that
(\ref{omega-convergence-reg}) is the series obtained by removing all the negative powers of $z_{l_1+j}$ $(j=1, ..., l_2-p)$ from the following series 
\begin{align}
    \langle w', Y_W(\omega, z_1) \cdots Y_W(\omega, z_{l_1})Y_W(\omega, z_{l_1+1})\cdots Y_W(\omega, z_{l_1+l_2-p})w_2\rangle. \label{omega-convergence-full}
\end{align}
which is absolutely convergent. Thus (\ref{omega-convergence-reg}) itself is also absolutely convergent. 
Moreover, note that (\ref{omega-convergence-full}) is the expansion of a rational function where for every $1\leq i<j\leq l_1+l_2-p$, negative powers of $z_i-z_j$ are expanded as power series of $z_j$. More precisely, (\ref{omega-convergence-full}) is the expansion of  
\begin{align}
    \frac{g(z_1, ..., z_{l_1+l_2-p})}{\prod\limits_{i=1}^{l_1+l_2-p} z_i^{p_i}\prod\limits_{1\leq i < j \leq l_1+l_2-p}(z_i-z_j)^{R_{ij}}}\label{omega-convergence-full-sum}
\end{align} 
where for each $1\leq i<j\leq l_1+l_2-p$, 
\begin{align}
    \iota_{1...n}\frac 1{(z_i-z_j)^{R_{ij}}} = \sum_{k=0}^{\infty} \binom{-R_{ij}}{k} z_i^{-P_{ij}-k}(-1)^k z_j^k \label{omega-convergence-full-factor}
\end{align}
It also follows from Theorem \ref{general-l-lemma} that (\ref{omega-convergence-full-sum}) is a linear combination of fractions of the same form as (\ref{omega-convergence-full-sum}) with 
$\sum\limits_{1\leq i < j \leq l_1+l_2-p}R_{ij} \leq 2(l_1+l_2-p)$. We use Lemma \ref{negative-removal} to conclude that the series (\ref{omega-convergence-reg}) converges absolutely to a linear combinations of rational functions of the form (\ref{omega-convergence-full-sum}) with $\sum\limits_{1\leq i < j \leq l_1+l_2-p} R_{ij} \leq 2(l_1+l_2-p)$. Therefore, each series of the form (\ref{omega-convergence-0}), evaluated at complex numbers $x_i = z_i$ and paired with elements in $w'$, converges absolutely to a linear combination of rational functions of the form (\ref{omega-convergence-sum}) with 
$$\sum_{1\leq i< j \leq l_1+l_2} P_{ij} = \sum_{1\leq i < j \leq l_1+l_2} (Q_{ij} + R_{ij})  \leq 2(l_1+l_2-p) + 2(p-1) = 2(l_1+l_2-1)$$
(Here we regard $Q_{ij}$ and $R_{ij}$ as zero if the indices are out of the original definition). This conclues the proof of the proposition. 
\end{proof}

\begin{thm}\label{convergence-1-gen-thm}
For every $l_1, l_2\in \N$, every $v_1, ..., v_{l_1+l_2}\in V$, every $w'\in W'$ and $w_1\in W$, the series 
\begin{align}
    \langle w', Y_W(v_1, z_1)\cdots Y_W(v_{l_1}, z_{l_1})\pi_{W_2} Y_W(v_{l_1+1}, z_{l_1+1})\cdots Y_W(v_{l_1+l_2}, z_{l_1+l_2})w_1\rangle 
\end{align}
converges absolutely in the region 
\begin{align}
    |z_1|>\cdots > |z_{l_1}|> |z_{l_1+1}|>\cdots > |z_{l_1+l_2}| > 0 \label{convergence-1-gen-target-region}
\end{align}
to a rational function that is a linear combination of rational functions of the form 
$$\frac{h(z_1, ..., z_{l_1+l_2})}{\prod\limits_{i=1}^{l_1+l_2} z_i^{P_i} \prod\limits_{1\leq i < j \leq l_1+l_2}(z_i - z_j)^{P_{ij}} }$$
with $\sum\limits_{1\leq i < j \leq l_1+l_2}P_{ij} \leq \wt(v_1) + \cdots + \wt(v_{l_1+l_2})-2$. 
\end{thm}

\begin{proof}
It suffices to focus on the case where 
$$v_{i} = L(-n_1^{(i)})\cdots L(-n_{s_i}^{(i)})\one.$$
We will argue by induction of $i$ (from $l_1+l_2$ to $l_1+1$) and $s_i$. For brevity, we only show the inductive step with $i\geq l_1+1$ here: suppose the conclusion holds for the series of the form
\begin{align}
    \langle w',  Y_W(\omega, z_1)\cdots Y_W(\omega, z_{l_1})\pi_{W_2} & Y_W(\omega, z_{l_1+1})\cdots Y_W(\omega, z_{l_1+i-1})\nonumber \\
    & \cdot Y_W(v_{l_1+i}, z_{l_1+i})Y_W(v_{l_1+l_2}, z_{l_1+l_2})w_1\rangle, 
\end{align}
then it also holds for the series of the form
\begin{align}
    \langle w',  Y_W(\omega, z_1)\cdots Y_W(\omega, z_{l_1})\pi_{W_2} & Y_W(\omega, z_{l_1+1})\cdots Y_W(\omega, z_{l_1+i-1})\nonumber \\
    & \cdot Y_W(L(-n)v_{l_1+i}, z_{l_1+i})Y_W(v_{l_1+l_2}, z_{l_1+l_2})w_1\rangle,\label{convergence-1-gen-target }
\end{align}
for every $n\geq 2$. The idea is to use the iterate formula 
\begin{align*}
    & Y_W(L(-n)v_{l_1+i}, z_{l_1+i})\\
    & = \Res_{z}\left( \iota_{z,z_{l_1+i}}(z-z_{l_1+i})^{-n-1} Y_W(\omega,z)Y(v_{l_1+i}, z_{l_1+i}) - \iota_{z_{l_1+i},z}(-z_{l_1+i}+z)^{-n-1} Y(v_{l_1+i}, z_{l_1+i})Y_W(\omega,z)\right)
\end{align*}

From the induction hypothesis, we know that the series 
\begin{align}
    \iota_{z,z_{l_1+i}}(z-z_{l_1+i})^{-n-1}\langle w',  Y_W(\omega, z_1)\cdots Y_W(\omega, z_{l_1})\pi_{W_2} & Y_W(\omega, z_{l_1+1})\cdots Y_W(\omega, z_{l_1+i-1})\nonumber \\
    & \cdot Y_W(\omega, z)Y_W(v_{l_1+i}, z_{l_1+i})\cdots Y_W(v_{l_1+l_2}, z_{l_1+l_2})w_1\rangle
    \label{convergence-1-gen-line-1}
\end{align}
converges absolutely in the region 
$$|z_1|>\cdots > |z_{l_1+i-1}| > |z| > |z_{l_1+i}| > \cdots > |z_{l_1+l_2}|> 0$$
to a linear combination of rational functions of the form
$$f_{i,1}(z_1, ..., z_{l_1+l_2}, z)=\frac{h_{i,1}(z_1, ..., z_{l_1+l_2}, z)}{z^p \prod\limits_{s=1}^{l_1+l_2}z_s^{p_s}(z-z_s)^{P_s}\prod\limits_{1\leq s < t \leq l_1+l_2}(z_s-z_t)^{P_{st}} (z-z_{l_1+i})^{n+1}}, $$ 
where $h_{i,1}$ is a polynomial function, and $$\sum\limits_{s=1}^{l_1+l_2}P_s + \sum\limits_{1\leq s < t \leq l_1+l_2} P_{st} \leq 2(l_1+i) + \wt(v_{l_1+i}) + \cdots + \wt(v_{l_1+l_2}) - 2. $$ 
Applying $\Res_z$ to (\ref{convergence-1-gen-line-1}) amounts to integrating $f_{i,1}(z_1, ..., z_{l_1+l_2}, z)$ along the curve $z= re^{i\theta}, \theta\in [0, 2\pi]$ with $ |z_{l_1+i-1}|>r>|z_{l_1+i}|$. From Cauchy integral theorem, $\Res_z$ of (\ref{convergence-1-gen-line-1}) converges absolutely in the region 
$$|z_1|>\cdots > |z_{l_1+i-1}| > |z_{l_1+i}| > \cdots > |z_{l_1+l_2}|> 0$$
(which is precisely the region (\ref{convergence-1-gen-target-region})) to the function
\begin{align}
    g_{i,1}(z_1, ..., z_{l_1+l_2}) &=\oint_{z=re^{i\theta}}f_{i,1}(z_1, ..., z_{l_1+l_2}, z)dz\\ &= \sum_{j=i}^{l_2}\oint_{C(z_{l_1+j})}f_{i,1}(z_1, ..., z_{l_1+l_2}, z)dz \nonumber\\
    &= \sum_{j=i}^{l_2} \Res_{z=z_{l_1+j}}f_{i,1}(z_1, ..., z_{l_1+l_2}, z)dz \label{convergence-1-gen-line-2}
\end{align}
where for each index $j$ in the sum, $\Res_{z=z_{l_1+j}}f_{i,1}(z_1, ..., z_{l_1+l_2}, z)$ is the coefficient of $(z-z_{l_1+j})^{-1}$ in the expansion of the rational function $f_{i,1}(z_1, ..., z_{l_1+l_2}, z)$ where 
\begin{itemize}
    \item negative powers of $z = z_{l_1+j} + (z-z_{l_1+j})$ are expanded as a power series in $(z-z_{l_1+j})$. 
    \item negative powers of $z-z_s = (-z_s + z_{l_1+j}) + (z-z_{l_1+j})$ are expanded as a power series in $(z-z_{l_1+j})$, for every $i\neq l_1+j$.  
\end{itemize}
So each summand in (\ref{convergence-1-gen-line-2}) is a linear combination of rational functions of the form 
$$\frac{h_{i,1,j}(z_1, ..., z_{l_1+l_2}, z)}{z_{l_1+j}^{p_1} \prod\limits_{s=1}^{l_1+l_2}z_s^{p_s}(z_{l_1+j}-z_s)^{P_{s, 1}}\prod\limits_{1\leq s < t \leq l_1+l_2}(z_s-z_t)^{P_{st}} (z_{l_1+j}-z_{l_1+i})^{n_1}}, $$ 
with 
\begin{align*}
    & p_1+\sum_{s=1}^{l_1+l_2}P_{s,1}+n_1+1=p+\sum_{s=1}^{l_1+l_2} P_s+\sum_{s=1}^{l_1+l_2}+n+1\\
    \ \Rightarrow &  \sum_{s=1}^{l_1+l_2}P_{s,1}+n_1=p-p_1+\sum_{s=1}^{l_1+l_2}P_s+n\leq \sum_{s=1}^{l_1+l_2}P_s+n
\end{align*}
Thus 
\begin{align*}
    \sum_{s=1}^{l_1+l_2}P_{s,1} + \sum_{1\leq s < t \leq l_1+l_2}P_{st} + n_1 & \leq \sum_{1\leq s < t \leq l_1+l_2}P_{st} + \sum_{s=1}^{l_1+l_2}P_s + n \\
    & = 2(l_1+i) + \wt(v_{l_1+i}) + \cdots + \wt(v_{l_1+l_2}) - 2 + n\\
    & = 2(l_1+i) + \wt(L(-n)v_{l_1+i}) + \cdots + \wt(v_{l_1+l_2}) - 2 
\end{align*}
Using an almost identical argument, we see that $\Res_z$ of the series
\begin{align}
    \iota_{z_{l_1+i},z}(-z_{l_1+i}+z)^{-n-1}\langle w',  Y_W(\omega, z_1)\cdots Y_W(\omega, z_{l_1})\pi_{W_2} & Y_W(\omega, z_{l_1+1})\cdots Y_W(\omega, z_{l_1+i-1})\nonumber \\
    & \cdot Y_W(v_{l_1+i}, z_{l_1+i})Y_W(\omega, z)\cdots Y_W(v_{l_1+l_2}, z_{l_1+l_2})w\rangle \label{convergence-1-gen-line-3}
\end{align}
converges in the same region to a linear combination of rational function $g_2(z_1, ..., z_{l_1+l_2})$ with the same structures of poles and the same estimate of sum orders as in $g_1(z_1, ..., z_{l_1+l_2})$. Using the iterate formula, we see that the same holds for the series (\ref{convergence-1-gen-target }). The inductive step is proved. Using the identical procedure, we may pass all the $\omega$'s before $\pi_{W_2}$ to generic elements in $V$. Details shall not be repeated here. 
\end{proof}

\begin{cor}
For every $v\in V$, the map $\pi_{W_2}Y_W(v, x)\pi_{W_1}: W_1 \to W_2((x))$ 
is in $\H_2(W_1, W_2)$. 
\end{cor}

\begin{proof}
Clearly from Theorem \ref{convergence-1-gen-thm}, for every $k,l\in \N$, every $u_1, ..., u_{k+l}\in V$, every $w_2'\in W_2', w_1\in W_1$,
    \begin{align}
        \langle w_2', Y_{W_2}(u_1, z_1)\cdots Y_{W_2}(u_k, z_k)\pi_{W_2}Y_W(v,z)  Y_W(u_{k+1}, z_{k+1})\cdots Y_W(u_{k+l}, z_{k+l})w_1\rangle \label{composability-1-gen-line-1}
    \end{align}
    converges absolutely in the region 
    $$|z_1|> \cdots > |z_k| > |z| > |z_{k+1}| > \cdots > |z_{k+l}| > 0$$
    to a linear combination of rational functions that is of the form 
    \begin{align}
        \frac{h(z_1, ..., z_k, z, z_{k+1}, ...,  z_{k+l})}{z^p\prod\limits_{i=1}^{k+l} z_i^{p_i}(z_i-z)^{P_i} \prod\limits_{1\leq i < j \leq k+l}(z_i - z_j)^{P_{ij}} }\label{composability-1-gen-sum}   
    \end{align}
    with $\sum\limits_{i=1}^{k+l}P_i + \sum\limits_{1\leq i < j \leq k+l}P_{ij} \leq \sum\limits_{i=1}^{k+l}\wt(u_i) + \wt(v) - 2$. To squeeze $\pi_{W_1}$ in between, we start by rewriting (\ref{composability-1-gen-line-1}) as 
    \begin{align}
        & \langle w_2', Y_{W_2}(u_1, z_1)\cdots Y_{W_2}(u_k, z_k)\pi_{W_2}Y_W(v,z)  \pi_{W_1}Y_W(u_{k+1}, z_{k+1})\cdots Y_W(u_{k+l}, z_{k+l})w_1\rangle \label{composability-1-gen-line-2}\\
        & + \langle w_2', Y_{W_2}(u_1, z_1)\cdots Y_{W_2}(u_k, z_k)\pi_{W_2}Y_W(v,z)  \pi_{W_2}Y_W(u_{k+1}, z_{k+1})\cdots Y_W(u_{k+l}, z_{k+l})w_1\rangle \label{composability-1-gen-line-3}
    \end{align}
    Note that since $W_2$ is a submodule, $\pi_{W_2}Y_W(v, z) \pi_{W_2} = Y_{W_2}(v,z)\pi_{W_2}$. Thus the convergence of (\ref{composability-1-gen-line-3}) follows from Theorem \ref{convergence-1-gen-thm}. Therefore  (\ref{composability-1-gen-line-2}) converges. We can repeat the same process consecutively to squeezing $\pi_{W_1}$ between the vertex operators and show convergence. Details should not be repeated here. 
    
    To verify the $N$-weight-degree condition with $N=2$, we start with the observation that expanding each rational function (\ref{composability-1-gen-sum}) 
    as a Laurent series in $z_1-z_{k+l}, ...,z-z_{k+l}, ..., z_{k+l-1}-z_{k+l}$ with coeffients in $\C[z_{k+l}, z_{k+l}^{-1}]$ amounts to 
    \begin{enumerate}
        \item expanding negative powers of $z= z_{k+l}-(z-z_{k+l})$ as a power series in $(z-z_{k+l})$;
        \item expanding negative powers of $z_i = z_{k+l}-(z_i-z_{k+l})$ as a power series in $(z_i-z_{k+l})$, $i=1, ..., k+l-1$;
        \item expanding negative powers of $z-z_i = -(z_i-z_{k+l})+(z-z_{k+l})$ as a power series in $(z-z_{k+l})$, for $i=1, ..., k$;
        \item expanding negative powers of $z-z_i = z-z_{k+l}-(z_i-z_{k+l})$ as a power series in $(z_i-z_{k+l})$, for $i=k+1, ..., k+l$;
        \item expanding negative powers of $z_i-z_j = z_i-z_{k+l}-(z_j-z_{k+l})$ as a power series in $(z_j-z_{k+l})$, $1\leq i < j \leq k+l$.
    \end{enumerate}
    Clearly, (1) and (2) do not contribute to the lowest total degree of $z_1-z_{k+l}, ..., z_i-z_{k+l}$, while (3), (4) and (5) restricts the sum of the degrees of the two variables to be the same as $-P_i$, $-P_i$ and $-P_{ij}$, respectively. Therefore, the lowest total degree of $z_1-z_{k+l}, ..., z_i-z_{k+l}$ is precisely 
    $$-\sum_{i=1}^{k+l}P_i-\sum_{1\leq i<j \leq k+l}P_{k+l} \geq 2 - \wt(u_1)\cdots - \wt(u_k) - \wt(v) - \wt(u_{k+1}) - \cdots - \wt(u_{k+l}). $$  
\end{proof}

\begin{rema}
For the Verma module $M(c,h)$ with $[c,h]$ in the block (d), we currently do not know how to show the convergence when the submodule is generated by two singular vectors. Some preliminary computation shows that the situation is much more complicated than what we observe here. 
\end{rema}

\subsection{General case}

From Proposition \ref{submodule-quotient}, it suffices that we consider the module $\bigoplus_{i=1}^n M(c, h_i)$. We start by classifying the submodules. 

\begin{prop}\label{submodule-classify}
Let $M_i = M(c,h_i)$ be Verma modules of blocks (a), (b) and (c), $i = 1,..., n$. Let $W$ be a submodule of $M=\bigoplus_{i=1}^n M_i$. Then $W$ is generated by singular vectors. Moreover, there exist submodules $N_i \subseteq M_i$ for each $i=1, ..., n$, such that $W$ is isomorphic to $\bigoplus_{i=1}^n N_i$. 
\end{prop}

\begin{proof}
The idea of the proof is suggested by Kenji Iohara. Reviewer 1 contributed the following argument that is a significant simplification. 

The conclusion clearly holds if $n=1$. Let $\pi_i: M \to M_i$ be the canonical projection, $i =1, ..., n$. Without loss of generality, we rearrange the indices, such that 
$$\pi_1 W = (\pi_1 + \pi_2)W \cdots = \left(\sum_{i=1}^{i_0}\pi_{i}\right) W = 0, \pi_{j}W\neq 0, j>i_0.$$
Then, we compare the lowest weights of $\pi_jW$ for $j=i_0+1, ..., n$. We rearrange the indices so that
$$\pi_{i_0+1}W \text{ is of minimal lowest weight}. $$
Clearly, $P = \left(\bigoplus\limits_{i>i_0+1} M_i\right) \cap W$ is the kernel of $\pi_{i_0+1}|_W: W\to \pi_{i_0+1} W$, and the sequence
\begin{align}
    0 \to P \to W \to \pi_{i_0+1} W \to 0 \label{submodule-classify-1}
\end{align}
is exact. As a submodule of $M_{i_0+1}$, $\pi_{i_0+1}W$ is generated by a singular vector $u$. Let $w\in W$ such that $\pi_{i_0+1}w = u$, i.e.,
    $$w = \left(0, ..., 0, u, w^{(i_0+2)}, ..., w^{(n)}\right),$$
    where $w^{(j)}$ are elements in $M_j$ that might be zero or nonzero, $j=i_0+2, ..., n$. By our choice of $i_0$, we know that $w$ is of the lowest conformal weight in $W$. Necessarily, $w$ is a singular vector in $W$. The submodule $\langle w\rangle$ generated by $w$ is isomorphic to $\pi_{i_0+1}W$. Therefore, the map $u\mapsto w$ defines a splitting to the exact sequence (\ref{submodule-classify-1}). If we set $N_{i_0+1} = \pi_{i_0+1}W$, then we have $W\simeq P \oplus \pi_{i_0+1}W$. Since $P$ is a submodule of $\left(\bigoplus\limits_{i>i_0+1} M_i\right)$, we may use the induction hypothesis to conclude the proof. 
\end{proof}

\begin{rema}\label{Index-rearrange}
The proof actually gives a ``row echelon form'' of singular vectors. After locating the index $i_0$ and the singular vector $w$ (which we shall denote by $w_{i_0+1}$) from the proof, we may repeat the process on the submodule $P$ in $\bigoplus\limits_{i>i_0+1}M_i$. At the end of the day, we will obtain singular vectors 
\begin{align*}
    w_{i_0+1} &= {\bigg (}0, ..., 0, & w_{i_0+1}^{(i_0+1)} &, ..., w_{i_0+1}^{(i_1)}, & w_{i_0+1}^{(i_1+1)}, ..., w_{i_0+1}^{(i_2)}& , &...&, & w_{i_0+1}^{(i_{r-1}+1)}, ..., w_{i_0+1}^{(n)}{\bigg )}.\\
    w_{i_1+1} &= {\bigg(}0, ..., 0, & 0 &, ..., 0,&  w_{i_1+1}^{(i_1+1)}, ..., w_{i_1+1}^{(i_2)}&, & ...& , & w_{i_1+1}^{(i_{r-1}+1)}, ..., w_{i_1+1}^{(n)}{\bigg)}.\\
    \vdots & \\
    w_{i_{r-1}+1}&={\bigg(}0, ..., 0, & 0 &, ..., 0, &  0, ..., 0&, & ...& , & w_{i_{r-1}+1}^{(i_{r-1}+1)}, ..., w_{i_{r-1}+1}^{(n)}{\bigg)}.
\end{align*}
with $w_{i_j+1}^{(i_j+1)}\neq 0$ for $i=0, ..., r-1$. This ``row echelon form'' of singular vectors will be useful in the proof of convergence. 
\end{rema}

\begin{thm}\label{Virasoro-final-thm}
Let $M = \bigoplus_{i=1}^n M(c, h_i)$ where each $M(c, h_i)$ is in block (a), (b) or (c). Then $M$ satisfies the condition in Definition \ref{composibility}. 
\end{thm}

\begin{proof}
The proof is significantly simplified by the observation of Reviewer 1. Let $M_i= M(c, h_i)$. We arrange the indices $i=1, ...., n$ as in Proposition \ref{submodule-classify} and Remark \ref{Index-rearrange}, so that the isomorphic image of $N_i$ in $W$ is generated by the singular vector $w_i$ of the form $(0, ..., 0, w_i^{(i)}, ..., w_i^{(n)})$, for $i=\alpha+1, ..., n$. With the same notations as in Remark \ref{Index-rearrange}, $w_i \neq 0$ only when $i = i_0+1, ..., i_{r-1}+ 1$. For each such $i$, since $N_i$ is a submodule of $M_i$ while $M_i$ is a Verma module in block (a), (b) or (c), $N_i$ itself is a Verma module. As a submodule of $M$, $N_i$ is generated by the singular vector $(0, ..., 0, w_i^{(i)}, 0, ..., 0)$. $w_i^{(j)}$ is a (possibly zero) singular vector in $M_j$ for every $j\geq i$. 

From what we proved Section \ref{Section-5-2}, each $N_i$ admits a complementary space $P_i$ in $M_i$ (for $i = 1, ..., \alpha$, $P_i$ can simply be chosen as $M_i$), such that for every $v\in V$, $\pi_{N_i} Y_{M_i}(v, x)|_{P_i} \in \H_2(P_i, N_i)$. Moreover, if $N_i \neq 0$, then $N_i$ is spanned by the vectors
\begin{align}
    (0, ..., 0, L(-n_1)\cdots L(-n_s) w_i^{(i)}, 0, ..., 0) \label{basis-N_i}
\end{align}
where $n_1\geq \cdots \geq n_s \geq 1$; $P_i$ is chosen as the subspace spanned by the vectors 
\begin{align}
    (0, ..., 0, L(-n_1)\cdots L(-n_s) L(-1)^{m}, 0, ..., 0) \label{basis-P_i}
\end{align}
where $n_1\geq \cdots \geq n_s \geq 2, m \leq \Ind(w_i^{(i)})$. Now we set  
$$P =\bigoplus_{i=1}^n P_i.$$ 
A dimension counting argument shows that $P$ is also a complementary space of $W$. Let $\rho_0: \bigoplus_{i=1}^n N_i \to W$ be the $V$-module isomorphism, such that
$$\rho_0(0, ..., 0, w_i^{(i)}, 0, ..., 0) = (0, ..., 0, w_i^{(i)}, w_{i}^{(i+1)}, ..., w_{i}^{(n)}). $$
Then clearly, for every $n_1\geq \cdots \geq n_s \geq 1$,  
\begin{align*}
    & \rho_0(0, ..., 0, L(-n_1)\cdots L(-n_s)w_i^{(i)}, 0, ..., 0) \\
    = & (0, ..., 0, L(-n_1)\cdots L(-n_s)w_i^{(i)}, L(-n_1)\cdots L(-n_s)w_{i}^{(i+1)}, ..., L(-n_1)\cdots L(-n_s)w_{i}^{(n)}). 
\end{align*}
We extend $\rho_0$ to a linear isomorphism $\rho: M\to M$ by setting $\rho|_P = id_P, \rho|_{N_i} = \rho_0|_{N_i}$ (note that the extended $\rho$ will not be a $V$-module homomorphism). Clearly, with respect to the basis consisting of vector of form (\ref{basis-N_i}) and (\ref{basis-P_i}), the linear map $\rho$ is triangular with all diagonal entries being 1. Thus, so is the matrix of $\rho^{-1}$, and $\rho^{-1}|_P = id_P$. We also note that both $\rho$ and $\rho^{-1}$ commutes with $L(-m)$ for every $m\geq 2$, since $L(-m)P \subseteq P$. 


Now we set 
$$\pi_N = \sum_{i=1}^n \pi_{N_i}$$
where $\pi_{N_i}$ is extended from the map $M_i\to N_i$ constructed in Section \ref{Section-5-2} to $\bigoplus_{i=1}^n M_i$ such that $\pi_{N_i}|_{M_j} = 0$ for every $1\leq i\neq j\leq n$. We note that since every $\pi_{N_i}$ commutes with $L(-m)$ for $m\geq 2$, so does $\pi_N$. Let $$\pi_W = \rho \circ  \pi_{N} \circ \rho^{-1}$$
Clearly, $\pi_W: M \to W$ is a projection and commutes with $L(-m)$ for every $m\geq 2$. It follows from an argument similar to those in Section 5.2 that for every $v\in V$, 
$$\pi_W Y_{M}(v, x) \pi_P  \in \H_2(P, W).$$
\end{proof}

\begin{rema}

Indeed, the situation for Verma modules of block (d) is much more complicated yet interesting. The preliminary attempts hints that we might need a completely different approach to show the convergence. But what we discovered in this section hints that the convergence requirements might not be a serious obstruction for the application of cohomology theory. We conjecture that category of finite length modules for the Virasoro VOA should all be in $\CC_N$. There should not be any restrictions on the central charges and the lowest weight. The conjecture should be attempted in future works. 
\end{rema}


\noindent {\small \sc Pacific Institute of Mathematics | University of Manitoba, 451 Machray Hall, 186 Dysart Road, Winnipeg, Manitoba, Canada R3T 2N2}

\noindent {\small \sc Department of Mathematics, University of Denver, 2390 S York St Rm 216, Denver 80210, USA}

\noindent{\em Email address}: fei.qi@du.edu | fei.qi.math.phys@gmail.com

The author states that there is no conflict of interest.

\end{document}